\documentclass[a4,14pt]{article}

\usepackage{url}
\usepackage[backend=bibtex,citestyle=numeric]{biblatex}
\usepackage[colorlinks=true, linkcolor=blue]{hyperref}
\usepackage{xr-hyper} 
\usepackage{fullpage}
\bibliography{biblio}

\usepackage[utf8]{inputenc} 
\usepackage[T1]{fontenc}    
\usepackage{url}            
\usepackage{booktabs}       
\usepackage{amsfonts}       
\usepackage{nicefrac}       
\usepackage{microtype}      
\usepackage{amsmath,amsthm,amsfonts,graphicx}
\usepackage[psamsfonts]{amssymb}
\usepackage{caption}
\usepackage[OT1]{fontenc}
\usepackage{bbm}
\usepackage{mathtools}
\usepackage{bbm}
\usepackage{colortbl}
\usepackage{listings}
\usepackage{float}
\newcommand{\Ncal}{\mathcal{N}}
\newcommand{\Lcal}{\mathcal{L}}
\newcommand{\cb}{\mathbf{c}}
\newcommand{\R}{\mathbb{R}}
\newcommand{\E}{\mathbb{E}}
\newcommand{\X}{\mathbbm X}
\newcommand{\V}{\rm V\it}
\newcommand{\1}{\mathbbm 1}
\newcommand{\Ccal}{\mathcal{C}}
\newcommand{\N}{\mathbb{N}}

\newtheorem{theorem}{Theorem}
\newtheorem{proposition}[theorem]{Proposition}
\newtheorem{lemma}[theorem]{Lemma}
\newtheorem{corollary}[theorem]{Corollary}
\newtheorem{definition}[theorem]{Definition}
\newtheorem{example}[theorem]{Example}
\newtheorem{remark}[theorem]{Remark}
\newtheorem*{theorem*}{Theorem}
\newtheorem*{proposition*}{Proposition}
\newtheorem*{lemma*}{Lemma}
\newtheorem*{corollary*}{Corollary}
\newtheorem*{definition*}{Definition}
\newtheorem*{example*}{Example}

\newtheorem{algorithm}{Algorithm}
\definecolor{Gray}{gray}{0.9}

\title{Robust Bregman Clustering}
\author{
  Br\'echeteau, Claire\\
  \texttt{claire.brecheteau@univ-rennes2.fr}\\
  Universit\'e Rennes 2 -- IRMAR
  \and
  Fischer, Aur\'elie\\
  \texttt{aurelie.fischer@lpsm.paris}\\
  Universit\'e de Paris -- LPSM
  \and
  Levrard, Cl\'ement\\
  \texttt{clement.levrard@lpsm.paris}\\
  Universit\'e de Paris -- LPSM
}
\date{}
\begin{document}
\maketitle

\begin{abstract}
Clustering with Bregman divergences encompasses a wide family of clustering procedures that are well-suited to mixtures of distributions from exponential families \cite{Ban05}. However these techniques are highly sensitive to noise. To address the issue of clustering data with possibly adversarial noise, we introduce a robustified version of Bregman clustering based on a trimming approach. We investigate its theoretical properties, showing for instance that our estimator converges at a sub-Gaussian rate $1/\sqrt{n}$, where $n$ denotes the sample size, under mild tail assumptions. We also show that it is robust to a certain amount of noise, stated in terms of Breakdown Point. We also derive a Lloyd-type algorithm with a trimming parameter, along with a heuristic to select this parameter and the number of clusters from sample. Some numerical experiments assess the performance of our method on simulated and real datasets.  
\end{abstract}
%
\section{Introduction}

Clustering is the problem of classifying data in groups of similar points, so that the groups are as homogeneous and at the same time as well separated as possible \cite{DHS}. There are no labels known in advance, so clustering is an unsupervised learning task. To perform clustering,  a distance-like function $d(\cdot,\cdot)$ is often needed to assess a notion of  proximity between points and the separation between clusters. 

Suppose that we know such a natural distance $d$, and assume that the points we want to cluster, say $X_1,\dots,X_n$, are i.i.d., drawn from an unknown distribution $P$, and take values in $\R^d$. For $k \geq 1$, designing $k$ meaningful classes with respect to $d$ can be achieved via minimizing the so-called empirical distortion
\[
R_n(\cb)=\frac{1}{n}\sum_{i=1}^{n}\min_{j\in[\![1,k]\!]}d(X_i,c_j),
\]
over all possible cluster centers or codebooks $\cb=(c_1,\dots,c_k)$, with notation $[\![1,k]\!]$ for $\{1,2,\dots,k\}$. This results in a set of $k$ codepoints. Clusters are then given by the sets of sample points that have the same closest codepoint.  

A classical choice of $d$ is the squared Euclidean distance, leading to the standard $k$-means clustering algorithm (see, e.g., \cite{Llo}). However, some of the desirable properties of the Euclidean distance can be extended to a broader class of dissimilarity functions, namely Bregman divergences. These distance-like functions, denoted by $d_\phi$ in the sequel, are indexed by strictly convex functions $\phi$. Introduced by \cite{Br67}, they are useful in a wide range of areas, among which statistical learning and data mining (\cite{Ban05}, \cite{CBL}), computational geometry \cite{Nie07}, natural sciences, speech processing and information theory \cite{GBG}. Squared Euclidean, Mahalanobis, Kullback-Leibler and $L^2$ distances are all particular cases of Bregman divergences.

A Bregman divergence is not necessarily a true metric, since it may be asymmetric or fail to satisfy the triangle inequality. However, Bregman divergences fulfill an interesting projection property which generalizes the Hilbert projection on a closed convex set \cite{Br67}. They also satisfy non-negativity and separation, convexity in the first argument and linearity (see \cite{Ban05}, \cite{Nie07}). Moreover, Bregman divergences are closely related to exponential families \cite{Ban05}. In fact, they are a natural tool to measure proximity between observations arising from a mixture of such distributions. Consequently, clustering with Bregman divergences is particularly well-suited in this case. 

Clustering with Bregman divergences allows to state the clustering problem within a contrast minimization framework. Namely, through minimizing $R_n(\cb)=P_n d_\phi(u,\cb)$, where $P_n$ denotes the empirical distribution associated with $\{X_1, \dots, X_n\}$ and $Qf(u)$ means integration of $f$ with respect to the measure $Q$, we intend to find a codebook $\hat{\cb}_n$ whose ``real'' distortion $P d_\phi(u,\hat{\cb}_n)$ is close to the optimal $k$-points distortion $R^*_k := \inf_\cb P d_\phi(u,\cb)$.
The convergence properties of empirical distortion minimizers are now quite well understood when the source distribution $P$ is assumed to have a finite support \cite{Linder02,Fischer10}, even in infinite-dimensional cases \cite{Biau08,Levrard15}. In real data sets, the source signal is often corrupted by noise, violating in most cases the bounded support assumption. In practice, data are usually pre-processed via an outlier-removal step that requires an important quantity of expertise. From a theoretical viewpoint, this corruption issue might be tackled by winsorizing or trimming classical estimators, or by introducing some new and robust estimators that adapt to heavy-tailed cases. Such estimators can be based on PAC-Bayesian or Median of Means techniques \cite{Catoni18,Joly15,Lecue17} for instance. In a nutshell, these estimators succeed in achieving sub-Gaussian deviation bounds under mild tail conditions such as bounded variances and expectations, and they are also provably robust to a certain amount of noise \cite{Lecue17}.

In the clustering framework, it is straightforward that the $k$-means procedure suffers from the same drawback as the empirical mean: only one adversarial datapoint is needed to drive both the empirically optimal codebook $\hat{\cb}_n$ and its distortion arbitrarily far from the optimal. In fact we show that it is the case with every possible Bregman divergence. 
Up to our knowledge, the only theoretically grounded attempt to robustify clustering procedures is to be found in \cite{Cuesta97}, where a trimmed $k$-means heuristic is introduced. See also \cite{Garcia08} for trimmed clustering with Mahalanobis distances. In some sense, this paper extends this trimming approach to the general framework of clustering with Bregman divergences. 

       We introduce some notation, background and fundamental properties for trimmed clustering with Bregman divergences in Section \ref{sec:clustering_with_bregman_div}. This will lead to the description of our robust clustering technique, based on the computation of a trimmed empirically optimal codebook $\hat{\cb}_{n,h}$, for a fixed trim level $h$. 
       
       Theoretical properties of our trimmed empirical codebook $\hat{\cb}_{n,h}$ are exposed in Section \ref{sec:theoretical_results}. To be more precise, we investigate convergence towards a trimmed optimal codebook $\cb^*_h$ in terms of distance and distortion, showing for instance that the excess distortion achieves a sub-Gaussian convergence rate of $O(1/\sqrt{n})$ in terms of sample size,  under a mild bounded variance assumption. This shows that our procedure can be thought of as robust whenever noisy situations are modeled as a signal corrupted with heavy-tailed additive noise. We also assess robustness of $\hat{\cb}_{n,h}$ in terms of Finite-sample Breakdown Point (see, e.g., \cite{Maronna06}), showing that our procedure can theoretically  endure a positive proportion of adversarial noise. A precise bound on this proportion is given, that illustrates the possible confusion between too small clusters and noise.
       
        Then, a modified Lloyd's type algorithm is proposed in Section \ref{sec:numerical_experiments}, along  with a heuristic to select both the trim level $h$ and the number $k$ of clusters  from data. The numerical performances of our algorithm are then investigated. We compare our method to trimmed $k$-means \cite{Cuesta97}, tclust \cite{Fritz12}, ToMATo \cite{Chazal_Oudot}, dbscan \cite{DBSCAN} and a trimmed version of $k$-median \cite{Cardot13}. 
Our algorithm with the appropriate Bregman divergence outperforms other methods on samples generated from Gaussian, Poisson, Binomial, Gamma and Cauchy mixtures.
We then investigate the performances of our method on real datasets. First, we consider daily rainfall measurements for January and September in Cayenne, from 2007 to 2017, and try to cluster data according to the month.  As suggested by \cite{Coe82}, our method with the divergence associated with Gamma distributions turns out to be the most accurate one. Second, we intend to cluster chunks of $5000$ words    picked from novels corresponding to $4$ different authors, based on stylometric descriptors \cite[Section 10]{Arnold15}, corrupted by noise. Following \cite{Evert04}, we show that our method with Poisson divergence is particularly well adapted for this framework.

At last, proofs are gathered in Sections \ref{sec:proofs_section_clusteringwithBregmanD} and  \ref{sec:proofs_section_theoretical_results}. Proofs of technical intermediate results are deferred to Sections \ref{tecsec:proofs_definition_DTM}, \ref{tecsec_proofs_sec_theoretical_results}, \ref{tecsec:inter_results_section_clustering_with_BD}, and \ref{tecsec:proofs_inter_theoretical_results}, along with some additional figures and results for Section \ref{sec:numerical_experiments}.

\section{Clustering with trimmed Bregman divergence}\label{sec:clustering_with_bregman_div}
\subsection{Bregman divergences and distortion}\label{sec:bregman_div}
A Bregman divergence is defined as follows.
\begin{definition}
	\label{def:bregman divergence pour les codebooks}
	Let $\phi$ be a strictly convex $\mathcal{C}_1$ real-valued function defined on a convex set $\Omega\subset\R^d$. The \textit{ {Bregman divergence}}  {$d_{\phi}$} is defined for all $x,y\in\Omega$ by
	\[d_{\phi}(x,y)=\phi(x)-\phi(y)-\langle\nabla_y\phi,x-y\rangle.\]
\end{definition}	
Observe that, since $\phi$ is strictly convex, for all $x,y\in\R^d$, $d_{\phi}(x,y)$ is non-negative and equal to zero if and only if $x=y$ (see \cite[Theorem 25.1]{Ro}). Note that by taking $\phi:x\mapsto\|x\|^2$, with $\|\cdot\|$ the Euclidean norm on $\R^d$, one gets $d_{\phi}(x,y)=\|x-y\|^2$. 
Let us present a few other examples:
\begin{enumerate}
	\item Exponential loss:  $\phi:x\mapsto e^x$, from $\R$ to $\R$, leads to $d_\phi(x,y)=e^x-e^y-(x-y)e^y$.
	\item  Logistic loss: $\phi:x\mapsto x\ln x+(1-x)\ln(1-x)$, from $[0,1]$ to $\R$, leads to $d_\phi(x,y)= x\ln\frac x y+(1-x)\ln\big(\frac{1-x}{1-y}\big)$.
	\item Kullback-Leibler: $\phi:x\mapsto \sum_{\ell=1}^dx_\ell\ln x_\ell$, from the $(d-1)-$simplex to $ \R$, leads to $d_\phi(x,y)=\sum_{\ell=1}^dx_\ell\ln \frac{x_\ell}{y_\ell}$.
\end{enumerate}
For any compact set $K \subset \Omega$, and $x \in \Omega$, we also define
\begin{align*}
d_\phi(K,x)  = \min_{y \in K} d_\phi(y,x) \quad \mbox{and} \quad 
d_\phi(x,K) = \min_{y \in K} d_\phi(x,y).
\end{align*} 
For every codebook $\cb=(c_1,c_2,\dots,c_k)\in\Omega^{(k)}$, 
	$d_{\phi}(x,\cb)$ is defined by $d_{\phi}(x,\cb)=\min_{i\in[\![1,k]\!]}d_{\phi}(x,c_i)$.  
The main property of Bregman divergences is that means are always minimizers of Bregman inertias, as exposed below. For a distribution $Q$ and a function $f$, we denote by $Qf(u)$ the integration of $f$ with respect to $Q$.
\begin{proposition}{\cite[Theorem 1]{Ban051}}\label{prop:bregman_bias_variance}
Let $P$ be a probability distribution, and let $\phi$ be a strictly convex $\mathcal{C}_1$ real-valued function defined on a convex set $\Omega\subset\R^d$. Then, for any $x$ $\in$ $\Omega$, 
\[
P d_\phi(u,x) = P d_\phi(u,Pu) + d_\phi(Pu,x).
\]
\end{proposition}
 As mentioned in \cite{Ban05}, this property allows to design iterative Bregman clustering algorithms that are similar to Lloyd's algorithm. 
Let $P$ be a distribution on $\R^d$, and $\cb$ a codebook. The clustering performance of $\cb$ will be measured via its \textit{distortion}, namely
\begin{align*}
R(\cb) = P d_\phi (u,\cb).
\end{align*}
When only an i.i.d. sample $\X_n = \{X_1, \dots, X_n\}$ is available, we denote by $R_n(\cb)$ the corresponding empirical distortion (associated with $P_n$). When $P$ is a mixture of distributions belonging to an exponential family, there exists a natural choice of Bregman divergence, as detailed in Section \ref{sec:numerical_experiments}. Standard Bregman clustering intends to infer a minimizer of $R$ via minimizing $R_n$, and works well in the bounded support case \cite{Fischer10}.

\subsection{Trimmed optimal codebooks}

     As for classical mean estimation, plain $k$-means is sensitive to outliers. An attempt to address this issue is proposed in \cite{Cuesta97,Gordaliza91}: for a trim level $h \in (0,1]$, both a codebook and  a subset of $P$-mass not smaller than $h$ (trimming set) are pursued. This heuristic can be generalized to our framework as follows. 
     
      For a measure $Q$ on $\R^d$, we write $Q \ll P$ (i.e., $Q$ is a sub-measure of $P$) if $Q(A) \leq P(A)$ for every Borel set $A$. Let $\mathcal{P}_h$ denote the set $\mathcal{P}_h = \{ \frac{1}{h} Q\mid Q \ll P, Q(\R^d)=h \}$, and $\mathcal{P}_{+h} = \cup_{s \geq h} \mathcal{P}_s$. By analogy with \cite{Cuesta97}, optimal trimming sets and codebooks are designed to achieve the optimal $h$-trimmed $k$-variation,  
      \begin{align*}
      V_{k,h}  :=   \inf_{\tilde{P} \in \mathcal{P}_{+h}} \inf_{\cb \in \Omega^{(k)}} R(\tilde{P},\cb),
      \end{align*}
where $R(\tilde{P},\cb)= \tilde{P}  d_\phi(u,\cb)$. In other words, $V_{k,h}$ is the best possible $k$-point distortion based on a normalized sub-measure of $P$. Intuitively speaking, the $h$-trimmed $k$-variation may be thought of as the $k$-points optimal distortion of the best "denoised" version of $P$, with denoising level $1-h$. For instance, in a mixture setting, if $P = \gamma P_0 + (1-\gamma) N$, where $P_0$ is a signal supported by $k$ points and $N$ is a noise distribution, then, provided that $h \leq \gamma$, $V_{k,h}=0$.

If $\cb$ is a fixed codebook, we denote by $B_\phi(\cb,r)$ (resp $\bar{B}_\phi(\cb,r)$) the open (resp. closed) Bregman ball with radius $r$, $\left \{x\mid \sqrt{d_\phi(x,\cb)} < r \right \}$ (resp. $\leq$), and by $r_h(\cb)$ the smallest radius $r\geq 0$ such that
\begin{align}\label{eq: def rayon r}
   P (B_\phi(\cb,r)) \leq h \leq P(\bar{B}_\phi(\cb,r)).
\end{align} 

We denote this radius by $r_{n,h}(\cb)$ when the distribution is $P_n$. Note that $r_{n,h}(\cb)^2$ is the Bregman divergence to the $\lceil nh \rceil$ $d_\phi$-nearest-neighbor of $\cb$ in $\X_n$.
Now, if $\mathcal{P}_h(\cb)$ is defined as the set of measures $\tilde{P}$ in $\mathcal{P}_h$ that coincide with $\frac{P}{h}$ on $B_\phi(\cb,r_h(\cb))$, with support included in $\bar{B}_\phi(\cb,r_h(\cb))$, a straightforward result is the following.
\begin{lemma}
\label{lm: ecriture DTM}
For all $\cb\in\Omega^{(k)}$, $h\in(0,1]$, $\tilde P\in \mathcal{P}_h$ and $\tilde{P}_\cb \in \mathcal{P}_h(\cb)$,
\[R(\tilde{P}_\cb,\cb) \leq R(\tilde{P},\cb).\]
Equality holds if and only if $\tilde P\in \mathcal{P}_h(\cb)$.
\end{lemma}
This lemma is a straightforward generalisation of results in \cite[Lemma 2.1]{Cuesta97}, \cite{Gordaliza91} or \cite{Merigot}. A short proof is given in Section \ref{tecsec_proof_lemma_ecriture_DTM}. As a consequence, for any codebook $\cb \in \Omega^{(k)}$ we may restrict our attention to sub-measures in $\mathcal{P}_h(\cb)$.
\begin{definition}\label{def:trimmed_distortion}
For $\cb \in \Omega^{(k)}$, the $h$-trimmed distortion of $\cb$ is defined by
\[
R_h(\cb) = h R(\tilde{P}_\cb,\cb),
\]
where $\tilde{P}_\cb \in \mathcal{P}_h(\cb)$.
\end{definition}
Note that since $R(\tilde{P}_\cb,\cb)$ does not depend on the choice of $\tilde{P}_\cb$ whenever $\tilde{P}_\cb \in \mathcal{P}_h(\cb)$, $R_h(\cb)$ is well-defined. As well, $R_{n,h}(\cb)$ will denote the trimmed distortion corresponding to the distribution $P_n$.
Another simple property of sub-measures can be translated in terms of trimmed distortion. 

\begin{lemma}
\label{lm: dtm decreasing}
Let $0<h<h'<1$ and $\cb\in\Omega^{(k)}$. Then
\[
{R_h(\cb)}/{h} \leq {R_{h'}(\cb)}/{h'}.
\]
Moreover, equality holds if and only if $P(B_\phi(\cb,r_{h'}(\cb)))=0$.
\end{lemma}
As well, this lemma generalizes previous results in \cite{Cuesta97, Gordaliza91}. A proof can be found in Section \ref{sec:proof_lem_dtm_decreasing}. Lemma \ref{lm: ecriture DTM} and Lemma \ref{lm: dtm decreasing} ensure that for a given $\cb$, optimal $\tilde{P}$ in $\mathcal{P}_{+h}$ for $R(\tilde{P}, \cb)$ can be found in $\mathcal{P}_h(\cb)$. Thus, 
the optimal $h$-trimmed $k$-variation may be achieved via minimizing the $h$-trimmed distortion.
\begin{proposition}\label{prop:equiv_distortion_variation} For every positive integer $k$ and $0<h<1$,
\[
hV_{k,h}=\inf_{\cb\in\Omega^{(k)}} R_h(\cb).
\]
\end{proposition}

This proposition is an extension of \cite[Proposition 2.3]{Cuesta97}. In other words, Proposition \ref{prop:equiv_distortion_variation} assesses the equivalence between minimization of our robustified distortion $R_h$, and the original robust clustering criterion depicted in \cite{Cuesta97} (extended to Bregman divergences). Thus, a good  codebook in terms of trimmed $k$-variation can be found by minimizing $R_h$.
\begin{definition}
An \textit{ {$h$-trimmed $k$-optimal codebook}} is any element  {$\cb^*$} in $\arg\min_{\cb\in\Omega^{(k)}}R_h(\cb)$.
\end{definition}
Under mild assumptions on
$P$ and $\phi$,  trimmed $k$-optimal codebooks exist.
\begin{theorem}\label{thm:existence_optimal} Let $0<h<1$, assume that $P \|u\| < + \infty$, $\phi$ is $\mathcal{C}^2$ and strictly convex and $F_0 = \overline{conv(supp(P))} \subset \mathring{\Omega}$, 
that is, the closure of the convex hull of the support of $P$ is a subset of the interior of $\Omega$. 
Then, the set $\arg\min_{\cb\in\Omega^{(k)}} R_h(\cb)$ is not empty.
\end{theorem}
A proof of Theorem \ref{thm:existence_optimal} is given in Section \ref{sec:proof_thm_existence_optimal}. Note that Theorem \ref{thm:existence_optimal} only requires $P\|u\|< + \infty$. This can be compared with the standard squared Euclidean distance case, where $P\|u\|^2 < +\infty$ is required for $R:\cb\mapsto P\|u-\cb\|^2$ to have minimizers. From now on we denote by $\cb^*_h$ a minimizer of $R_h$, and by $\hat{\cb}_{n,h}$ a minimizer of the empirical trimmed distortion $R_{n,h}$.

\subsection{Bregman-Voronoi cells and centroid condition}\label{sec: Bregman-Voronoi cells and centroid condition}

Similarly to the Euclidean case, the clustering associated with a codebook $\cb$ will be given by a tesselation of the ambient space. To be more precise, for $\cb \in \Omega^{(k)}$ and $i \in [\![1,k]\!]$, the Bregman-Voronoi cell associated with $c_i$ is $V_i(\cb) = \{ x\mid\forall j \neq i \quad d_\phi(x,c_i) \leq d_\phi(x,c_j)\}$. Some further results on the geometry of Bregman Voronoi cells might be found in \cite{Nie07}. Since the $V_i(\cb)$'s do not form a partition, $W_i(\cb)$ will denote a subset of $V_i(\cb)$ so that $\{W_1(\cb), \dots, W_k(\cb)\}$ is a partition of $\R^d$ (for instance break the ties of the $V_i$'s with respect to the lexicographic rule). Proposition \ref{prop:centroid} below extends the so-called centroid condition in the Euclidean case to our Bregman setting.

\begin{proposition}\label{prop:centroid}
    Let $\cb \in \Omega^{(k)}$ and $\tilde P_\cb \in \mathcal{P}_h(\cb)$. Assume that for all $i \in [\![1, k]\!]$, $\tilde P_\cb(W_i(\cb)) >0$, and denote by $\mathbf{m}$ the codebook of the local means of $\tilde P_\cb$. In other words, $m_i = \tilde P_\cb (u\mathbbm{1}_{W_i(\cb)}(u))/\tilde P_\cb(W_i(\cb))$. Then
    \begin{align*}
    R_h(\cb) \geq R_h(\mathbf{m}), 
\end{align*}         
    with equality if and only if for all $i$ in $[\![1,k]\!]$, $c_i=m_i$.
\end{proposition}
Proposition \ref{prop:centroid} is a straightforward consequence of Proposition \ref{prop:bregman_bias_variance}, that emphasizes the key property that Bregman divergences are minimized by expectations (this is not the case for the $L_1$ distance for instance). In addition, it can be proved that Bregman divergences are the only loss functions satisfying this property \cite{Ban051}. In line with \cite{Ban05} for the non-trimmed case, Proposition \ref{prop:centroid} provides an iterative scheme to minimize $R_h$, that is detailed in Section \ref{sec:numerical_experiments}.

\section{Theoretical results}\label{sec:theoretical_results}
\subsection{Convergence of a trimmed empirical distortion minimizer}

This section is devoted to investigate the convergence of a minimizer $\hat{\cb}_{n,h}$ of the empirical trimmed distortion $R_{n,h}$. Throughout this section $\phi$ is assumed to be $\mathcal{C}^2$, and $F_0=\overline{conv(supp(P))}\subset\mathring{\Omega}$.  
We begin with a generalization of \cite[Theorem 3.4]{Cuesta97}, assessing the almost sure convergence of optimal empirical trimmed codebooks. 
\begin{theorem}
	\label{thm: convergence ps des codebooks}
	Assume that $P$ is absolutely continuous with respect to the Lebesgue measure and satisfies $P\|u\|^p<\infty$ for some $p>2$, then there exists $\cb^*_h$ an optimal codebook such that
	\[
	\lim_{n\rightarrow +\infty} R_{n,h}(\hat{\cb}_{n,h})=R_h(\cb^*_h)\text{ a.e.}.
	\]
	Moreover, up to extracting a subsequence, we have 
	\[\lim_{n\rightarrow+\infty}D(\hat{\cb}_{n,h},\cb^*_h)=0\text{ a.e.},\]
	where $D(\cb,\cb')=\min_{\sigma\in\Sigma_k}\max_{i\in[\![1,k]\!]}|c_i-c'_{\sigma(i)}|$ and $\Sigma_k$ denotes the set of all permutations of $[\![1,k]\!]$. At last, if $\cb^*_h$ is unique, then $\lim_{n\rightarrow+\infty}D(\hat{\cb}_{n,h},\cb^*_h)=0\text{ a.e.}$ (without taking a subsequence).
	
\end{theorem}
Note that, contrary to \cite[Theorem 3.4]{Cuesta97}, uniqueness of trimmed optimal codebooks is not required in Theorem \ref{thm: convergence ps des codebooks}. A proof is given in Section \ref{sec:proof_thm_convergence ps des codebooks}. Interestingly, slightly milder conditions are required for the trimmed distortion of $\hat{\cb}_{n,h}$ to converge towards the optimal at a parametric rate. 

\begin{theorem}\label{thm:slow_rates}
	Assume that $P\|u\|^p < \infty $, where $p \geq 2$. Further, if $R^*_{k,h}$ denotes the $h$-trimmed optimal distortion with $k$ points, assume that $R^*_{k-1,h} - R^*_{k,h} >0$. Then, for $n$ large enough, with probability larger than $1 - n^{-\frac{p}{2}} - 2 e^{-x}$, we have
	\begin{align*}
		R_h(\hat{\cb}_{n,h}) - R_h(\cb^*_h) \leq \frac{C_P}{\sqrt{n}}(1 + \sqrt{x}).
	\end{align*}
\end{theorem} 
The requirement $R^*_{k-1,h} - R^*_{k,h} >0$ ensures that optimal codebooks will not have empty cells. Note that if $R^*_{k-1,h} - R^*_{k,h} =0$, then there exists a subset $A$ of $\R^d$ satisfying $P(A) \geq h$ and such that the restriction of $P$ to $A$ is supported by at most $k-1$ points, that allows optimal $k$-points codebooks with at least one empty cell. It is worth mentioning that Theorem \ref{thm:slow_rates} does not require a unique trimmed optimal codebook, and only requires an order $2$ moment condition for $\hat{\cb}_{n,h}$ to achieve a sub-Gaussian rate in terms of trimmed distortion. This condition is in line with the order $2$ moment condition required in \cite{Joly15} for a robustified estimator of $\cb^*$ to achieve similar guarantees, as well as the finite-variance condition required in \cite{Catoni18} in a mean estimation framework. A proof of Theorem \ref{thm:slow_rates} is given in Section \ref{sec:proof_thm_slow_rates}. To derive results in expectation, a technical additional condition is needed.

\begin{corollary}\label{cor:slow_rates_expectation}
	Assume that there exists a non-decreasing convex function $\psi$ such that 
	\[\sup_{c \in B(0,t) \cap F_0} \| \nabla_c \phi \| \leq \psi(t).\]
	Assume that $P\|u\|^p < \infty$, with $p\geq 2$, and let $q=p/(p-1)$ be the harmonic conjugate of $p$. If $P \|u\|^q \psi^q \left ( \frac{k\|u\|}{h} \right )< \infty$, then
	\begin{align*}
		\E \left ( R_h(\hat{\cb}_{n,h}) - R_h(\cb^*_h) \right ) \leq \frac{C_P}{\sqrt{n}}.
	\end{align*}
\end{corollary}
A proof of Corollary \ref{cor:slow_rates_expectation} is given in Section \ref{sec:proof_cor_slow_rates_expectation}. Note that such a function $\psi$ exists in most of the classical cases. The requirement $P \|u\|^q \psi^q ( {k\|u\|}/{h} )$ roughly ensures that $P d_\phi^q(u,\hat{\cb}_{n,h})$ remains bounded whenever the event described in Theorem \ref{thm:slow_rates} does not occur. The moment condition required by Corollary \ref{cor:slow_rates_expectation} might be quite stronger than the order $2$ condition of Theorem \ref{thm:slow_rates}, as illustrated below.
\begin{enumerate}
\item In the $k$-means case $\phi(x) = \|x\|^2$ and $\Omega = \R^d$, we can choose $\psi(t) = 2t$. The condition of Corollary \ref{cor:slow_rates_expectation} is satisfied for $P\|u\|^3 < + \infty$.

\item For $\phi(x) = \exp(x)$, $\Omega = \R$, we may choose $\psi(t) = \exp(t)$. The condition of Corollary \ref{cor:slow_rates_expectation} may be written for $p=2$ as $P u^2 \exp\left(\frac{2k|u|}{h}\right) < + \infty$.
\end{enumerate}

\subsection{Robustness properties of trimmed empirical distortion minimizers}
This section is devoted to discuss to what extent the trimming procedure we propose implies robustness of our estimator to adversarial contamination. First we choose to assess robustness via the so-called Finite-sample Breakdown Point \cite{Donoho83}, that seizes what proportion of adversarial noise can be added to a dataset without making estimators getting arbitrarily large. To be more precise, for an amount $s$ of adversarial points $\{x_1, \dots, x_s\}$, we denote by $P_{n+s}$ the empirical distribution associated with $\X_n \cup \{x_1, \dots, x_s\}$ and by $\hat{\cb}_{n+s,h}$ a minimizer of $\cb \mapsto {R}_{n+s,h}(\cb)$. We may then define the Finite-sample Breakdown Point (FSBP) as follows:
\begin{align*}
\widehat{BP}_{n,h} := \inf \left \{ \frac{s}{n+s} \mid \sup_{\{x_1, \dots, x_s\}} \| \hat{\cb}_{n+s,h}\| = + \infty \right \}.
\end{align*}
To give an intuition, the standard mean (minimizer of the $1$-trimmed empirical distortion for $k=1$, $\phi(u) = \|u\|^2$) has breakdown point $\frac{1}{n+1}$, whereas $h$-trimmed means have breakdown point roughly $1-h$ (see, e.g., \cite[Section 3.2.5]{Maronna06}). According to \cite[Theorem 1]{Vandev93}, this is also the case whenever $\phi$ is strictly convex (but still $k=1$). In the case $k>1$, as noticed in \cite{Cuesta97} for trimmed $k$-means, the breakdown point may be much smaller than $1-h$. Note that if an $h$-trimmed optimal codebook has a too small cluster, then adding an adversarial cluster with greater weight might switch the roles between noise and signal, resulting in an $h$-trimmed codebook that allocates one point to the adversarial cluster and trims the too small optimal cluster. To quantify this intuition, we introduce the following discernability factor $B_h$.
\begin{definition}\label{def:discernability_factor}
Let $h \in ]0,1[$, and, for $b \leq h$, denote by $h_b^- = (h-b)/(1-b)$, $h_b^+ = h/(1-b)$. The discernability factor $B_h$ is defined as 
\[
B_h=\sup \left \{b\geq 0 \mid b \leq h \wedge (1-h)\text{ and }\min_{j\in[\![2, k]\!]} R^*_{j-1,h_b^-} - R^*_{j,h_b^+} >0 \right \}. 
\]    
\end{definition}
In fact, $h - h_{B_h}^- = (1-h)B_h/(1-B_h)$ is the portion of mass in an optimal $k$-points $h$ trimming set that may be considered as noise by an optimal $k-1$-points $h_{B_h}^-$ trimming set. As exposed in the following proposition, $B_h$ is related to the minimum cluster weight of optimal $h$-trimmed codebooks. 
\begin{proposition}\label{prop:properties_discernability_factor}
Assume that the requirements of Theorem \ref{thm:existence_optimal} are satisfied. If $R^*_{k-1,h}-R^*_{k,h}>0$, then $B_h >0$.

Moreover, for any $j\in[\![1,k]\!]$, if $\cb^{*,(j)}$ is a $j$-points $h$-trimmed optimal codebook and $p_{j,h} = h \min_{p\in[\![1,j]\!]} \tilde{P}_{\cb^{*,(j)}} \left ( W_p(\cb^{*,(j)}) \right )$, with $\tilde{P}_{\cb^{*,(j)}} \in \mathcal{P}_h(\cb^{*,(j)})$, then
\[
B_h(1-(h-p_{j,h})) \leq p_{j,h}.
\]
\end{proposition} 
A proof of Proposition \ref{prop:properties_discernability_factor} is given in Section \ref{sec:proof_prop_properties_discernability_factor}. Theorem \ref{thm:FSBP_2} below makes connection between this discernability factor and robustness properties of optimal $k$-points $h$-trimmed codebook, stated in terms of Bregman radius.
\begin{theorem}\label{thm:FSBP_2}
For $\ell \geq 1$, let $R^*_{\ell,h}$ denote the $\ell$-points $h$-trimmed optimal distortion. Assume that $P\|u\|^p < +\infty$, for some $p \geq 2$. Moreover, assume that $R^*_{k-1,h}-R^*_{k,h} >0$. Let $b<B_h$, and assume that $s/(n+s) \leq b$. Then, for $n$ large enough, with probability larger than $1-n^{-\frac{p}{2}}$,
\[
\max_{j\in[\![1,k]\!]} d_\phi\left ( B(0,C_{P,b}) ,\hat{c}_{n+s,h,j} \right ) \leq K_{P,b},
\]
where $C_{P,b}$ and $K_{P,b}$ do not depend on $n$ nor $s$.
\end{theorem}
A proof of Theorem \ref{thm:FSBP_2} is given in Section \ref{sec:proof_thm_FSBP2}. Theorem \ref{thm:FSBP_2} guarantees that the proposed trimming procedure is robust in terms of Bregman divergence, that is, the corrupted empirical distortion minimizer belongs to some closed Bregman ball, provided the proportion of noise is smaller than the discernability factor introduced in Definition \ref{def:discernability_factor}. Unfortunately Bregman balls might not be compact sets if $c \mapsto d_\phi(x,c)$ is not a proper map. For instance, with $\phi(x) = e^x$ and $\Omega = \mathbb{R}$, we have $] - \infty, 0 ] \subset\{c \mid d_\phi(0,c) \leq 1 \}$. In the proper map case, Theorem \ref{thm:FSBP_2} entails that the FSBP is larger than $B_h$, with high probability, for $n$ large enough. In the other case, Corollary \ref{cor:FSBP_3} below ensures that this breakdown point is positive, provided that $p>2$.
\begin{corollary}\label{cor:FSBP_3}
Assume that $P\|u\|^p < +\infty$, for $p>2$. Under the assumptions of Theorem \ref{thm:FSBP_2}, there exists $c>0$ such that, almost surely, for $n$ large enough, $\widehat{BP}_{n,h} \geq c$.

In addition, if, for every $x \in \Omega$, $c \mapsto d_\phi(x,c)$ is a proper map, then almost surely, for $n$ large enough $\widehat{BP}_{n,h} \geq B_h$.
\end{corollary}   
A proof of Corollary \ref{cor:FSBP_3} can be found in Section \ref{sec:proof_cor_FSBP3}. Corollary \ref{cor:FSBP_3} guarantees that our trimmed Bregman clustering procedure is asymptotically robust in the usual sense to a certain proportion of adversarial noise, contrary to plain Bregman clustering whose FSBP is $1/(n+1)$. However this unknown authorized proportion depends on both the choice of Bregman divergence and the discernability factor $B_h$. In the proper map case, the FSBP is larger than $B_h$. Note that for $x \in \Omega$, $c \mapsto d_\phi(x,c)$ is proper whenever $\phi$ is strictly convex, that is the case for trimmed $k$-means \cite{Cuesta97}. For this particular Bregman divergence, the result of Corollary \ref{cor:FSBP_3} is provably tight. 
\begin{example}\label{Ex:FSBP}
Let $\phi_1=\|.\|^2$, $\phi_2 = \exp(-.)$, $\Omega = \R$, $P = (1-p) \delta_{-1} +  p\delta_{1}$, with $p\leq 1/2$. Then, for $\phi = \phi_j$, $j \in \{1,2\}$, $k=2$ and $h > (1-p)$, we have $B_h = \frac{h+p-1}{p} \wedge (1-h)$. Let $Q_{\gamma,N} = (1-\gamma)P + \gamma \delta_{N}$. The following holds.
\begin{itemize}
\item If $(1+p)h>1$, $B_h = 1-h$, and for every $\gamma > 1-h$, any sequence of optimal $2$-points $h$-trimmed codebook $\cb^{*}_2(Q_{\gamma,N})$ for $Q_{\gamma,N}$ satisfies 
\[
\lim_{ N \rightarrow + \infty} \|\cb^{*}_2 (Q_{\gamma,N})\| = + \infty.
\]
\item If $(1+p)h \leq 1$, then $B_h = \frac{h+p-1}{p}$, and, for $\gamma=B_h$, $(-1,N)$ is an optimal $2$-points $h$-trimmed codebook for $Q_{\gamma,N}$.
\end{itemize}
\end{example}
The calculations pertaining to Example \ref{Ex:FSBP} may be found in Section \ref{tecsec:proof_ex_FSBP}. Note that upper bounds on the FSBP when $n \rightarrow + \infty$ may be derived for Example \ref{Ex:FSBP} using standard deviation bounds. Example \ref{Ex:FSBP} illustrates the two situations that can be encountered when some adversarial noise is added, depending on the balance between trim level and smallest optimal cluster. If the trim level is high enough compared to the smallest mass of an optimal cluster (first case), then the breakdown point is simply $1-h$, that is the amount of points that can be trimmed. This corresponds to the breakdown point of the trimmed mean (see, e.g., \cite{Vandev93}). When the trim level becomes small compared to the smallest mass of an optimal cluster (second case), optimal codebooks for the perturbed distribution can be codebooks that allocate one point to the noise and trim the small optimal cluster, leading to a breakdown point possibly smaller than $1-h$. This corresponds to the situation exposed in Proposition \ref{prop:properties_discernability_factor}. In both cases, the breakdown point is smaller than $B_h$, thus, according to Corollary \ref{cor:FSBP_3}, it is equal to $B_h$. 

As mentioned in \cite{Cuesta97} for the trimmed $k$-means, in practice, breakdown point and choice of the correct number of clusters are closely related questions. This point is illustrated in Section \ref{sec:author_clustering}, where the correct number of clusters depends on what is considered as noise. From a  theoretical viewpoint, this question is tackled  by Corollary \ref{cor:FSBP_3} and Example \ref{Ex:FSBP}, in the proper map case.

\section{Numerical experiments}\label{sec:numerical_experiments}
\subsection{Description of the algorithm}
The algorithm we introduce is inspired by the trimmed version of Lloyd's algorithm \cite{Cuesta97}, and is also a generalization of the Bregman clustering algorithm \cite[Algorithm 1]{Ban05}. 
We assume that  we observe $\{X_1, \dots, X_n\} = \X_n$, and that the mass parameter $h$ equals $\frac{q}{n}$ for some positive integer $q$. We also let $C_j$ denote the subset of $[\![1,n]\!]$ corresponding to the $j$-th cluster. 

\begin{algorithm}{Bregman trimmed $k$-means}\label{algo:BTKM} 
\begin{itemize}
\item \textbf{Input:}  $\{X_1, \dots, X_n\} = \X_n$, $q$, $k$.
\item \textbf{Initialization:}
Sample $c_1$, $c_2$,\dots $c_k$ from $\X_n$ without replacement, $\cb^{(0)} \leftarrow (c_1, \hdots, c_k)$.

\item \textbf{Iterations:} Repeat until stabilization of $\cb^{(t)}$. 
\begin{itemize}
\item $NN_q^{(t)} \leftarrow$ indices of the $q$ smallest values of  $d_\phi(x,\cb^{(t-1)})$, $x \in \X_n$.
\item For $j\in[\![1,k]\!]$, $C_j^{(t)} \leftarrow W_j(\cb^{(t-1)}) \cap NN_q^{(t)}$. 
\item For $j\in[\![1,k]\!]$, $c_j^{(t)} \leftarrow \left (\sum_{x\in C_j^{(t)}} x \right )/{\left | C_j^{(t)} \right |}$.
\end{itemize}

\item \textbf{Output:} $\cb^{(t)}$, $C_1^{(t)}, \dots, C_k^{(t)}$.
\end{itemize}
\end{algorithm}

As for every EM-type algorithm, initialization may be crucial. This  point  will not be theoretically investigated in this paper. In practice, several random starts will be proceeded. More sophisticated strategies, such as \texttt{$k$-means ++} \cite{Arthur07}, could be an efficient way to address the initialization issue.
An easy consequence of Proposition \ref{prop:centroid} for the empirical measure $P_n$ associated with $\X_n$ is the following. For short we denote by $R_{n,h}$ the trimmed distortion associated with $P_n$.
\begin{proposition}
\label{prop: convergence algo trimmed k-means Bregman}
Algorithm \ref{algo:BTKM} converges to a local minimum of the function $R_{n,h}$.
\end{proposition}

It is worth mentioning that in full generality the output of Algorithm \ref{algo:BTKM} is not a global minimizer of $R_{n,h}$. However, suitable clusterability assumptions as in \cite{Kumar10, Tang16, Levrard18} might lead to further guarantees on such an output.

\subsection{Exponential Mixture Models}
In this section we describe the generative models onto which Algorithm \ref{algo:BTKM} will be applied. Namely, we consider mixtures of distributions belonging to some exponential family. As presented in \cite{Ban05}, a distribution from an exponential family may be associated to a Bregman divergence via Legendre duality of convex functions. For a particular distribution, the corresponding Bregman divergence is more adapted for the clustering than other divergences \cite{Ban05}.

Recall that an exponential family associated to a proper closed convex function $\psi$ defined on an open parameter space $\Theta\subset\R^d$ is a family of distributions $\mathcal F_{\psi}=\left\{P_{\psi,\theta}\mid\theta\in\Theta\right\}$, such that, for all $\theta\in\Theta$, $P_{\psi,\theta}$,  defined on  $\R^d$, is absolutely continuous with respect to some distribution $P_0$, with Radon-Nikodym density $p_{\psi,\theta}$ defined for all $x\in\Omega$ by
\[p_{\psi,\theta}(x)=\exp(\langle x,\theta\rangle-\psi(\theta)).\]
The function $\psi$ is called the cumulant function and $\theta$ is the natural parameter.
For this model, the expectation of $P_{\psi, \theta}$ may be expressed as $\mu(\theta)=\nabla_\theta\psi$.
We define $$\phi(\mu)=\sup_{\theta\in\Theta}\left\{\langle\mu,\theta\rangle-\psi(\theta)\right\}.$$
By Legendre duality, for all $\mu$ such that $\phi$ is defined, we get $\phi(\mu)=\langle\theta(\mu),\mu\rangle-\psi(\theta(\mu))$,
with $\theta(\mu)=\nabla_\mu\phi$.
The density of $P_{\psi,\theta}$ with respect to $P_0$ can be rewritten using the Bregman divergence associated to $\phi$ as follows:
\[p_{\psi,\theta}(x)=\exp(-d_{\phi}(x,\mu)+\phi(x)).\]

In the next experiments, we use Gaussian, Poisson, Binomial and Gamma mixture distributions and the corresponding Bregman divergences. Table \ref{fig:tabdistr} presents the 4 densities together with the functions $\psi$ and $\phi$, as well as the associated Bregman divergences $d_\phi$.

\begin{table}[H]
	\begin{tabular}{l c c c}
		\hline
		Distribution & $p_{\psi,\theta}(x)$ & $\theta$ & $\psi(\theta)$\\
		\hline
		\rowcolor{Gray}
		\cellcolor{white}Gaussian & $\frac{1}{\sqrt{2\pi\sigma^2}}\exp\left(-\frac{(x-a)^2}{2\sigma^2}\right)$ & $\frac{a}{\sigma^2}$ & $\frac{\sigma^2}{2}\theta^2$\\
		Poisson & $\frac{\lambda^x\exp(-\lambda)}{x!}$ & $\log(\lambda)$ & $\exp(\theta)$\\
		\rowcolor{Gray}
		\cellcolor{white}Binomial & $\frac{N!}{x!(N-x)!}q^x(1-q)^{N-x}$ & $\log\left(\frac{q}{1-q}\right)$ & $N\log\left(1+\exp(\theta)\right)$\\
		Gamma & $\frac{x^{k-1}\exp(-\frac{x}{b})}{\Gamma(k)b^k}$ & $-\frac{k}{\mu}$ & $k\log\left(-\frac1\theta\right)$\\
		\hline
	\end{tabular}

	\begin{tabular}{l c c c}
		\hline
		Distribution & $\mu$ & $\phi(\mu)$ & $d_{\phi}(x,\mu)$\\
		\hline
		\rowcolor{Gray}
		\cellcolor{white}Gaussian & $a$ & $\frac{1}{2\sigma^2}\mu^2$ & $\frac{1}{2\sigma^2}(x-\mu)^2$\\
		Poisson & $\lambda$ & $\mu\log(\mu)-\mu$ & $x\log\left(\frac x\mu\right)-(x-\mu)$\\
		\rowcolor{Gray}
		\cellcolor{white}Binomial & $Nq$ & $\mu\log\left(\frac\mu N\right)+(N-\mu)\log\left(\frac{N-\mu}{N}\right)$ & $x\log\left(\frac x\mu\right)+(N-x)\log\left(\frac{N-x}{N-\mu}\right)$\\
		Gamma & $kb$ & $-k+k\log\left(\frac{k}{\mu}\right)$ & $\frac{k}{\mu}\left(\mu\log\left(\frac{\mu}{x}\right)+x-\mu\right)$\\
		\hline
	\end{tabular}

\smallskip

	\caption{Exponential family distributions and associated Bregman divergences.}\label{fig:tabdistr}
\end{table}
As emphasized in \cite{Ban05}, clustering with Bregman divergence may be thought of as a hard-threshold model-based clustering scheme, where components of the model are assumed to belong to some exponential family. The following Remark \ref{ex:bregman_clustering_MM} gives an illustration of this connection in a simple case. 
\begin{remark}\label{ex:bregman_clustering_MM}
We let $k=2$, $\theta_1\neq \theta_2$, $z_1^*, \dots, z_n^*$ be hidden labels in $\{1,2\}$, and $X_1,\dots,X_n$ be an independent sample such that $X_i$ has density
\[
\1_{z_i^*=1}p_{\psi,\theta_1}(x) + \1_{z_i^*=2}p_{\psi,\theta_2}(x),
\]
where $p_{\psi,\theta_j}(x)=\exp(-d_\phi(x,\mu_j)+\phi(x))$, for $j \in \{1,2\}$. The parameters of this model are $(z_i^*)_{i\in[\![1,n]\!]}, \theta_1, \theta_2$. This model slightly differs from a classical mixture model since the labels are not assumed to be drawn at random. 

Let $z_{i,j}$, $i\in[\![1,n]\!]$, $j \in \{1,2\}$, denote assignment variables, that is such that $z_{i,j}=1$ if $X_i$ is assigned to class $j$ and 0 otherwise. Also denote by $m=\sum_{i=1}^nz_{i,1}$, $n-m=\sum_{i=1}^nz_{i,2}$, $\bar{X}_1 = \sum_{i=1}^n X_i z_{i,1}/m$, $\bar{X}_2 = \sum_{i=1}^n X_i z_{i,2} /(n-m)$. Maximizing the log-likelihood of the observations boils down to maximizing in $(z_{i,j})_{i,j}$:
\begin{multline*}
\ln \prod_{i=1}^{n}\exp \left [ - z_{i,1} d_\phi\left(X_i,\bar{X}_1 \right ) - z_{i,2} d_\phi\left(X_i,\bar{X}_2\right) +\phi(X_i) \right ]
\\=-\sum_{i=1}^{n}z_{i,1}d_\phi\left(X_i,\bar{X}_1\right)-\sum_{i=1}^{n}z_{i,2}d_\phi\left(X_i,\bar{X}_2\right)+\sum_{i=1}^{n}\phi(X_i).
		\end{multline*}
On the other hand, since optimal codebooks are local means of their Bregman-Voronoi cells (Proposition \ref{prop:centroid}), minimizing $P_n d_\phi(.,\cb)$ is equivalent to minimizing
$\sum_{i=1}^{n}z_{i,1}d_\phi\left(X_i,\bar{X}_1\right) +\sum_{i=1}^{n}z_{i,2}d_\phi\left(X_i,\bar{X}_2\right)$. Thus, clustering with Bregman divergences is the same as maximum likelihood clustering based on this model. 
Further, if we assume that $\mu_1$ and $\mu_2$ are known, then the Bregman assignment rule $x\mapsto\arg\min_{j  \in \{1,2\}} d_\phi(x,\mu_j)$ is the Bayes rule.
\end{remark}

\subsection{Calibration of trimming parameter and number of clusters}
\label{sub: trimming parameter}

 When the number of clusters $k$ is known beforehand, we propose the following heuristic to select the trimming parameter $q$, that is, the number of points in the sample which are assigned to a cluster and not considered as noise. We let  $q$ vary from $1$ to the sample size $n$,  plot the curve $q \mapsto cost[q]$ where $cost[q]$ denotes the optimal empirical distortion at trimming level $q$, and choose $q^\star$ by seeking for a cut-point on the curve. Indeed, when the parameter $q$ gets large enough, it is likely that the procedure begins to assign outliers to clusters, which dramatically deprecates the empirical distortion.

Whenever both $k$ (number of clusters) and $q$ are unknown, we propose to select these two parameters following the same principle as the algorithm \texttt{tclust} \cite{Fritz12}. First we draw, for different values of $k$, the cost curves $q \mapsto cost_k[q]$, for $1 \leq q \leq n$. For each curve, the $q$'s for which there is an abrupt slope increase can correspond to cases where outliers are assigned to clusters, or where some small clusters are included in the set of signal points (if $k$ is chosen too small). In the sequel, we split $[\![1,n]\!]$ into several bins $[\![q_{j},q_{j+1}]\!]$. On every such bin, we select a $k$ that provides a significant cost decrease, as well as the $q$ yielding a slope jump. Note that this heuristic may result in several possible pairs $(k,q)$, corresponding to different point of views, depending on what data point are considered as outliers or not. An illustration of this fact is given in Section \ref{sec:author_clustering}, where outliers consist in small additional clusters.

\subsection{Comparative performances of  Bregman clustering for mixtures with noise}
\label{sec: Comparative performances of  Bregman clustering for mixtures with noise}

To assess the good behavior of our procedure with respect to outliers, we replicate some experiments in \cite{Ban05}, with additional noise. We consider mixture models of Gaussian, Poisson, Binomial, Cauchy and Gamma distributions in $\mathbb{R}^2$. Namely, we sample $100$ points from $X=(X^1,X^2)$, where $X^1$ and $X^2$ are independent, distributed according to a mixture distribution with $3$ components. In each case, the means of the components are set to $10,20,40$. The weights of the components are $(1/3,1/3,1/3)$. We also consider a mixture of 3 different components in $\mathbb{R}^2$: Gamma, Gaussian and Binomial, with respective means $10$, $20$ and $40$. In the Gaussian case, the standard deviations of the components are set to $5$, in the Binomial case, the number of trials are set to $100$ and in the Gamma case the shape parameters are set to $40$. Since Gaussian and Cauchy distributions take negative values, we force the points from each components to lie respectively in the squares $[0,20]^2$, $[0,40]^2$ and $[0,80]^2$. 
$20$ outliers are added, uniformly sampled on $[0,60]^2$.
 
First, we use Algorithm \ref{algo:BTKM} with $20$ random starts for each of these noisy  mixture distributions, using the corresponding divergence, and also make the same experiment for the Cauchy distribution with squared Euclidean distance. 
     For these procedures, we select $k$ and $q$ following the heuristic exposed in Section \ref{sub: trimming parameter}. According to Figure \ref{fig:choicekq_simul}, this leads to the choice $k=3$, $q = 104$ for the Gaussian mixture and $q = 110$ for the other mixtures. The resulting partitions for the selected parameters are depicted in Figure \ref{fig:Clustering2}. 
     
      	 \begin{figure}[H]
	     \begin{minipage}[h]{.3\linewidth}
		\centering\includegraphics[scale = 0.14]{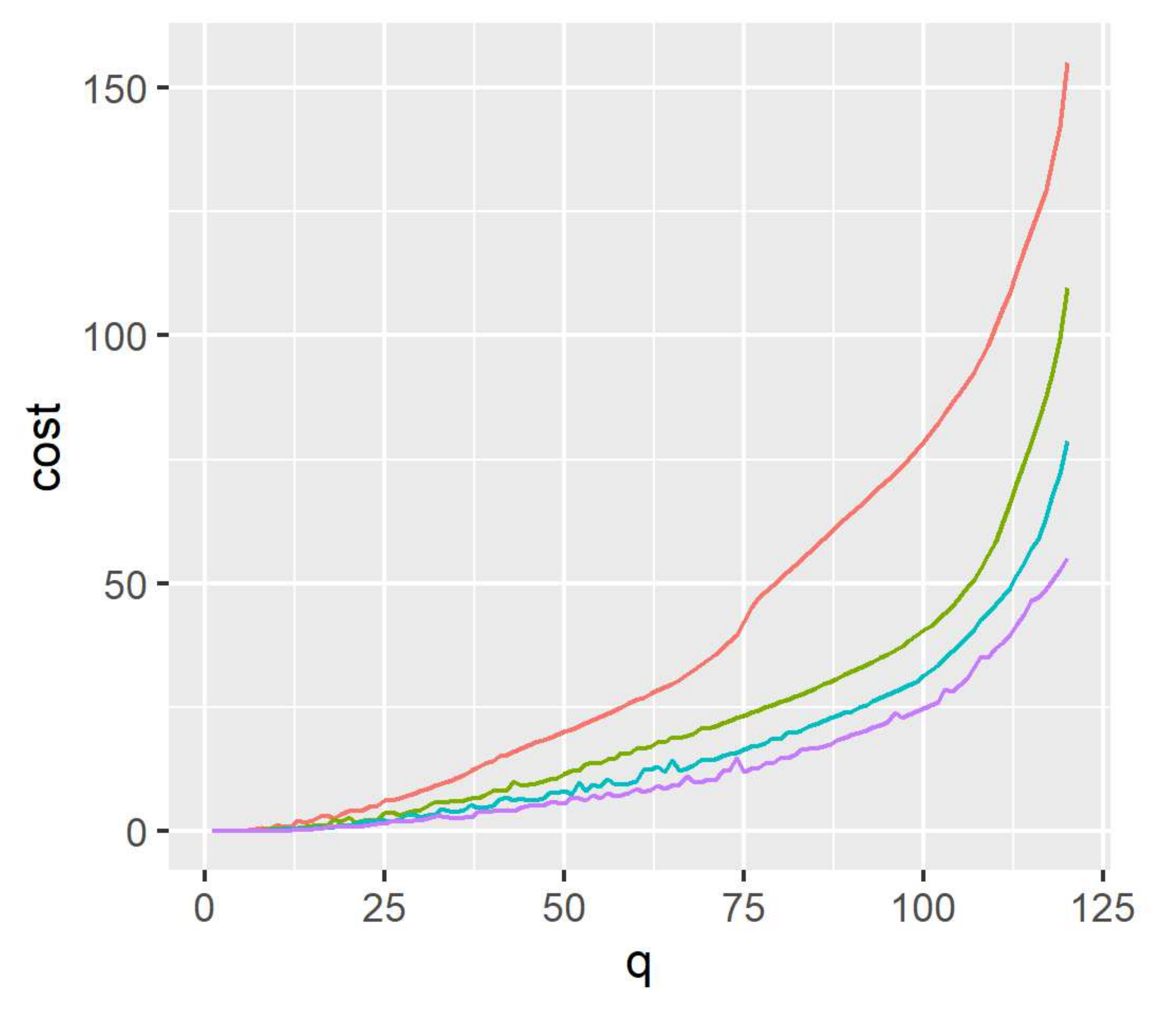}\\
		\footnotesize{Gaussian mixture}
	\end{minipage}\hfill  
	\begin{minipage}[h]{.3\linewidth}
		\centering\includegraphics[scale = 0.14]{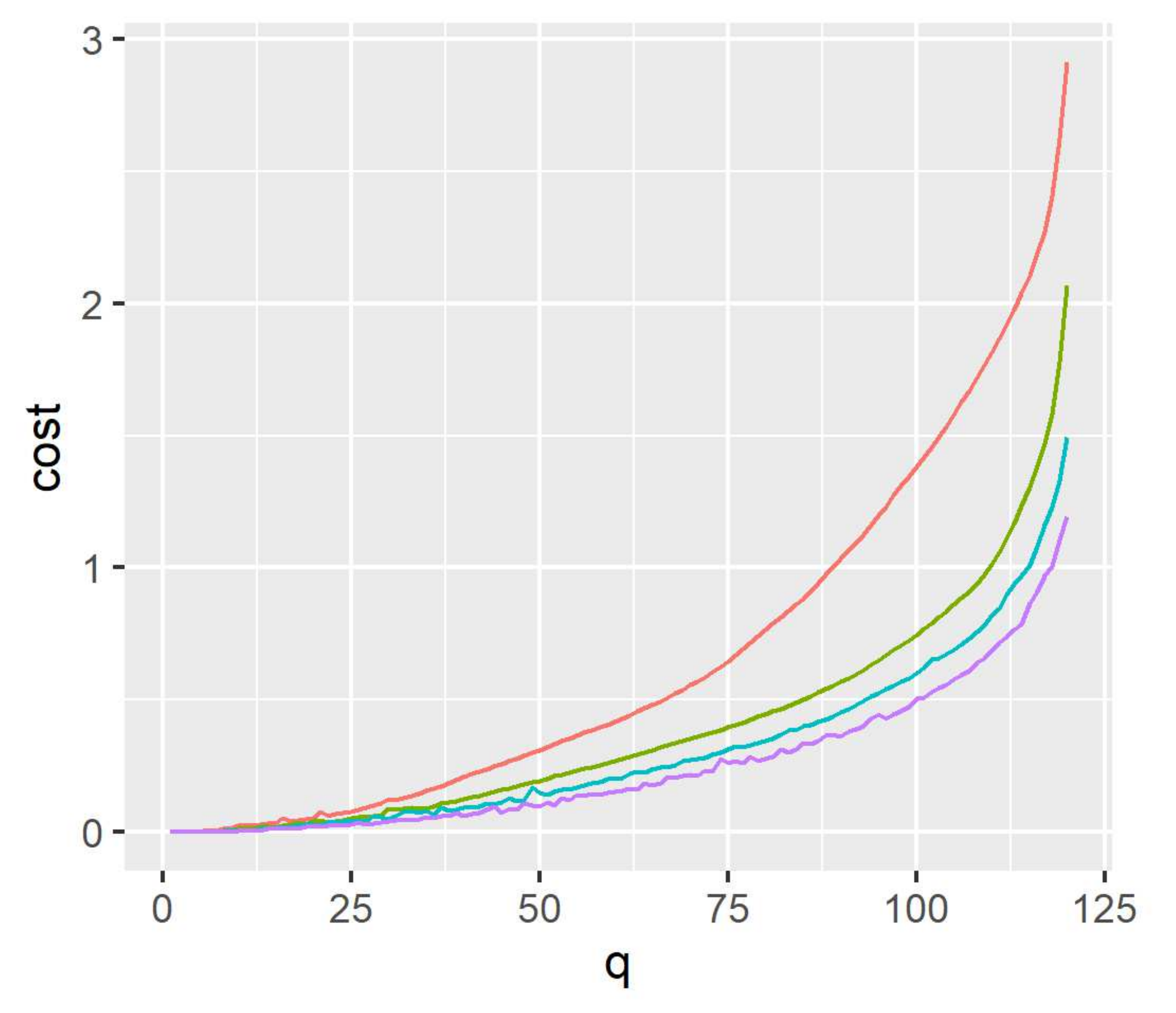}\\
		\footnotesize{Poisson mixture}
		\end{minipage}\hfill
     \begin{minipage}[h]{.36\linewidth}
		\centering\includegraphics[scale = 0.14]{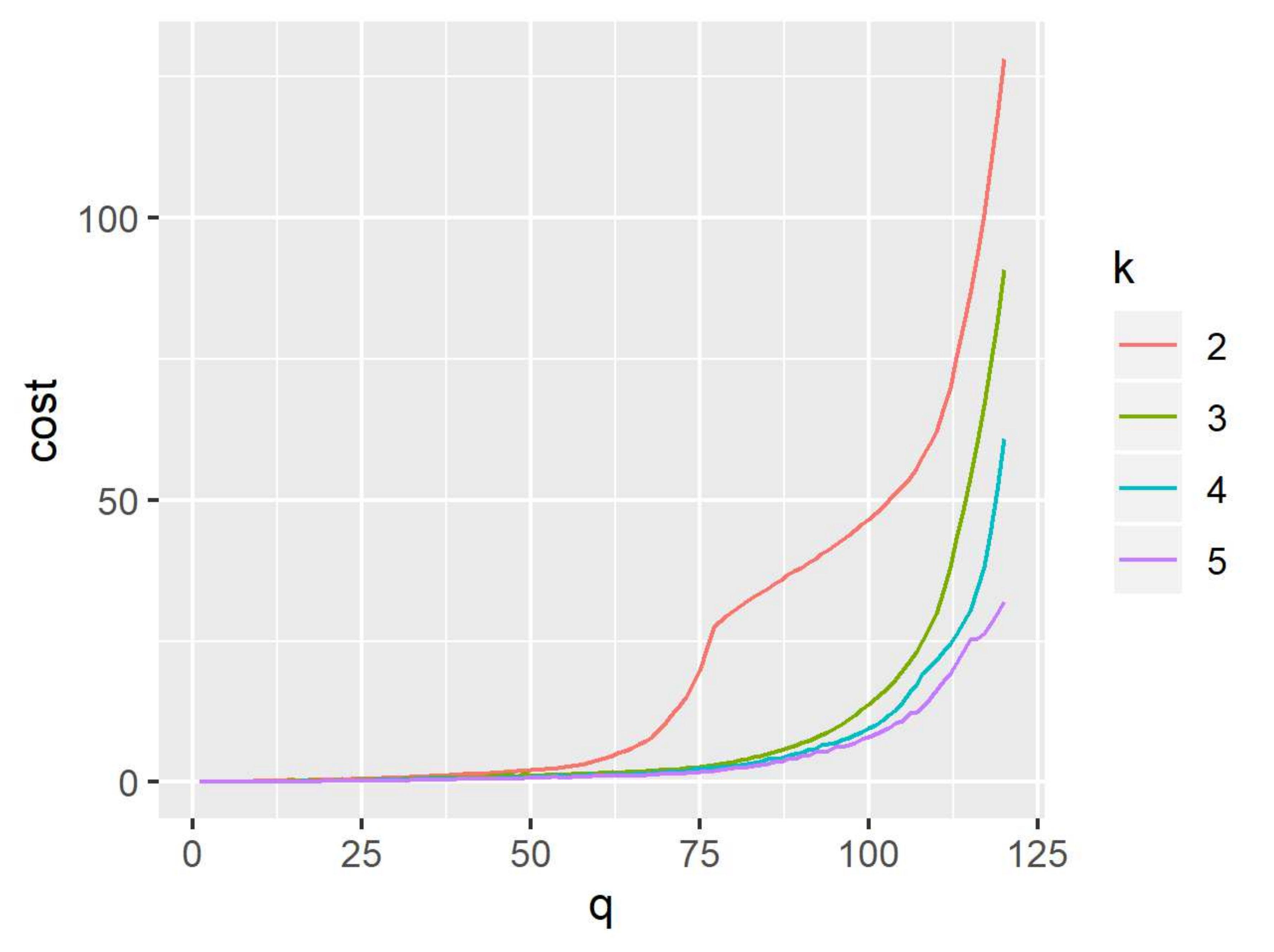}\\
		\footnotesize{Cauchy mixture}
		\end{minipage}
		\caption{Cost curves for selection of $k$ and $q$.}
		\label{fig:choicekq_simul}
		\end{figure}
		
		\begin{figure}[h]
	\begin{minipage}[h]{.3\linewidth}
		\centering\includegraphics[scale=0.14]{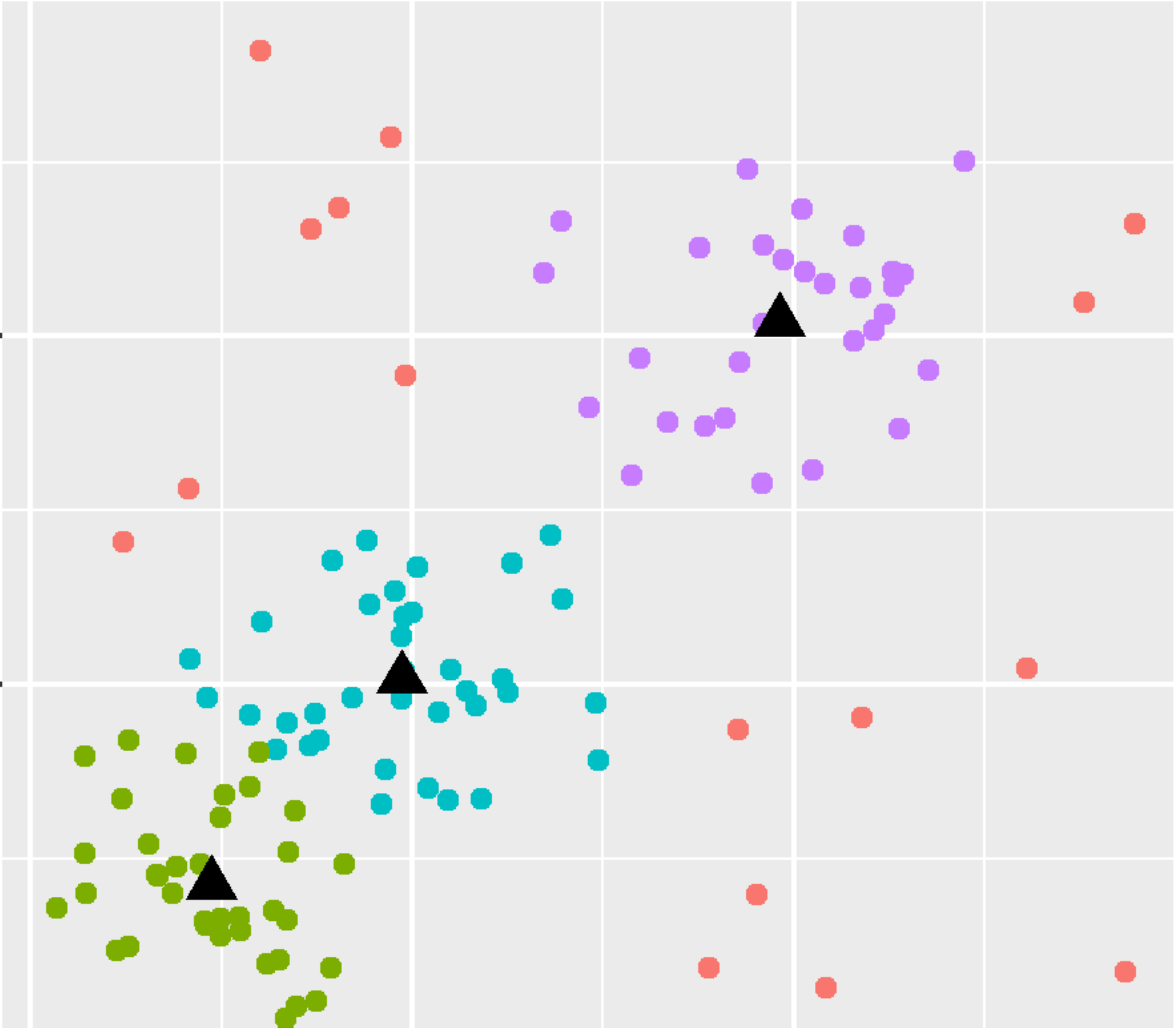}\\
		Gaussian
	\end{minipage}  \hfill
	\begin{minipage}[h]{.3\linewidth}
		\centering\includegraphics[scale=0.14]{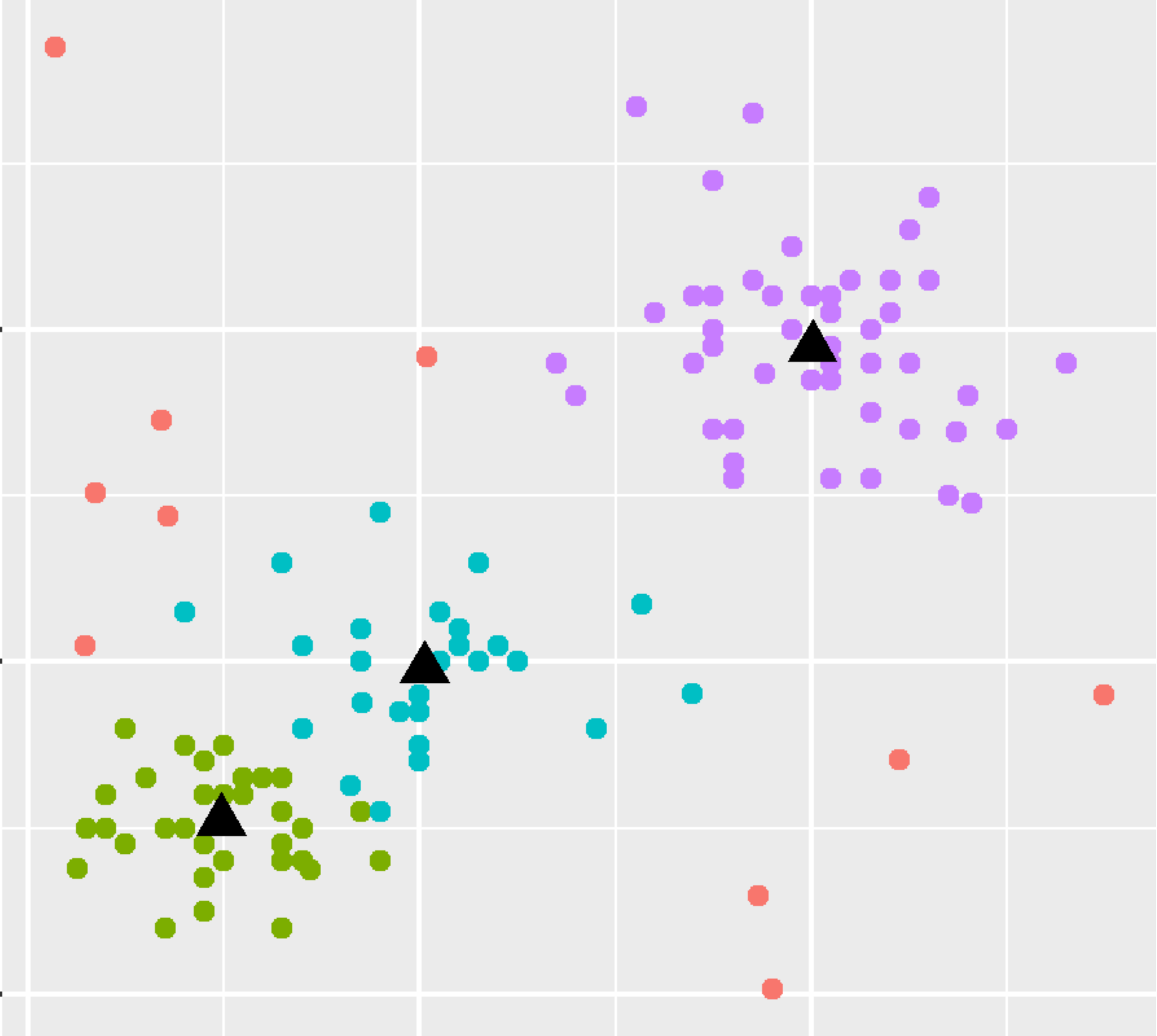}\\
		Poisson
	\end{minipage} \hfill
	\begin{minipage}[h]{.36\linewidth}
		\centering\includegraphics[scale=0.14]{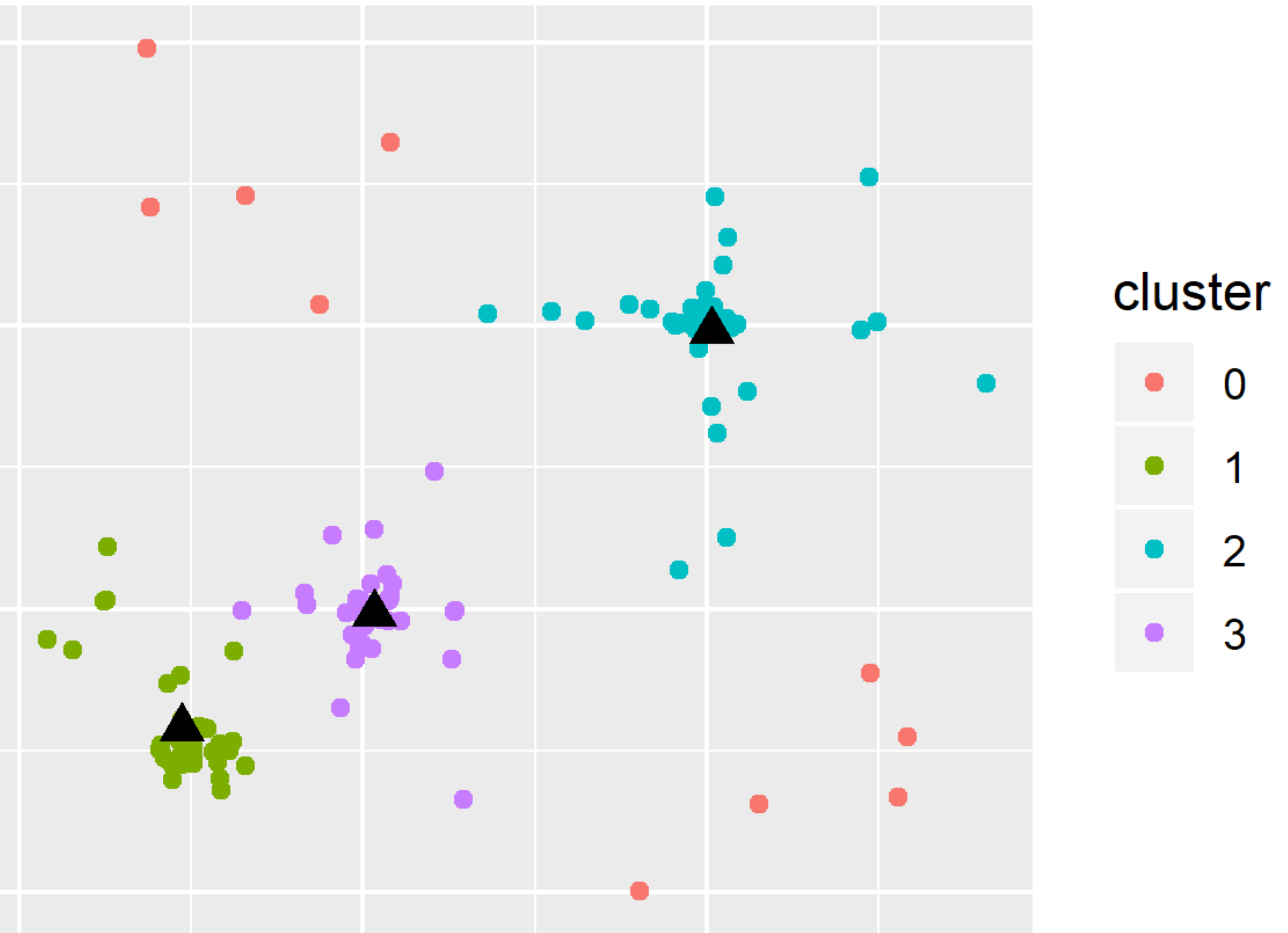}\\
		Cauchy
	\end{minipage}
	\caption{Clustering associated to the selected parameters $k$ and $q$, where cluster $0$ refers to noise.}\label{fig:Clustering2}
\end{figure}
Then we compare the proposed method, in every case, to clustering with other Bregman divergences (including trimmed $k$-means \cite{Cuesta97}), trimmed $k$-median \cite{Cardot13}, \texttt{tclust} \cite{Fritz12}, and density/distance functions-based clustering schemes such as a robustified version of the classical single linkage procedure, the ToMATo algorithm \cite{Chazal_Oudot} with the inverse of the distance-to-measure function \cite{Chazal11} and \texttt{dbscan} \cite{DBSCAN}. Details concerning these methods are available in Section \ref{sec: Details sur les procedures de clustering}. Quality of partitions is assessed via the normalized mutual information (NMI, \cite{SG}) with respect to the ground truth clustering, where the ``noise" points are assigned to one same cluster. 


\begin{figure}
\includegraphics[scale = 0.55]{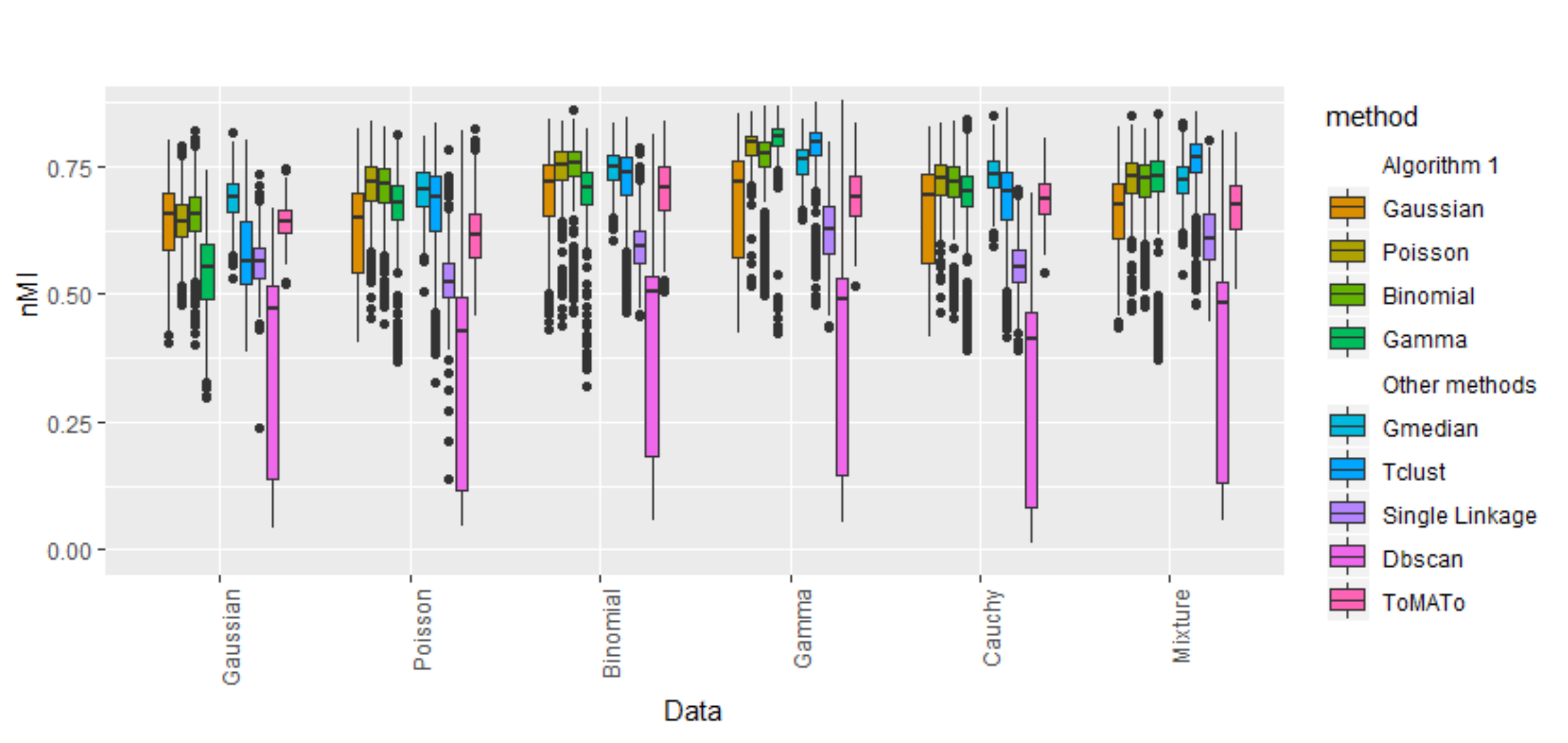}
\caption{Comparison of robust clustering methods, for mixtures of Binomial, Gamma, Gaussian, Poisson, Cauchy, and heterogeneous distributions.}
\label{fig:boxplot_NMI}
\end{figure}
This experiment is repeated $1000$ times, the results in terms of NMI's are exposed in Figure \ref{fig:boxplot_NMI}: Algorithm \ref{algo:BTKM} refers to our method with $q=110$ and $k=3$. For the Cauchy and heterogeneous distributions, the Gaussian divergence is less efficient, since the 3 clusters have increasing variances. This divergence is well suited for clusters with the same variance, the Gamma and Poisson divergence for data with increasing variance, the Binomial divergence for data with increasing and then decreasing variances, for the proper parameter $N$. It is also possible to choose a different Bregman divergence for the different coordinates. Further explanations and numerical illustrations are available in Section \ref{sec: Discussion about the choice of the Bregman divergence}. Choosing between a Gamma or a Poisson divergence depends on the knowledge on the data, as illustrated in the following section with two different real datasets. Note that Algorithm \ref{algo:BTKM} with the proper Bregman divergence (almost) systematically outperforms other clustering schemes. This point is confirmed in Section \ref{sec: Les echantillons de taille 12000} for large datasets ($n=12000$). 


\subsection{Daily waterfall data}

We consider the daily rainfalls (expressed in mm) for january ($241$ data points) and september ($88$ data points), from 2007 to 2017, in Cayenne/Rochambeau. Datapoints are defined as the amount of rain within a rainy day. According to \cite{Coe82}, the positive daily rainfalls within one month are often modeled as Gamma distribution with parameters depending on the month. We experiment Algorithm \ref{algo:BTKM} with the Gamma and the Gaussian Bregman divergences (the latter is plain trimmed $k$-means). The NMI's between the true labels (i.e. the month from which the datapoint was extracted) and the labels returned by the algorithm for different trimming parameters $q$ are depicted in the right panel of Figure \ref{fig:waterfall_clusterings_NMIS}. When $q$ is small, the Gaussian divergence yields better NMI's than Gamma. In this case, outliers are considered as a significant cluster in the computation of the NMI. Thus, the ``outlier'' cluster associated with Gaussian divergence seems closer to a real cluster than the Gamma one. When $q$ is large enough (small amount of outliers), the clustering associated with Gamma divergence outperforms the Gaussian clustering. The left panel of Figure \ref{fig:waterfall_clusterings_NMIS} depicts the associated clustering, for $q=300$. Of course we cannot expect a perfect clustering since the true clusters are not well-separated. However, it seems that the Gamma divergence clustering allows to consider small precipitations as outliers, contrary to the Gaussian case.  
This point can be further exploited to choose in practice an appropriate Bregman divergence for to the data to be clustered. For instance, in the case of positive data points, if noise points are expected close to zero, then a Poisson or Gamma divergence might be more suitable than a Gaussian one. Again, the choice of an appropriate Bregman divergence depends on prior knowledge on the structure of data and noise. 
\begin{figure}[H]
	\begin{minipage}[h]{.28\linewidth}
\centering\includegraphics[scale=0.070]{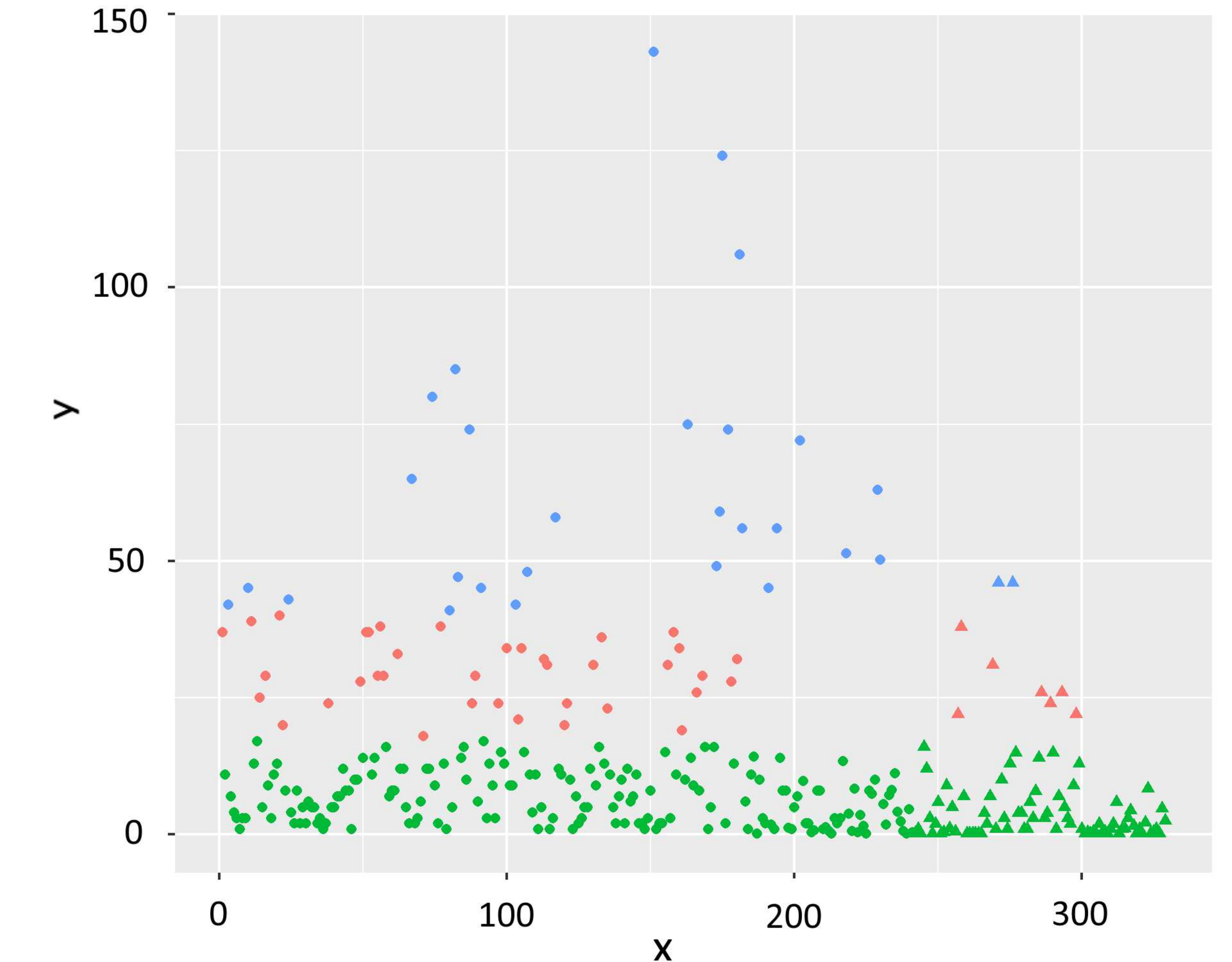}
	\end{minipage}  \hfill
	\begin{minipage}[h]{.35\linewidth}
		\centering\includegraphics[scale=0.070]{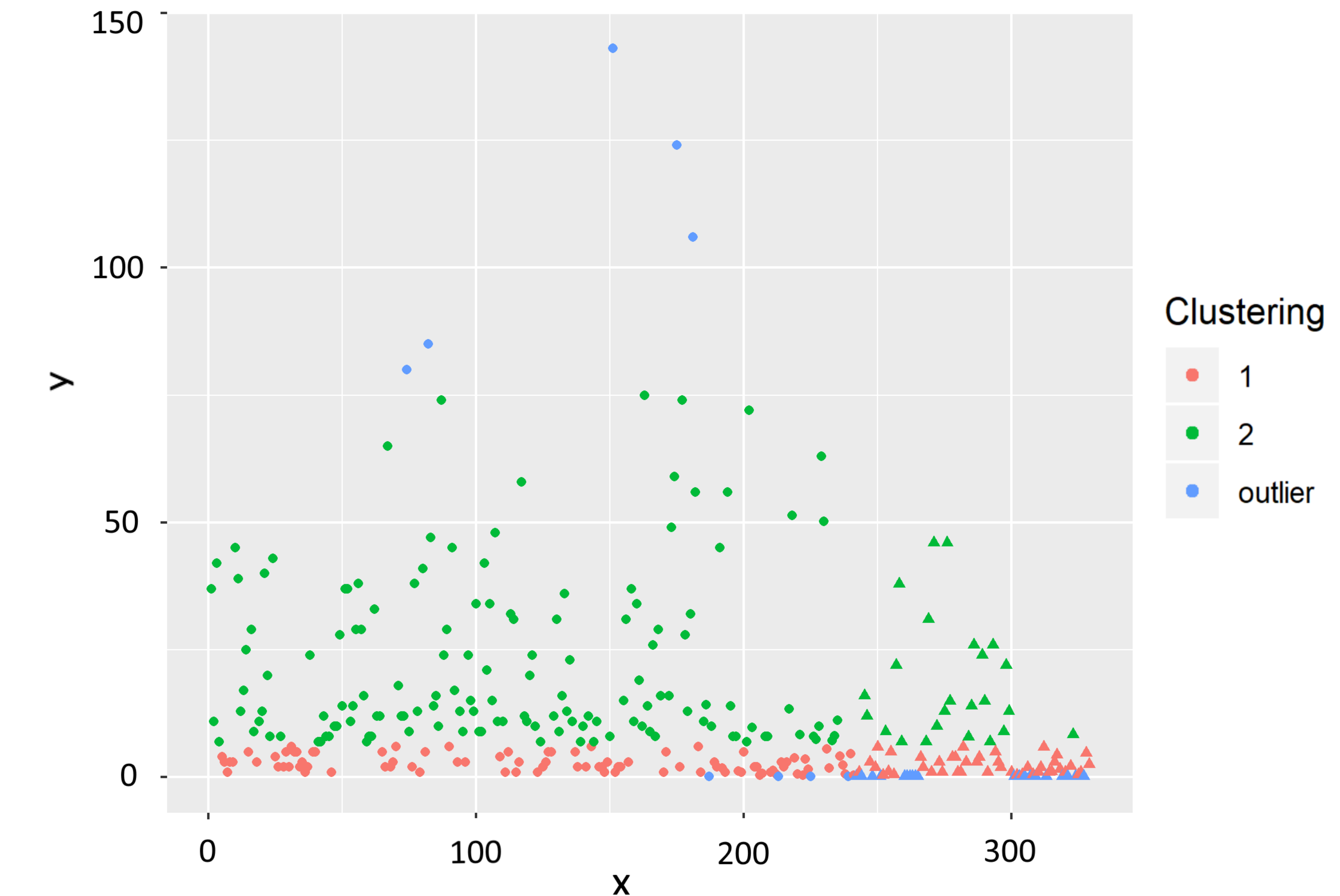}\\
	\end{minipage} \hfill
	\begin{minipage}[h]{.35\linewidth}
\centering\includegraphics[scale=0.070]{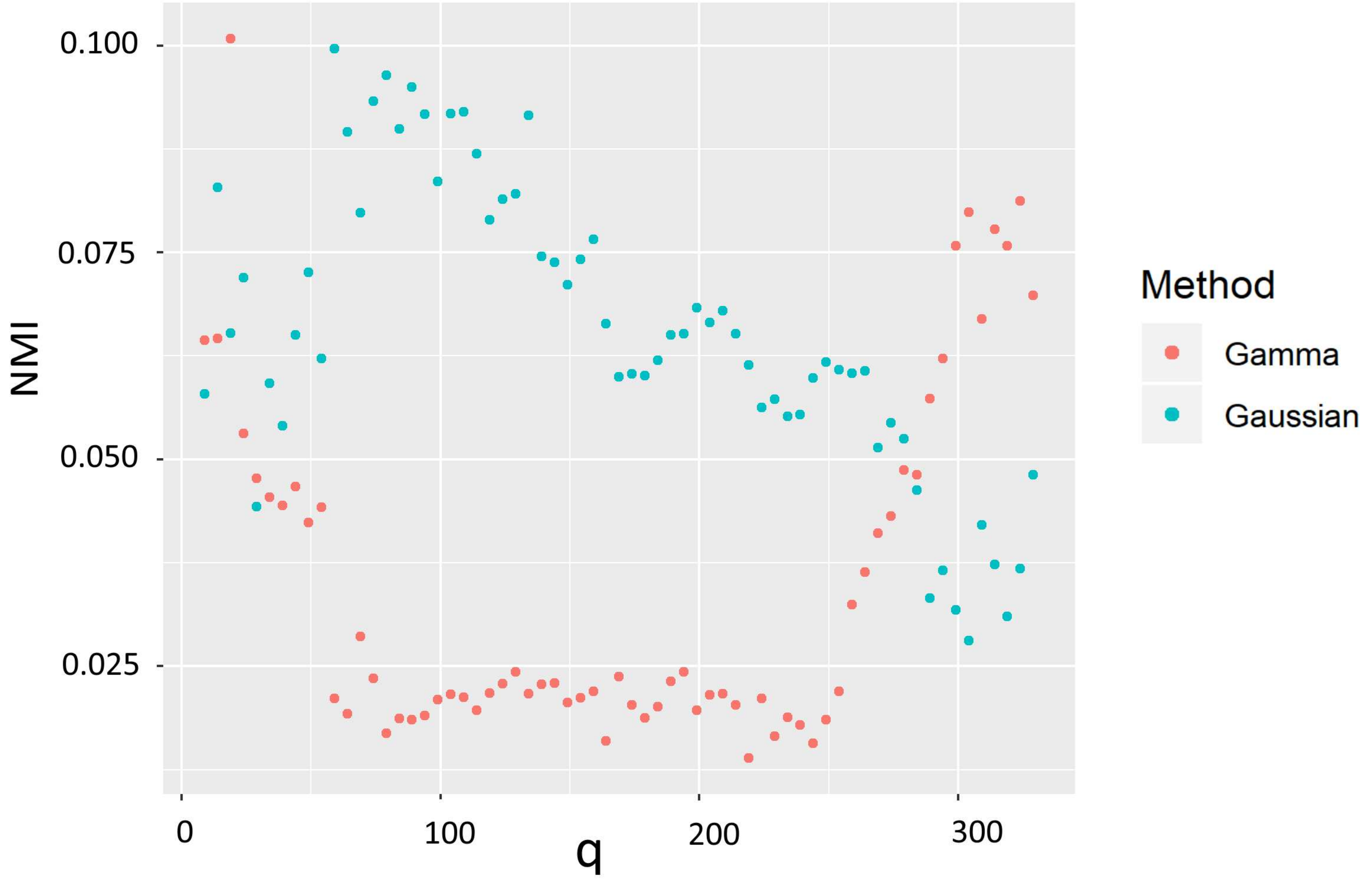}
	\end{minipage}  \hfill
	\caption{From left to right: Clustering with Gaussian divergence and $q=300$.  Clustering with Gamma divergence and $q=300$.  NMI as a function of the trimming parameter $q$.}
	\label{fig:waterfall_clusterings_NMIS}
\end{figure}

\subsection{Authors stylometric clustering}\label{sec:author_clustering}

In this Section we perform clustering on texts based on stylometric descriptors exposed in \cite[Section 10]{Arnold15}. To be more precise, raw data consist in $26$ annotated texts from $4$ authors (Mark Twain, Sir Arthur Conan Doyle, Nathaniel Hawthorne and Charles Dickens). These texts are available as supplementary material for \cite{Arnold15}, and are framed as a sequence of lemmatized string characters (for instance "be" and "is" are instances of the same lemma "be"). Following \cite{Arnold15}, we base our stylometric comparison on lemmas corresponding to nouns, verbs and adverbs, and split every original text in chunks of size $5000$ of such lemmas that will be considered as data points. Then the $50$ overall most frequent lemmas are chosen, and every chunk is described as the vector of counts of these lemmas within it. Thus, signal points consists of $189$ count vectors with dimension $50$, originating from $4$ different authors. 

The signal points are corrupted using the same process for the $8$ State of the Union Addresses given by Barack Obama (available in \texttt{obama} dataset from package CleanNLP in R), resulting in $5$ additional points, and for the King James Version of the Bible (available on Project Gutenberg) that we preliminary lemmatize using the CleanNLP package, resulting in $15$ more additional points. Our final dataset consists of the $189$ signal points and the $20$ outlier points described above. Slightly anticipating, these $20$ outliers might also be thought of  as two additional small clusters with size $5$ and $15$.

Since every individual lemma count can be modeled as a Poisson random variable in the random character sequence model \cite{Evert04}, the appropriate Bregman divergence for this dataset is likely to be the Poisson divergence. In the following, we compare our method with Poisson divergence to trimmed $k$-means, trimmed $k$-medians, and $t$-clust.
\begin{figure}[h]
	\begin{minipage}[h]{.9\linewidth}
		\centering\includegraphics[scale=0.13]{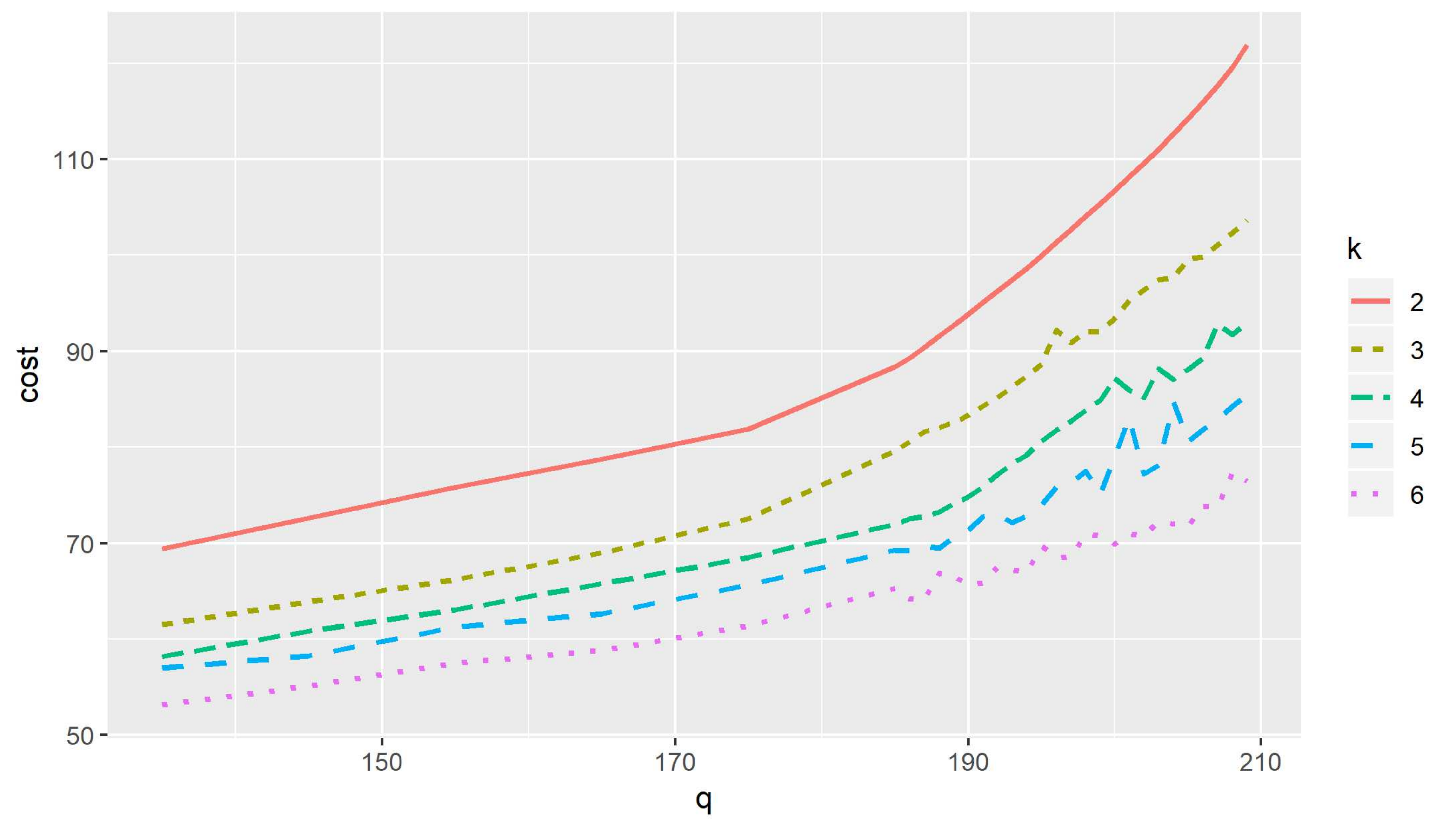}\\
	\end{minipage}
\caption{Cost curves for authors clustering with Poisson divergence.\label{Poisson divergence for authors _ cost}}
\end{figure}
 
In Figure \ref{Poisson divergence for authors _ cost}, we draw the cost of our method as a function of $q$, for different cluster numbers $k$. According to this figure, several choices of $k$ and $q$ are possible. For values of $q$ up to $175$, the significant jumps in the risk function are for $k=3$ and $k=6$. For $k=3$, the slope heuristic yields $q=175$, whereas for $k=6$ the slope heuristic suggests that no data points might be considered as outliers. When $q$ ranges between $175$ and $193$, the significant distortion jumps are for $k=4$ and $k=6$, another possible choice is then $k=4$ and $q=188$. When $q$ is larger than $193$, the only significant jump is for $k=6$. To summarize, the pairs $(k=3, q=175)$, $(k=4, q=188)$, $(k=6, q=n=209)$ seem reasonable. These three solutions correspond  to the 3 natural trimmed partitions: clustering only $3$ authors writings (Twain writings being considered as outliers), clustering the 4 authors writings and removing the outliers from the Bible and B. Obama addresses, and at last clustering the six sources of writings (none of them being considered as noise). The two latter situations are depicted in Figure \ref{Our Author clustering with Poisson divergence}, in the $2$-dimensional basis given by a linear discriminant analysis of the proposed clustering.

\begin{figure}[h]
		\centering\includegraphics[width=.5\linewidth]{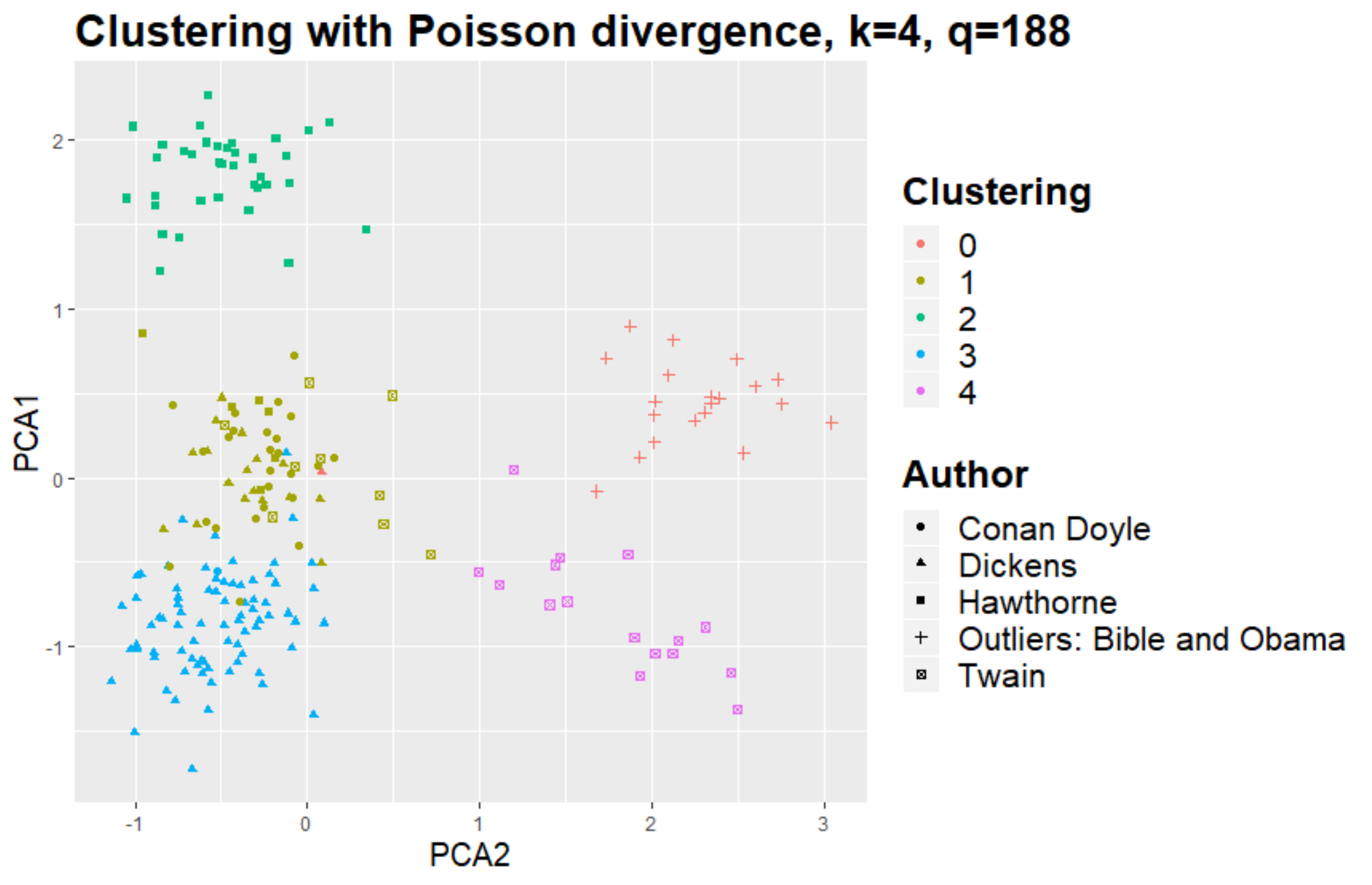}\hfill
		\includegraphics[width=.5\linewidth]{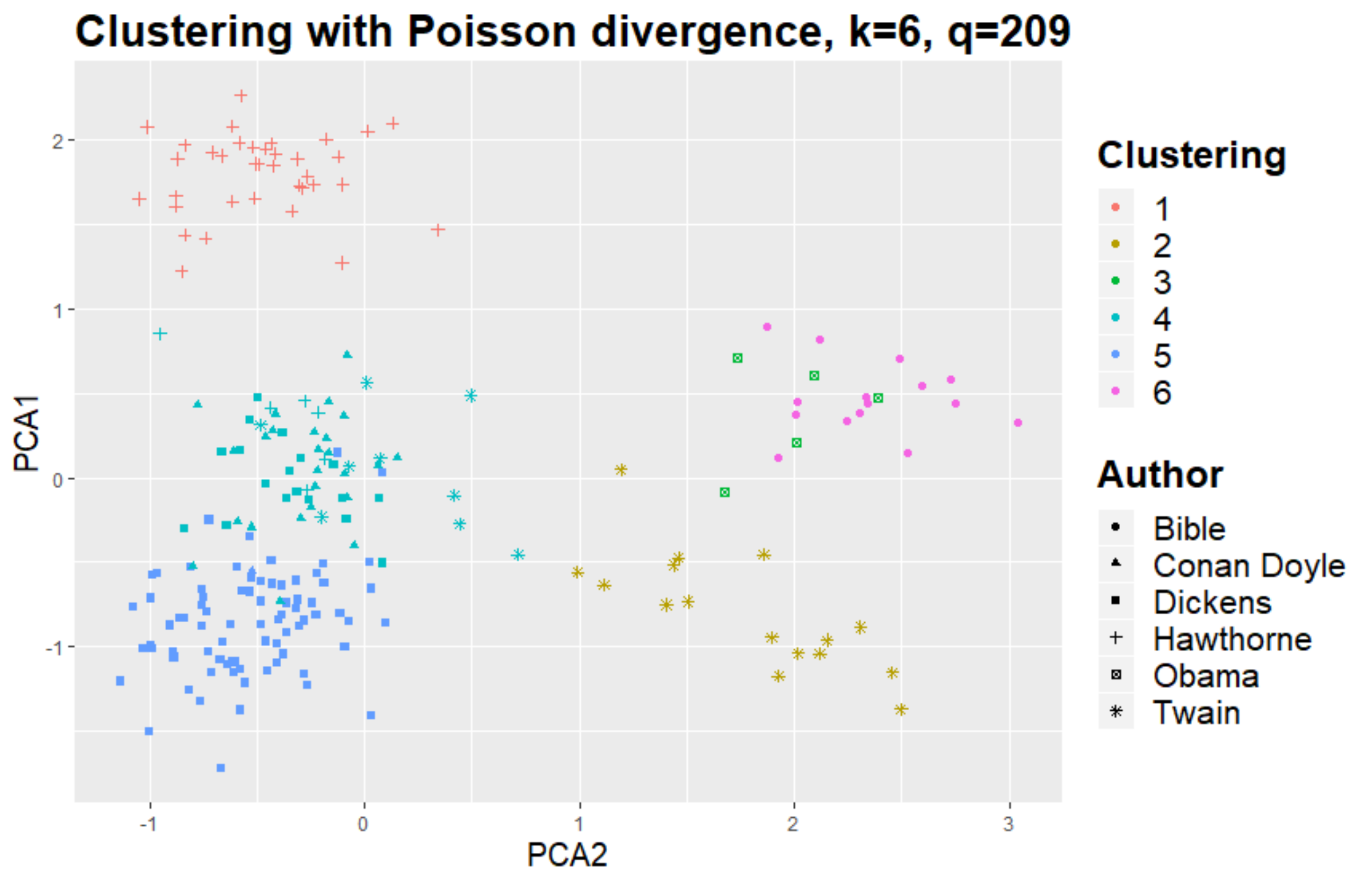}
\caption{Author stylometric clustering with Poisson divergence.\label{Our Author clustering with Poisson divergence}}
\end{figure}
For $k=6$ and $q=209$, our clustering globally retrieves the corresponding author. 
 When $k=4, q=188$ is chosen, outliers are correctly identified and only one sample text from C. Dickens is labeled as outlier. The sample points seem on the whole well classified, that is assessed by a NMI of $0.7347$. 
 This performance is compared with the other clustering algorithms in Table \ref{fig:author_NMI}. Note that values of $q$ have been chosen to minimize the NMI, leading to $q=190$ for trimmed $k$-means, $q=202$ for trimmed $k$-medians, and $q=184$ for \texttt{tclust}. The NMI curves may be found in Section \ref{tecsec:supp_author}.

\begin{table}[h]
	\begin{tabular}{l c c c c}
		\hline
		Method & trimmed $4$-means & trimmed $4$-medians & tclust & Poisson \\
		\hline
		\rowcolor{Gray}
		\cellcolor{white}NMI & $0.5336$ & $0.4334$ & $0.4913$ & $0.7347$\\
		\hline
		\end{tabular}
		\smallskip		
\caption{Comparison of robust clustering methods for Author retrieving.}	
\label{fig:author_NMI}	
\end{table}
     The associated partitions for $k$-median and \texttt{tclust} are depicted in Figure \ref{fig:author_clusterings_4_kmedian_tclust}, showing that these two methods fail in correctly identifying  outliers.		
\begin{figure}[H]
		\centering\includegraphics[width=0.5\linewidth]{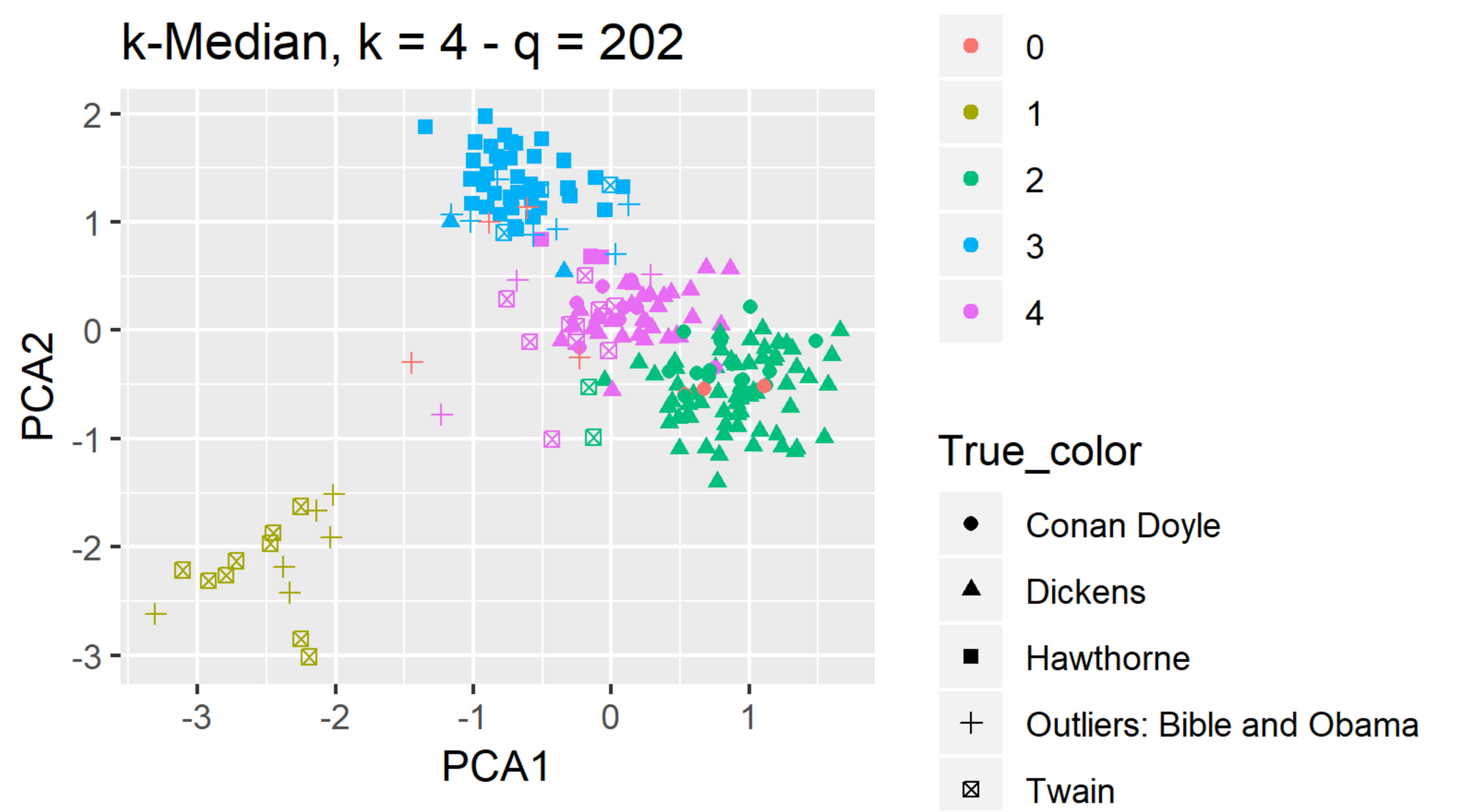}\hfill
\includegraphics[width=0.5\linewidth]{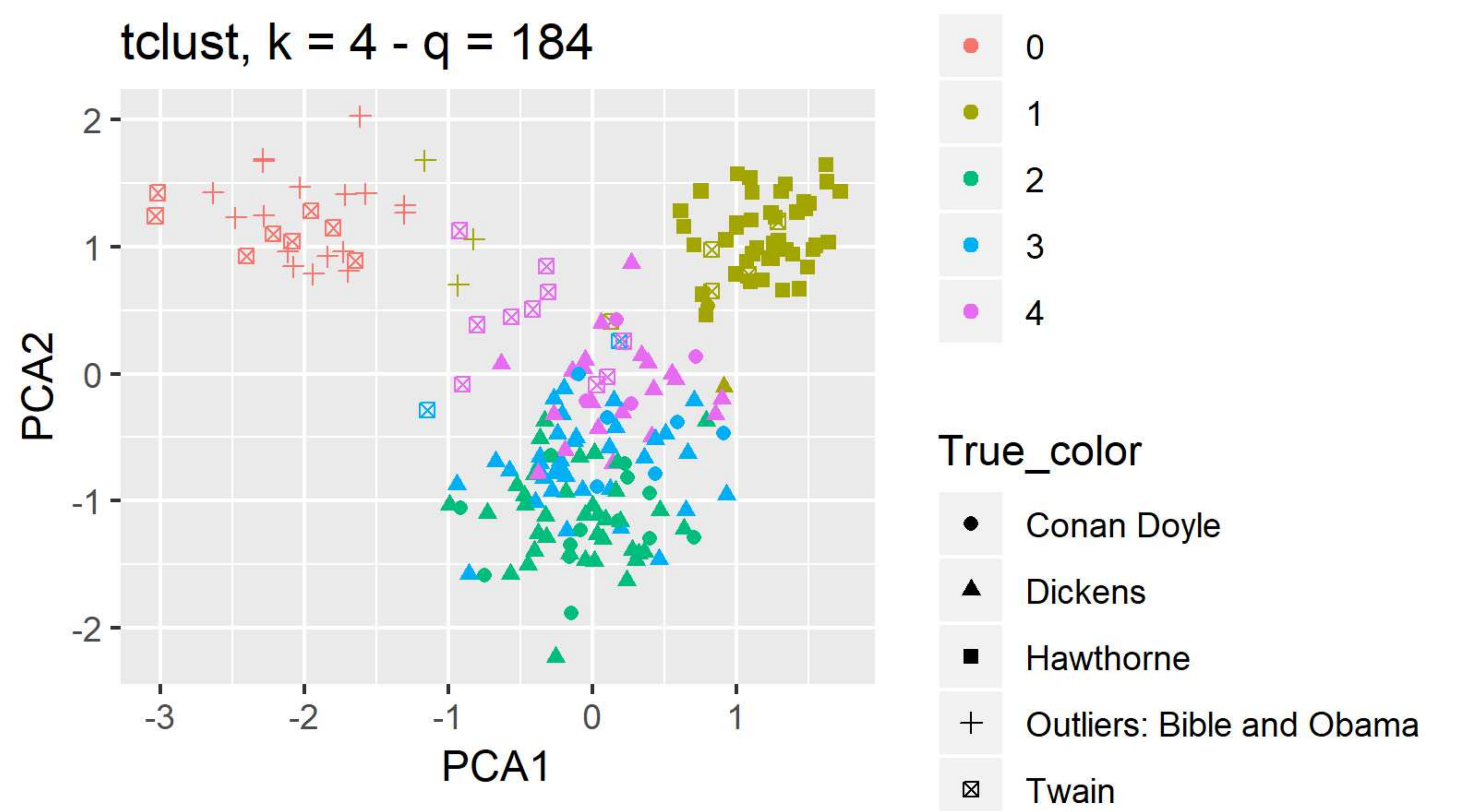}
\caption{Author stylometric clustering with trimmed $k$-median and tclust.}\label{fig:author_clusterings_4_kmedian_tclust}	
\end{figure}

\section{Proofs for Section \ref{sec:clustering_with_bregman_div}}\label{sec:proofs_section_clusteringwithBregmanD}
\subsection{Intermediate results}
The proofs of Theorem \ref{thm:existence_optimal}, Theorem \ref{thm: convergence ps des codebooks} and Theorem \ref{thm:slow_rates} make extensive use of  the following lemmas, whose proofs are deferred to Section \ref{tecsec:inter_results_section_clustering_with_BD}. The first of them is a global upper bound on the radii $r_h(c)$, when $c$ is in a compact subset of $\Omega$.

\begin{lemma}\label{lem:bound_radius_balls}
Assume that $\phi$ is $\mathcal{C}^2$ and $F_0 = \overline{conv(supp(P))}\subset\mathring{\Omega}$. Then, for every $h\in(0,1)$ and $K>0$, there exists $r^+<\infty$ such that 
\begin{align*}
\sup_{c \in F_0 \cap \bar{B}(0,K), s \leq h} r_{s}(c) \leq r^+.
\end{align*}
As a consequence, if $\cb$ is a codebook with a codepoint $c_{j_0} \in F_0$ satisfying $\|c_{j_0}\| \leq K$ and $s\leq h$, then $r_{s}(\cb) \leq r^+$.
\end{lemma}
Next, the following lemma makes connections between the difference of Bregman divergences and distance between codebooks.
      \begin{lemma}\label{lm:eq_normes_compact}
      Assume that $F_0 \subset\mathring{\Omega}$ and $\phi$ is $\mathcal{C}^2$ on $\Omega$.
     Then, for every $K>0$, there exists $C_K >0$ such that for every $\cb$ and $\cb'$ in $\bar{B}(0,K) \cap F_0$, and $x \in \Omega$, 
     \[
     \left | d_{\phi}(x,\cb) - d_{\phi}(x,\cb') \right | \leq C_K D(\cb,\cb') \left ( 1 + \|x\| \right ),
     \]
     where $D(\cb,\cb')=\min_{\sigma\in\Sigma_k}\max_{j\in[\![1,k]\!]}|c_j-c'_{\sigma(j)}|$ (cf. Theorem \ref{thm: convergence ps des codebooks}).
      \end{lemma}
We will also need a continuity result on the function $(s,\cb) \mapsto R_s(\cb)$.
\begin{lemma}
\label{lm: continuite de VPhc}
Assume that $F_0 \subset\mathring{\Omega}$, $P\|u\|<\infty$ and $\phi$ is $\mathcal{C}^2$ on $\Omega$.
Then the map $(s,\cb)\rightarrow R_s(\cb)$ is continuous. Moreover, for every $h\in(0,1)$, $\epsilon>0$ and $K>0$, there is $s_0<h$ such that
\[\forall s_0<s<h,\,\sup_{\cb\in \left(F_0\cap\bar{B}(0,K)\right)^{(k)}}R_h(\cb)-R_s(\cb)\leq\epsilon.\]
\end{lemma}
\subsection{Proof of Lemma \ref{lm: dtm decreasing}}\label{sec:proof_lem_dtm_decreasing}
Set $0<h<h'<1$, and recall that $F_\cb^{-1}(u)=r^2_{u}(\cb)$ denotes the $u$-quantile of the random variable $d_{\phi}(X,\cb)$ for $X\sim P$ and $u\in[0,1]$. Since $F_\cb^{-1}$ is non-decreasing, we may write
\begin{align*}
\frac{R_h(\cb)}{h} =\int_{0}^1F_\cb^{-1}(hu) du
\leq\int_{0}^1F_\cb^{-1}(h'u) d u
=\frac{R_{h'}(\cb)}{h'}.
\end{align*}
Equality holds if and only if $F_\cb^{-1}(hu)=F_\cb^{-1}(h'u)$ for almost all $u\in[0,1]$. Since $F_\cb^{-1}$ is non-decreasing, $L_\cb := \lim_{l'\rightarrow 0}F_\cb^{-1}(l')$ exists. Moreover,  for $l<h'$, $F_\cb^{-1}(hu)=F_\cb^{-1}(h'u)$ a.s., and $F_\cb^{-1}(l)=\lim_{n \rightarrow \infty}F_\cb^{-1} \left ( (h/h')^n l \right ) = L_\cb$, that is, $r^2_{l}(\cb)=\lim_{l'\rightarrow 0}r^2_{l'}(\cb)$. From \eqref{eq: def rayon r}, it follows that $P(B_\phi(\cb,r_{h'}(\cb)))=0$. Conversely, equality holds when $P(B_\phi(\cb,r_{h'}(\cb)))=0$. \qed

\subsection{Proof of Theorem \ref{thm:existence_optimal}}\label{sec:proof_thm_existence_optimal}
The intuition behind the proof of Theorem \ref{thm:existence_optimal} is that optimal codebooks satisfy a so-called centroid condition, namely their code points are means of their trimmed Bregman-Voronoi cells. Thus, provided that optimal Bregman-Voronoi cells have enough weight, the assumption $P\|u\| < + \infty$ leads to a bound on the norm of these code points. This idea is summarized by the following lemma, that is also a key ingredient in the proofs of the results of Section \ref{sec:theoretical_results}.
     \begin{lemma}
\label{lm: key lemma h- h+}
Assume that the requirements of Theorem \ref{thm:existence_optimal} are satisfied. For every $k\geq 2$,
if $R^*_{k-1,h} - R^*_{k,h} >0$, then
\[
\alpha := \min_{j\in[\![2,k]\!]} R^*_{j-1,h} - R^*_{j,h} >0.
\]
Moreover there exist $h^-,h^+\in(0,1)$ with $h\in(h^-,h^+)$ such that, for every $j \in [\![2,k]\!]$, $R^*_{j-1,h^-} - R^*_{j,h^+} \geq \frac{\alpha}{2}$.

For every $b\in(0,h\wedge(1-h)]$, set $h_b^- = (h-b)/(1-b)$ and $h_b^+ = h/(1-b)$. Let $b$ be such that $\min_{j\in[\![2,k]\!]} R^*_{j-1,h_b^-} - R^*_{j,h_b^+} >0$, and, for $\kappa_1$ in $(0,1)$, set $b_1 = \kappa_1b$. Then, for every $s \in [h_{b_1}^-,h_{b_1}^+]$ and $j \in [\![1,k]\!]$, there exists a minimizer $\cb^*_{j,s}$ of $R_{j,s}$ satisfying
\[
\forall p \in [\![1,j]\!], \quad \| c^*_{j,s,p} \| \leq \frac{P \|u\|}{b(1-h)(1-\kappa_1)}. 
\] 
\end{lemma}
The proof of Lemma \ref{lm: key lemma h- h+} is deferred to Section \ref{tecsec:inter_results_section_clustering_with_BD}. When $R^*_{k-1,h} - R^*_{k,h} >0$, Theorem \ref{thm:existence_optimal} follows from Lemma \ref{lm: key lemma h- h+}. In the case where $R^*_{k-1,h} - R^*_{k,h} =0$, there exists a set $A$ with $P(A) \geq h$ such that the restriction of $P$ to $A$ is supported by at most $k-1$ points. These $k-1$ points provide an optimal $k$-points codebook. Hence the result of Theorem \ref{thm:existence_optimal}. \qed

\subsection{Proof of Proposition \ref{prop:properties_discernability_factor}}\label{sec:proof_prop_properties_discernability_factor}
The first part of Proposition \ref{prop:properties_discernability_factor} follows from Lemma \ref{lm: key lemma h- h+}. Indeed, if $h^-$ and $h^+$ are such that $\min_{j\in[\![2,k]\!]} R^*_{j-1,h^-} - R^*_{j,h^+} >0$, then for $b$ small enough so that $h^- \leq h_b^- < h < h_b^+ \leq h^+$, we have $\min_{j\in[\![2,k]\!]} R^*_{j-1,h_b^-} - R^*_{j,h_b^+} \geq \min_{j\in[\![2,k]\!]} R^*_{j-1,h^-} - R^*_{j,h^+} >0$. 

We turn to the second part of Proposition \ref{prop:properties_discernability_factor}. Let $\cb^{*,(j)}$ be a $j$-points $h$-trimmed optimal codebook, and $p_{j,h} = h \min_{l\in[\![1,j]\!]} \tilde{P}_{\cb^{*,(j)}} \left ( W_l(\cb^{*,(j)}) \right )$, where $\tilde{P}_{\cb^{*,(j)}} \in \mathcal{P}_h(\cb^{*,(j)})$. Let $\tau_{j,h}$ denote the $[0,1]$-valued function such that $h\tilde{P}_{\cb^{*,(j)}} = P \tau_{j,h}$. Assume that $p_{j,h} = P\tau_{j,h}(u) \mathbbm{1}_{W_1(\cb^{*,(j)})}(u)$, without loss of generality. Then we have
\begin{align*}
R^*_{j,h} \geq  \sum_{l=2}^k P d_\phi(u,c^{*,(j)}_l)(u) \tau_{j,h}(u) \mathbbm{1}_{W_l(\cb^{*,(j)})}(u) 
               \geq R^*_{j-1,h-p_{j,h}}. 
\end{align*}   
Thus, $(h-B_{h})/(1-B_h) \geq h-p_{j,h}$, that entails $(1-(h-p_{j,h}))B_h \leq p_{j,h}$.
\section{Proofs for Section \ref{sec:theoretical_results}}\label{sec:proofs_section_theoretical_results}
\subsection{Intermediate results}

      Theorem \ref{thm: convergence ps des codebooks} and \ref{thm:slow_rates} require some additional probabilistic results that are gathered in this subsection. Some of them are applications of standard techniques, their proofs are thus deferred to Section \ref{tecsec:proofs_inter_theoretical_results}, for the sake of completeness. We begin with  deviation bounds.

\begin{proposition}\label{prop:deviations}
With probability larger than $1-e^{-x}$, we have, for $k \geq 2$,
\begin{align}\label{eq:deviationVCballs}
\sup_{\cb \in (\R^d)^{(k)}, r \geq 0} \left | (P - P_n) \mathbbm{1}_{B_\phi(\cb,r)} \right | & \leq C \sqrt{\frac{k(d+1) \log(k)}{n}} + \sqrt{\frac{2x}{n}}, \\
\sup_{\cb \in (\R^d)^{(k)}, r \geq 0} \left | (P - P_n) \mathbbm{1}_{\partial B_\phi(\cb,r)} \right | & \leq C \sqrt{\frac{k(d+1) \log(k)}{n}} + \sqrt{\frac{2x}{n}},\notag
\end{align}
where $\partial B_\phi(\cb,r)$ denotes $\left \{ x \mid d_\phi(x,\cb)=r^2 \right \}$ and $C$ denotes a universal constant. Moreover, if $r^+$ and $K$ are fixed, and $P \|u\|^2 \leq M_2^2<\infty$, we have, with probability larger than $1-e^{-x}$,
\begin{align}\label{eq:deviationcontrastfunction}
\sup_{\cb \in \left(\bar{B}(0,K)\cap F_0\right)^{(k)}, r \leq r^+}\left | (P - P_n) d_{\phi}(.,\cb) \mathbbm{1}_{B_\phi(\cb,r)} \right |  & \leq (r^+)^2 \left [C_{K,r^+,M_2} \frac{\sqrt{kd \log(k)}}{\sqrt{n}} +  \sqrt{\frac{2x}{n}} \right ], \\
\sup_{\cb \in \left(\bar{B}(0,K)\cap F_0\right)^{(k)}, r \leq r^+}\left | (P - P_n) d_{\phi}(.,\cb) \mathbbm{1}_{\partial B_\phi(\cb,r)} \right | & \leq (r^+)^2 \left [C_{K,r^+,M_2} \frac{\sqrt{kd\log(k)}}{\sqrt{n}} +  \sqrt{\frac{2x}{n}} \right ], \notag
\end{align}
where we recall that $\overline{conv(supp(P))}=F_0 \subset \mathring{\Omega}$.
\end{proposition}
A key intermediate result shows that, on the probability events defined above, empirical risk minimizers must have bounded codepoints.
\begin{proposition}\label{prop:bounded_empirical_codebooks}
Assume that $P\|u\|^p < + \infty $ for some $p \geq 2$, and let $b>0$ be such that $\min_{j\in[\![2,k]\!]} R^*_{j-1,h_b^-}-R^*_{j,h_b^+} >0$, where $h_b^- = (h-b)/(1-b)$, $h^+_b=h/(1-b)$, as in Lemma \ref{lm: key lemma h- h+}. Let $\kappa_2<1$, and denote by $b_2 = \kappa_2 b$.
 Then there exists $C_{P,h,k,\kappa_2,b}$ such that, for $n$ large enough, with probability larger than $1-n^{-\frac{p}{2}}$, we have, for all $j\in[\![2,k]\!]$, and $i\in[\![1,j]\!]$, 
\begin{align*}
\sup_{ h_{b_2^-} \leq s \leq h} \|\hat{c}_{j,s,i}\| \leq C_{P,h,k,\kappa_2,b},
\end{align*}
where $\hat{\cb}_{j,s}$ denotes a $j$-points empirical risk minimizer with trimming level $s$.
\end{proposition}

To prove Theorem \ref{thm: convergence ps des codebooks}, a more involved version of Markov's inequality is needed, stated below.
\begin{lemma}
\label{lm: Marcinkiewicz}
If $P\|u\|^p<\infty$ for some $p\geq 2$, then there exists some positive constant $C$ such that with probability larger than $1-n^{-\frac{p}{2}}$, $P_n\|u\|\leq C$.
\end{lemma}
At last, a technical lemma on empirical quantiles of  Bregman divergences will be needed.
\begin{lemma}
\label{lemme: r_borne}
Let $(P_n)_{n\in\N}$ be a sequence of probabilities that converges weakly to a distribution $P$. Assume that $supp(P_n)\subset supp(P)\subset\R^d$, $F_0 = \overline{conv(supp(P))}\subset\mathring{\Omega}$ and $\phi$ is $\Ccal_2$ on $\Omega$. Then, for every $h\in(0,1)$ and $K>0$, there exists $K_+>0$ such that for every $\cb\in\Omega^{(k)}$ satisfying $|c_i|\leq K$ for some $i\in[\![1,k]\!]$ and every $n\in\N$,
\[r_{n,h}(\cb)\leq r_+ = \sqrt{4(2K+K_+)\sup_{c\in F_0\cap\bar{B}(0,2K+K_+)}\|\nabla_c\phi\|}.\]
\end{lemma}
   \subsection{Proof of Theorem \ref{thm: convergence ps des codebooks}}\label{sec:proof_thm_convergence ps des codebooks}  
The proof of Theorem \ref{thm: convergence ps des codebooks} is an adaptation of the proof of \cite[Theorem 3.4]{Cuesta97}. First note that since $\phi$ is strictly convex and continuous, $\psi:x\mapsto \phi(x)-\langle x,a\rangle + b$ is also strictly convex and continuous, for every $a$, $b$. Thus $\psi^{-1}(\{0\})$ is a closed set. 
Moreover, since $\psi$ is strictly convex, any line that contains $0$ contains at most two points of $\psi^{-1}(\{0\})$. Thus, the Lebesgue measure of $\psi^{-1}(\{0\})$ is $0$. Since $P$ is absolutely continuous, it follows that boundaries of Bregman balls have $P$-mass equal to $0$.

According to Proposition \ref{prop:bounded_empirical_codebooks}, provided that $P\|u\|^p < + \infty$, for some $p>2$, 
there exists $C_P>0$ such that for some $N\in\N$, $\sum_{n\geq N}P(\max_{i\in[\![1,k]\!]}\left\|\hat c_{n,h,i}\right\|>C_P)<\infty$. Thus, 
the Borel-Cantelli Lemma ensures that, a.s. for $n$ large enough, for every $i\in[\![1,k]\!]$, $\left\|\hat c_{n,h,i}\right \|\leq C_P$.
According to the Skorokhod's representation theorem in the Polish space $\R^d$, there exists a measured space $(\tilde\Omega,\tilde{\mathcal{F}},\tilde P)$ and a sequence of random variables $(X_n)_{n\in\N}$ along with a random variable $X$ on $(\tilde\Omega,\tilde{\mathcal{F}},\tilde P)$ such that $X_n\sim P_n$, $X\sim P$ and $X_n$ converges to $X$ $\tilde P$-a.s.

Denote by $\cb^*$ a minimizer of $\cb\mapsto\V^P_{\phi,h}(\cb)$, $r'_n=r_{n,h}(\cb^*)$ and $\tau'_n$ a $[0,1]$-valued measurable function such that $hP_{n,\cb^*,h}=P_n\tau'_n$, that is, such that $P_n\tau'_n(u)=h$ and
\[\1_{B_\phi(\cb^*,r'_n)}\leq\tau'_n\leq\1_{\bar{B}_\phi(\cb^*,r'_n)}.\]

According to Lemma \ref{lemme: r_borne}, with $K=\|c^*_1\|$ for instance, it comes $r'_n\leq r^+$, for some finite $r^+$.
Thus, up to extracting a subsequence, we may assume that $r'_n\rightarrow r'_0$ for some $r'_0\leq r_+$.
Moreover, it holds
\begin{multline}
\label{eq: borne des DFI}
\left|d_{\phi}(X_n,\cb^*)-d_{\phi}(X,\cb^*)\right|\leq |\phi(X_n)-\phi(X)| 
 +\max_{j\in[\![1,k]\!]}\|\nabla_{c^*_j}\phi\||X_n-X|.
\end{multline}
Thus, $d_{\phi}(X_n,\cb^*) \rightarrow d_{\phi}(X,\cb^*)$ a.e. when $n\rightarrow\infty$.
As a consequence, $\tau'_n(X_n)\rightarrow\1_{B_\phi(\cb^*,r'_0)}(X)$ $\tilde P$-a.e.
The dominated convergence theorem yields $h=P_n\tau'_n(u)\rightarrow P(B_\phi(\cb^*,r'_0))$. Thus, $\1_{B_\phi(\cb^*,r'_0)} = \tau_0$ $P$-a.e where $\tau_0$ denotes the trimming set associated with $\cb^*$ and $P$.
Moreover, since $\tau'_n(X_n)d_{\phi}(X_n,\cb^*)$ is bounded by $r_+$ and converges to $\tau_0(X)d_{\phi}(X,\cb^*)$ a.e., the dominated convergence theorem entails 
\[
R_{n,h}(\hat{\cb}_n) \leq R_{n,h}(\cb^*)\leq \E\left[\tau'_n(X_n)d_{\phi}(X_n,\cb^*)\right]\rightarrow \E\left[\tau_0(X)d_{\phi}(X,\cb^*)\right].
\]
Thus, $\limsup_{n\rightarrow\infty} R_{n,h}(\hat{\cb}_n) \leq R_{k,h}^*$.

Since, for $n\geq N$ and every $i\in[\![1,k]\!]$, $\|\hat c_{n,i}\|\leq C_P$, we have $\hat c_{u(n),i}\rightarrow c_i$ for some $c_i\in F_0\cap\bar{B}(0,C_P)$, where $\hat{\cb}_{u(n)}$ is a subsequence. Set $\cb=(c_1,c_2,\dots,c_k)$. Again, according to Lemma \ref{lemme: r_borne} with $K=C_P$, it comes that $r_{u(n),h}(\hat\cb_{u(n)})\rightarrow r$ for some $r\geq 0$.  Therefore, from \eqref{eq: borne des DFI}, Lemma \ref{lm:eq_normes_compact} and the continuity of $P$,
\[\lim_{n\rightarrow\infty}\tau_{u(n)}(X_{u(n)}) = \1_{B_\phi(\cb,r)}(X) \quad a.e.,\]
where $\tau_{u(n)} = \mathbbm{1}_{B_\phi(\hat{\cb}_{u(n)},r_{u(n),h}(\hat{\cb}_{u(n)}))}$. 
According to the dominated convergence theorem, we have $h = P(B_\phi(\cb,r)) = P_{u(n)}(\tau_{u(n)}(u))$.
Again, the dominated convergence theorem implies that
\[
\lim\inf_{n\rightarrow\infty} R_{u(n),h}(\hat{\cb}_{u(n)}) \geq P \1_{B_\phi(\cb,r)}(u)d_\phi(u,\cb) = R_{h}(\cb) \geq R^*_{k,h}.
\]

As a consequence, $\lim_{n \rightarrow \infty} R_{u(n),h}(\hat{\cb}_{u(n)}) = R_{h}(\cb) = R_{k,h}^*$ and $\cb$ is an optimal trimmed codebook.
Since, given a subsequence of $(\hat\cb_n)_{n\in\N}$, we may find a subsequence of indices $u(n)$ such that $\lim_{n \rightarrow \infty} R_{u(n),h}(\hat{\cb}_{u(n)}) =  R_{k,h}^*$, we deduce that $\lim_{n \rightarrow + \infty} R_{n,h}(\hat{\cb}_n) = R^*_{k,h}$. 
 
 Now assume that $\cb^*_h$ is unique. Then, for every subsequence of $(\hat\cb_n)_{n\in\N}$, there exists  $u(n)$ such that $\hat{\cb}_{u(n)} \rightarrow \cb = \cb^*_h$. Thus, a.e., $\hat{\cb}_n \rightarrow \cb^*_h$. \qed

\subsection{Proof of Theorem \ref{thm:slow_rates}}\label{sec:proof_thm_slow_rates}
For $h>0$ and a codebook $\cb$, we denote by $\tau_h(\cb)$ the trimming function $\1_{B_\phi(\cb,r_h(\cb))} + \delta_h(\cb)\1_{\partial B_\phi(\cb,r_h(\cb))}$, so that $P \tau_h(\cb)/h \in \mathcal{P}_h(\cb)$. We also denote by $\hat{\tau}_h(\cb)$ its empirical counterpart. Note that $\delta_h(\cb)$ and $\hat{\delta}_h(\cb)$ are smaller than $1$. It follows that
\begin{align*}
|P \tau_h(\cb) - P \hat{\tau}_h(\cb) | & = |(P-P_n) \hat{\tau}_h(\cb)| \\ & \leq |(P-P_n)B_\phi(\cb,r_{n,h}(\cb))| + \hat{\delta}(\cb)|(P-P_n)\partial B_\phi(\cb,r_{n,h}(\cb))| \\
 & \leq |(P-P_n)B_\phi(\cb,r_{n,h}(\cb))| + |(P-P_n)\partial B_\phi(\cb,r_{n,h}(\cb))|.
\end{align*}
As well, we bound $|P_n \tau_h(\cb) - P_n \hat{\tau}_h(\cb) |$ the same way. Combining Lemma \ref{lm: key lemma h- h+} and Proposition \ref{prop:bounded_empirical_codebooks}, we consider a probability event onto which, for all $j$, $\| \hat{\cb}_{n,j} \| \leq C_P$ and $\sup_{\cb \in (F_0 \cap \bar{B}(0,C_P))^{(k)}} r_{n,h}(\cb) \vee r_{h}(\cb)\leq r^+$. This occurs with probability at least $1 - n^{-p/2}$ for $C_P$ and $r^+$ large enough (more details are given in  Section \ref{tecsec:Proof_Proposition_bounded_empirical_codebook}). We also assume that the deviation bounds of Proposition \ref{prop:deviations} hold, with parameter $C_P$ and $r_+$, to define a global probability event with mass larger than $1-n^{-p/2}-2e^{-x}$. On this event, we have
\begin{align*}
R_h(\hat\cb_n)-R^*_{k,h}  & = \begin{multlined}[t] P d_{\phi}(u,\hat{\cb}_n) \hat{\tau}_h(\hat{\cb}_n) - P d_{\phi}(u,\cb^*) \hat{\tau}_h(\cb^*) 
 \\ + P d_{\phi}(u,\hat{\cb}_n) (\tau_h(\hat{\cb}_n) - \hat{\tau}_h(\hat{\cb}_n)  - (P d_{\phi}(u,\cb^*) (\tau_h(\cb^*) - \hat{\tau}_h(\cb^*)) 
 \end{multlined} \\
& \leq 2 \sup_{\cb \in (F_0 \cap \bar{B}(0,C_P))^{(k)}, r \leq r^+} |(P-P_n) d_{\phi}(u,\cb)\1_{B_\phi(\cb,r)}(u)| \\
 & \quad + 2 (r^+)^2 \sup_{ \cb \in \Omega^{(k)}, r \geq 0} |(P-P_n) B_{\phi}(\cb,r)| \\
 & \quad + 2 \sup_{\cb \in (F_0 \cap \bar{B}(0,C_P))^{(k)}, r \leq r^+} |(P-P_n) d_{\phi}(u,\cb)\1_{\partial B_\phi(\cb,r)}(u)| \\
 & \quad + 2 (r^+)^2 \sup_{ \cb \in \Omega^{(k)}, r \geq 0} |(P-P_n) \partial B_{\phi}(\cb,r)|. 
\end{align*}  
Therefore, $R_h(\hat\cb_n)-R^*_{k,h} \leq {C_P(1 + \sqrt{x})}/{\sqrt{n}}$,  for some constant $C_P$. \qed

\subsection{Proof of Corollary \ref{cor:slow_rates_expectation}}\label{sec:proof_cor_slow_rates_expectation}
Denote by $A$ the intersection of the probability events described in Proposition \ref{prop:bounded_empirical_codebooks}, that has probability larger than $1-n^{-\frac{p}{2}}$. Decomposing the excess risk as in the proof of Theorem \ref{thm:slow_rates} yields
\begin{align*}
R_h(\hat{\cb}_n) - R^*_{k,h} & = (R_h(\hat{\cb}_n) - R^*_{k,h}) \mathbbm{1}_{A} + (R_h(\hat{\cb}_n) - R^*_{k,h})\mathbbm{1}_{A^c}.
\end{align*}
According to Proposition \ref{prop:deviations}, we have $\mathbb{E} ((R_h(\hat{\cb}_n) - R^*_{k,h}) \mathbbm{1}_{A}) \leq C_P/\sqrt{n}$.
It only remains to bound the expectation of the second term. This is the aim of the following Lemma, whose proof is deferred to Section \ref{tecsec:proof_lem_expected_distortion_tail}.
\begin{lemma}\label{lem:expected_distortion_tail}
Assume that $P\|u\|^q \psi^q(k\|u\|/h) < \infty$. Then there exists a constant $C_q$ such that $\E R_h^q(\hat{\cb}_n) \leq C_P^q$. 
\end{lemma}
Equipped with Lemma \ref{lem:expected_distortion_tail}, we may bound $\E(R_h(\hat{\cb}_n)\1_{A^c})$ as follows, using Hölder's inequality,
\begin{align*}
\E(R_h(\hat{\cb}_n)\1_{A^c}) & \leq \left (\mathbb{P}(A^c) \right )^{\frac{1}{p}} \left (\E R_h^q(\hat{\cb}_n) \right )^\frac{1}{q} \leq C_P/{\sqrt{n}}.
\end{align*}

\subsection{Proof of Theorem \ref{thm:FSBP_2}}\label{sec:proof_thm_FSBP2}
A key ingredient of the proof of Theorem \ref{thm:FSBP_2} is the following lemma, ensuring that every cell of a trimmed and corrupted empirical distortion minimizer contains a minimal portion of signal points. In what follows, the $\hat{\tau}$'s are the trimming function with respect to $P_n$ (uncorrupted sample), as defined in Section \ref{sec:proof_thm_slow_rates}.  
\begin{lemma}\label{lem:enough_weight_optimalcells}
Assume that $B_h>0$ (see Definition \ref{def:discernability_factor}), let $b < B_h$ and $b<b_1<B_h$ such that $b = \kappa_1 b_1$, with $\kappa_1 <1$. Denote by $\beta_1 = (1-\kappa_1)b_1\left [h \wedge (1-h)\right ]/2$. Assume that $s/(n+s) \leq b$. Then, for $n$ large enough, with probability larger than $1-n^{-\frac{p}{2}}$, we have, for all $j\in[\![1,k]\!]$, 
\[
P_n \left ( \hat{\tau}_{h_b^-}(\hat{\cb}_{n+s,h})\1_{W_j(\hat{\cb}_{n+s,h})} \right ) \geq \beta_1.
\]
\end{lemma}
The proof of Lemma \ref{lem:enough_weight_optimalcells} is postponed to Section \ref{tecsec:proof_lem_enough_weight_optimalcells}. We are now in a position to prove Theorem \ref{thm:FSBP_2}.
\begin{proof}[Proof of Theorem \ref{thm:FSBP_2}]
We adopt the same notation and assumptions as in the proof of Lemma \ref{lem:enough_weight_optimalcells}. We may write
\begin{align*}
(n+s)\hat{R}_{n+s,h} (\hat{\cb}_{n+s,h}) & \leq n \left ( R^*_{h_b^+ + \beta_n} + \alpha_n \right ).
\end{align*}
On the other hand, recall that for all $j\in[\![1,k]\!]$, 
$
P_n \left ( \hat{\tau}_{h_b^-}(\hat{\cb}_{n+s,h})\1_{W_j(\hat{\cb}_{n+s,h})} \right ) \geq \beta_1
$,
and denote by $\mathbf{m}=(m_1, \hdots, m_k)$ the codebook such that
\[
m_j = \left [{P_n ( u \hat{\tau}_{h_b^-}(\hat{\cb}_{n+s,h})\1_{W_j(\hat{\cb}_{n+s,h})} )}\right ] / \left [{P_n \left ( \hat{\tau}_{h_b^-}(\hat{\cb}_{n+s,h})\1_{W_j(\hat{\cb}_{n+s,h})} \right )} \right ].
\]
Then, for all $j$, $\|m_j\| \leq {P_n \|u\|}/{\beta_1} \leq C_P$. Using Proposition \ref{prop:bregman_bias_variance}, we may write
\begin{align*}
\hat{R}_{n+s,h} (\hat{\cb}_{n+s,h}) & \geq n\hat{R}_{n,h_b^-}(\hat{\cb}_{n+s,h})/(n+s) \\ 
& \geq \begin{multlined}[t] \frac{n}{n+s} \left [ \sum_{j=1}^{k} P_n \left ( \hat{\tau}_{h_b^-}(\hat{\cb}_{n+s,h})\1_{W_j(\hat{\cb}_{n+s,h})} \right ) d_\phi(m_j,\hat{c}_{n+s,h,j}) \right . \\
 + \left . \vphantom{\sum_{j=1}^{k}} \hat{R}_{n,h^-_b}(\mathbf{m}) \right ].
 \end{multlined}
\end{align*} 
The last term satisfies 
\begin{align*}
\hat{R}_{n,h^-_b}(\mathbf{m})  & \geq P_n d_\phi(u,\mathbf{m})\hat{\tau}_{h_b^-}(\mathbf{m})  \geq P_n d_\phi(u,\mathbf{m})\tau_{h_b^- -\beta_n}(\mathbf{m})  \\
                            &  \quad \geq P d_\phi(u,\mathbf{m})\tau_{h_b^- -\beta_n}(\mathbf{m}) - \alpha_n \geq R^*_{h_b^- - \beta_n} - \alpha_n. 
\end{align*}
Thus, for all $j\in[\![1,k]\!]$, 
\begin{align}\label{eq:PSBP_bregman_closeness}
\beta_1 d_\phi(m_j,\hat{c}_{n+s,h,j}) & \leq \left [ 2 \alpha_n + R^*_{h_b^+ + \beta_n} - R^*_{h_b^- - \beta_n} \right ].
\end{align}
\end{proof}

\subsection{Proof of Corollary \ref{cor:FSBP_3}}\label{sec:proof_cor_FSBP3}
With the same setting as Theorem \ref{thm:FSBP_2}, according to \eqref{eq:PSBP_bregman_closeness}, we have, almost surely, for $n$ large enough
$
d_\phi(K,\hat{\cb}_{n+s,j}) \leq   2[ 2 \alpha_n + R^*_{h_b^+ + \beta_n} - R^*_{h_b^- - \beta_n} ]/[b(1-\kappa_1)(h \wedge (1-h)]
$,
whenever $ s/(n+s) \leq b$ and $B_h \kappa_1 > b$. Now let  $c$ be defined as
\[
c=\sup \{ r >0 \mid \{x \mid d_\phi(K,x) \leq r \} \subset B(K,1) \}.
\]
If $c>0$, then requiring $b$ small enough and $s/(n+s) \leq b$ ensures that $d_\phi(K,\hat{\cb}_{n+s,j}) \leq c/2$, almost surely, for $n$ large enough, hence the result. 
 
Thus, it remains to prove that $c >0$. Assume that for every $r >0$, $\{x \mid d_\phi(K,x) \leq r \} \nsubseteq B(K,1)$. Then there exists $x_0 \in K$ and a sequence $v_n$  satisfying $d_\phi(x_0,v_n) \rightarrow 0$, along with $\|x_0 - v_n\| > 1$. Noting that, for $t\geq 0$, $d_{\phi}(x_0 ,v_n + t(v_n-x_0)) \geq d_{\phi}(x_0 ,v_n )$ yields $d_\phi(x_0,v_n) \geq d_\phi(x_0,v'_n)$, where $v'_n = x_0 + (v_n - x_0)/\|v_n-x_0\|$. Therefore, $d_\phi(v_n,x_0) \geq \inf_{\|u-x_0\|=1} d_\phi(x_0,u) >0$, hence the contradiction. 

At last, if $c\mapsto d_\phi(x,c)$ is a proper map, then $
d_\phi(K,\hat{\cb}_{n+s,j}) \leq   2[ 2 \alpha_n + R^*_{h_b^+ + \beta_n} - R^*_{h_b^- - \beta_n} ]/[b(1-\kappa_1)(h \wedge (1-h)]
$,
whenever $ s/(n+s) \leq b$ and $B_h \kappa_1 > b$ entails that, almost surely, for $n$ large enough, $\|\hat{\cb}_{n+s}\|_2 < + \infty$, thus $\widehat{BP}_{n,h} > b$. Hence $\widehat{BP}_{n,h} \geq B_h$, almost surely for $n$ large enough.   
\qed

\printbibliography


\section{Technical proofs for Section \ref{sec:clustering_with_bregman_div}}\label{tecsec:proofs_definition_DTM}
\subsection{Proof of Lemma \ref{lm: ecriture DTM}}\label{tecsec_proof_lemma_ecriture_DTM}
\begin{lemma*}[\ref{lm: ecriture DTM}]
For all $\cb\in\Omega^{(k)}$, $h\in(0,1]$, $\tilde P\in \mathcal{P}_h$ and $\tilde{P}_\cb \in \mathcal{P}_h(\cb)$,
\[R(\tilde{P}_\cb,\cb) \leq R(\tilde{P},\cb).\]
Equality holds if and only if $\tilde P\in \mathcal{P}_h(\cb)$.
\end{lemma*}
For $u\in[0,1]$, let $F_\cb^{-1}(u)=r_{u}^2(\cb)$ denote the $u$-quantile of the random variable $d_{\phi}(X,\cb)$ for $X\sim P$. That is,
\begin{align*}
F_\cb^{-1}(u)&=\inf\left\{s\geq 0\mid  \mbox{ with probability }\geq u, d_\phi(X,\cb)\leq s \right\}\\
&=\inf\left\{r^2\geq 0\mid P(\bar{B}_\phi(\cb,r))\geq u\right\}.
\end{align*}
If $\tilde F_\cb^{*\ -1}(u)$ denotes the $u$-quantile of $d_{\phi}(\tilde X^*,\cb)$, for $\tilde X^*\sim\tilde P_{\cb}\in \mathcal{P}_{h}(\cb)$, then $\tilde F_\cb^{*\ -1}(u)=F_\cb^{-1}(hu)$.
Let $U$ be a random variable uniform on $[0,1]$, then $\tilde F_\cb^{*\ -1}(U)$ and $d_{\phi}(\tilde X^*,\cb)$ have the same distribution. Thus, we may write:
\begin{align*}
R(\tilde{P}_\cb,\cb)&=\E_{\tilde X^*}d_{\phi}(\tilde X^*,\cb)=\int_{0}^1F_\cb^{-1}(hu) du.
\end{align*}
Let $\tilde P\in \mathcal{P}_h(P)$ be a Borel probability measure on $\Omega$ such that $h\tilde P$ is a sub-measure of $P$, and let $\tilde F_\cb^{-1}(u)$ denote the $u$-quantile of $d_{\phi}(\tilde X,\cb)$ for $\tilde X\sim\tilde P$. Since $P(B_\phi(\cb,u))\geq h\tilde P(B_\phi(\cb,u))$, it holds that $\tilde F_\cb^{-1}(u)\geq F_\cb^{-1}(hu)$ . Thus, we may write 
\begin{align*}
R(\tilde{P}, \cb) & = \int_{0}^1\tilde F_\cb^{-1}(u) du \geq R(\tilde{P}_\cb,\cb).
\end{align*}
Note that equality holds if and only if $\tilde F_\cb^{-1}(u)=\tilde F_\cb^{*\ -1}(u)$ for almost all $u\in[0,1]$, that is $\tilde{P} \in \mathcal{P}_h(\cb)$. \qed
\section{Technical proofs for Section \ref{sec:theoretical_results}}\label{tecsec_proofs_sec_theoretical_results}
\subsection{Proofs for Example \ref{Ex:FSBP}}\label{tecsec:proof_ex_FSBP}
\begin{example*}[\ref{Ex:FSBP}]
Let $\phi_1=\|.\|^2$, $\phi_2 = \exp(-.)$, $\Omega = \R$, $P = (1-p) \delta_{-1} +  p\delta_{1}$, with $p\leq 1/2$. Then, for $\phi = \phi_j$, $j \in \{1,2\}$, $k=2$ and $h > (1-p)$, we have $B_h = \frac{h+p-1}{p} \wedge (1-h)$. Let $Q_{\gamma,N} = (1-\gamma)P + \gamma \delta_{N}$. The following holds.
\begin{itemize}
\item If $(1+p)h>1$, $B_h = 1-h$, and for every $\gamma > 1-h$, any sequence of optimal $2$-points $h$-trimmed codebook $\cb^{*}_2(Q_{\gamma,N})$ for $Q_{\gamma,N}$ satisfies 
\[
\lim_{ N \rightarrow + \infty} \|\cb^{*}_2 (Q_{\gamma,N})\| = + \infty.
\]
\item If $(1+p)h \leq 1$, then $B_h = \frac{h+p-1}{p}$, and, for $\gamma=B_h$, $(-1,N)$ is an optimal $2$-points $h$-trimmed codebook for $Q_{\gamma,N}$.
\end{itemize}
\end{example*}
We have, for $P$ and any $s \in [0,1]$, $R^*_{2,s}=0$, and $R^*_{1,s} > 0$ if and only if $s >(1-p)$. Thus, for any $h>(1-p)$ and $b \leq (1-h)$, $R^*_{1,h_b^-} > R^*_{2,h_b^+}$ if and only if $h_{b}^- > (1-p)$ that is equivalent to $b < (h+p-1)/p$. We deduce that $B_h = \frac{h+p-1}{p}\wedge (1-h)$. 

Assume that $(h+p-1)/p \leq (1-h)$. Then $B_h = (h+p-1)/p$. For $\gamma = B_h$ and $\cb=(-1,N)$, we have that $Q_{\gamma,N}(\{-1,N\}) = (1-\gamma)(1-p) + \gamma=h$, hence $R_{h,\gamma,N}(\cb)=0$, where $R_{h,\gamma,N}$ denotes the $h$-trimmed distortion with respect to $Q_{\gamma,N}$. 

Now assume that $(h+p-1)/p>(1-h)$. Let $\gamma >(1-h)$. Then, for any $\cb \in \R^{(2)}$, if $\tau_h(\cb)$ is such that $Q_{\gamma,N} \tau_h(\cb)$ is a submeasure of $Q_{\gamma,N}$ with total mass $h$, then $Q_{\gamma,N} \tau_{h}(\cb)(u) \1_{\{N\}}(u) \geq \gamma - (1-h)$. Now, if $\cb^*_{2}(Q_{\gamma,N})$ is an optimal two-points quantizer for $Q_{\gamma,N}$, Proposition \ref{prop:centroid} ensures that for every $j \in \{1,2\}$, 
\[
c^*_{2,j}(Q_{\gamma,N}) = \frac{P u \tau_h(\cb^*_{2}(Q_{\gamma,N}))(u) \1_{W_j(\cb^*_{2}(Q_{\gamma,N}))}(u)}{P \tau_h(\cb^*_{2}(Q_{\gamma,N}))(u) \1_{W_j(\cb^*_{2}(Q_{\gamma,N}))}(u)}.
\]
Thus, we may assume that $Q_{\gamma,N}\tau_h(\cb^*_{2}(Q_{\gamma,N}))(u) \1_{W_2(\cb^*_{2}(Q_{\gamma,N}))}(u) \1_{\{N\}}(u)  \geq \gamma - (1-h)$, without loss of generality. Hence, for $N$ large enough,
\[
c^*_{2,2}(Q_{\gamma,N}) \geq {-(1-p) + N(\gamma - (1-h))} \underset{N \rightarrow \infty}{\longrightarrow} + \infty.
\]
\qed

\section{Technical proofs for Section \ref{sec:proofs_section_clusteringwithBregmanD}}\label{tecsec:inter_results_section_clustering_with_BD}
\subsection{Proof of Lemma \ref{lem:bound_radius_balls}}
\begin{lemma*}[\ref{lem:bound_radius_balls}]
Assume that $\phi$ is $\mathcal{C}^2$ and $F_0 = \overline{conv(supp(P))}\subset\mathring{\Omega}$. Then, for every $h\in(0,1)$ and $K>0$, there exists $r^+<\infty$ such that 
\begin{align*}
\sup_{c \in F_0 \cap \bar{B}(0,K), s \leq h} r_{s}(c) \leq r^+.
\end{align*}
As a consequence, if $\cb$ is a codebook with a codepoint $c_{j_0} \in F_0$ satisfying $\|c_{j_0}\| \leq K$ and $s\leq h$, then $r_{s}(\cb) \leq r^+$.
\end{lemma*}
\begin{proof}[Proof of Lemma \ref{lem:bound_radius_balls}]
Let $K_+$ be such that $P(B(0,K_+)) > h$.
Thus, if $c \in \bar{B}(0,K)$, $P(B(c,K+K_+)) > h$.
Since $B(c,K+K_+) \subset B(0,2K + K_+)$, and $\phi$ is $\mathcal{C}^2$, according to the mean value theorem, there exists $C_+$ such that, 
for all $x,\,y\in B(c,K+K_+) \cap F_0$,  $d_{\phi}(x,y) \leq C_+ \|x-y\|$.
Therefore, for every $c \in \bar{B}(0,K)$, $P\left(B_\phi\left(c,\sqrt{C_+(2K+K_+)}\right)\right) > h$.
Hence $r_{s}(c) \leq r_h(c) \leq \sqrt{C_+(2K+K_+)} = r^+$.

At last, if $\cb$ is such that $c_{j_0} \in \bar{B}(0,K) \cap F_0$, then $\bar{B}_\phi(c_{j_0},r_{h}(c_{j_0})) \subset \bar{B}_\phi(\cb,r^+)$.
Therefore $P(\bar{B}_\phi(\cb,r^+)) \geq s$ for every $s\leq h$, hence $r_{s}(\cb) \leq r^+$. 
\end{proof}

\subsection{Proof of Lemma \ref{lm:eq_normes_compact}}
\begin{lemma*}[\ref{lm:eq_normes_compact}]
      Assume that $F_0 \subset\mathring{\Omega}$ and $\phi$ is $\mathcal{C}^2$ on $\Omega$.
     Then, for every $K>0$, there exists $C_K >0$ such that for every $\cb$ and $\cb'$ in $\bar{B}(0,K) \cap F_0$, and $x \in \Omega$, 
     \[
     \left | d_{\phi}(x,\cb) - d_{\phi}(x,\cb') \right | \leq C_K D(\cb,\cb') \left ( 1 + \|x\| \right ),
     \]
     where $D(\cb,\cb')=\min_{\sigma\in\Sigma_k}\max_{j\in[\![1,k]\!]}|c_j-c'_{\sigma(j)}|$ (cf. Theorem \ref{thm: convergence ps des codebooks}).
      \end{lemma*}
\begin{proof}[Proof of Lemma \ref{lm:eq_normes_compact}]
      The set $F_0 \cap \bar{B}(0,K)$ is a convex compact subset of $\mathring{\Omega}$. Let $x \in \R^d$ and $\cb,\cb' \in\left(F_0 \cap \bar{B}(0,K)\right)^{(k)}$. Since $\phi$ and $x\mapsto\nabla\phi(x)$ are $\mathcal{C}^1$, the mean value theorem yields that for every $j\in[\![1,k]\!]$,
      \begin{align*}
      \left | d_{\phi}(x,c_j) - d_{\phi}(x,c'_j) \right | & \leq \begin{multlined}[t] \left | \phi(c'_j) - \phi(c_j) \right | + \left | \left\langle x, \nabla_{c'_j} \phi - \nabla_{c_j}\phi \right\rangle \right | \\ + \left | \left\langle \nabla_{c'_j} \phi, c'_j \right\rangle - \left\langle \nabla_{c_j} \phi,c_j \right\rangle \right | \end{multlined} \\
      & \leq C_K \|c_j - c'_j \| (1+\|x\|),      
      \end{align*}
      for some constant $C_K$. Thus, 
      \[\left | d_{\phi}(x,\cb) - d_{\phi}(x,\cb') \right | \leq C_K (1+ \|x\|) \max_j \| c_j - c'_j\|. \]
      \end{proof}      

\subsection{Proof of Lemma \ref{lm: continuite de VPhc}}
\begin{lemma*}[\ref{lm: continuite de VPhc}]
Assume that $F_0 \subset\mathring{\Omega}$, $P\|u\|<\infty$ and $\phi$ is $\mathcal{C}^2$ on $\Omega$.
Then the map $(s,\cb)\rightarrow R_s(\cb)$ is continuous. Moreover, for every $h\in(0,1)$, $\epsilon>0$ and $K>0$, there is $s_0<h$ such that
\[\forall s_0<s<h,\,\sup_{\cb\in \left(F_0\cap\bar{B}(0,K)\right)^{(k)}}R_h(\cb)-R_s(\cb)\leq\epsilon.\]
\end{lemma*}
\begin{proof}[Proof of Lemma \ref{lm: continuite de VPhc}]
According to Lemma \ref{lm: ecriture DTM} and Lemma \ref{lm:eq_normes_compact}, for every $h\in(0,1)$, $\cb,\cb'\in\left(F_0\cap\bar{B}(0,K)\right)^{(k)}$ and $\tilde P_{\cb'}\in\mathcal{P}_h(\cb')$,
\begin{align*}
R_h(\cb)-R_h(\cb')\leq h (\tilde P_{\cb'}d_{\phi}(u,\cb)-\tilde P_{\cb'}d_{\phi}(u,\cb')) 
&\leq h \tilde P_{\cb'}|d_{\phi}(u,\cb)-d_{\phi}(u,\cb')|\\
&\leq C_K D(\cb,\cb')(1+P\|u\|),
\end{align*}
for some $C_K>0$.
As a consequence, $|R_h(\cb)-R_h(\cb')|\rightarrow 0$ when $D(\cb,\cb')\rightarrow 0$.
Now, let $s<h$, and let $\alpha_s$ and $\alpha_h$ be such that $\frac{1}{h}(P \1_{B_\phi(\cb,r_h(\cb))} + \delta_h P \1_{\partial B_\phi(\cb,r_h(\cb))}) \in \mathcal{P}_h(\cb)$ (resp. $\frac{1}{s}(P \1_{B_\phi(\cb,r_s(\cb))} + \delta_s P \1_{\partial B_\phi(\cb,r_s(\cb))}) \in \mathcal{P}_s(\cb)$). Then 
\begin{align*}
R_h(\cb)-R_s(\cb)& = \begin{multlined}[t] P d_{\phi}(u,\cb) \left(\1_{B_\phi(\cb,r_{h}(\cb))}(u) + \delta_h \1_{\partial B_\phi(\cb,r_{h}(\cb))}(u) \right ) \\ 
- P d_{\phi}(u,\cb) \left (\1_{B_\phi(\cb,r_{s}(\cb))}(u) +  \delta_s \1_{\partial B_\phi(\cb,r_{s}(\cb))} (u)\right) \end{multlined} \\
&  \leq r_{h}^2(\cb)\left(h-s\right).
\end{align*}
Moreover, according to Lemma \ref{lem:bound_radius_balls}, $\sup_{\cb\in\left(F_0\cap\bar{B}(0,K)\right)^{(k)}}r_{h}(\cb)\leq r^+$ for some $r^+<\infty$, hence the result.
\end{proof}

\subsection{Proof of Lemma \ref{lm: key lemma h- h+}}\label{tecsec:proof_lm_keylemma_h-h+}
Throughout this section, for any $\cb \in \Omega^{(k)}$ and $s \in ]0,1]$, we denote by $\tau_s(\cb)$ a map in $[0,1]$ such that $\frac{1}{s}P \tau_s(\cb) \in \mathcal{P}_s(\cb)$, and by $T_s$ the operator
\begin{align*}
T_s: \left \{ \begin{array}{@{}ccl}
\Omega^{(k)} & \rightarrow & F_0^{(k)} \\
\cb & \mapsto & \left ( \frac{Pu \1_{W_j(\cb)}(u) \tau_s(\cb)(u)}{P \tau_s(\cb)(u) \1_{W_j(\cb)}(u)} \right )_{j\in[\![1,k]\!]},
\end{array}
\right .
\end{align*}
with the convention $T_s(\cb)_j = c_j$ whenever $P \tau_s(\cb) \1_{W_j(\cb)} = 0$. The statement of Lemma \ref{lm: key lemma h- h+} is recalled below.
\begin{lemma*}[\ref{lm: key lemma h- h+}]
Assume that the requirements of Theorem \ref{thm:existence_optimal} are satisfied. For every $k\geq 2$,
if $R^*_{k-1,h} - R^*_{k,h} >0$, then
\[
\alpha := \min_{j\in[\![2,k]\!]} R^*_{j-1,h} - R^*_{j,h} >0.
\]
Moreover there exists $0<h^- < h < h^+<1$ such that, for every $j \in [\![2,k]\!]$, $R^*_{j-1,h^-} - R^*_{j,h^+} \geq \frac{\alpha}{2}$.

For any $h\wedge(1-h)\geq b>0$, denote by $h_b^- = (h-b)/(1-b)$, $h_b^+ = h/(1-b)$. Let $b$ be such that $\min_{j\in[\![2,k]\!]} R^*_{j-1,h_b^-} - R^*_{j,h_b^+} >0$, and, for $\kappa_1$ in $(0,1)$, denote by $b_1 = \kappa_1b$. Then, for any $s \in [h_{b_1}^-,h_{b_1}^+]$ and $j \in [\![1,k]\!]$, there exists a minimizer $\cb^*_{j,s}$ of $R_{j,s}$ satisfying
\[
\forall p \in [\![1,j]\!], \quad \| c^*_{j,s,p} \| \leq \frac{P \|u\|}{b(1-h)(1-\kappa_1)}. 
\] 
\end{lemma*}
The proof of Lemma \ref{lm: key lemma h- h+} proceeds from two intermediate results, Lemma \ref{lem:keylemmah-h+_part1} and Lemma \ref{lem:keylemmah-h+_part2} below. First, Lemma \ref{lem:keylemmah-h+_part1} ensures that there exists bounded optimal codebooks whenever $R^*_{k-1,h} - R^*_{k,h} >0$.
\begin{lemma}\label{lem:keylemmah-h+_part1}
For every $k\geq 2$,
if $R^*_{k-1,h} - R^*_{k,h} >0$, then
\[
\alpha := \min_{j\in[\![2,k]\!]} R^*_{j-1,h} - R^*_{j,h} >0.
\]
Moreover there exists $0<h^- < h < h^+<1$ and $C_{h^-,h^+}$ such that, for every $j \in [\![2,k]\!]$ and $s \in [h^-,h^+]$,
\begin{itemize}
\item $R^*_{j-1,h^-} - R^*_{j,h^+} \geq \frac{\alpha}{2}$.
\item For every $\frac{\alpha}{4}$-minimizer $\cb^*_{j,s}$ of $R^*_{j,s}$, $\sup_{p\in[\![1,j]\!]} \|T_s(\cb^*_{j,s})_p\| \leq C_{h^-,h^+}$.
\item There is a minimizer $\cb^*_{j,s}$ of $R^*_{j,s}$ such that $\forall p\in[\![1,j]\!]$, $\|c^*_{j,s,p}\|\leq C_{h^-,h^+}$ and $c^*_{j,s,p}\in F_0$.
\end{itemize}
\end{lemma}
A proof of Lemma \ref{lem:keylemmah-h+_part1} is given in Section \ref{sec:proof_lemma_key_lemma_h+h-_part1}. The other intermediate result, Lemma \ref{lem:keylemmah-h+_part2}, ensures that optimal codebooks cells have enough mass.
\begin{lemma}\label{lem:keylemmah-h+_part2}
Assume that $R^*_{k-1,h} - R^*_{k,h} >0$, and let $b$ be such that $\min_{j\in[\![2,k]\!]} R^*_{j-1,h_b^-} - R^*_{j,h_b^+} >0$. Let $\kappa_1 <1$ and $b_1 = \kappa_1 b$. Then, for any $s \in [h_{b_1}^-, h_{b_1}^+]$ and $j \in [\![1,k]\!]$, if $\cb^*_{j,s}$ is a minimizer of $R_{j,s}$, we have
\[
\forall p \in [\![1,j]\!] \quad P \tau_s(\cb^*_{j,s})\1_{W_p(\cb^*_{j,s})} \geq b(1-h)(1-\kappa_1).
\]
\end{lemma} 
The proof of Lemma \ref{lem:keylemmah-h+_part2} is given in Section \ref{sec:proof_lem:keylemmah-h+_part2}. Equipped with these two lemmas, we are in position to prove Lemma \ref{lm: key lemma h- h+}.
\begin{proof}[Proof of Lemma \ref{lm: key lemma h- h+}]
Assume that $R^*_{k-1,h} - R^*_{k,h} >0$, then Lemma \ref{lem:keylemmah-h+_part1} entails that there exists $b$ such that $\min_{j\in[\![2,k]\!]}R^*_{j-1,h_b^-}- R^*_{j,h_b^+} >0$. Moreover, for any $s \in [h_{b_1}^-, h_{b_1}^+]$, and $j\in[\![1,k]\!]$, Lemma \ref{lem:keylemmah-h+_part1} provides a minimizer $\cb^*_{j,s}$ of $R_{j,s}$. According to Proposition \ref{prop:centroid}, $T_s(\cb^*_{j,s})$ is a $R_{j,s}$-minimizer. According to Lemma \ref{lem:keylemmah-h+_part2}, it satisfies, for all $p \in [\![1,j]\!]$,
\begin{align*}
\|(T_s(\cb^*_{j,s}))_p\| \leq \frac{P\|u\|}{b(1-h)(1-\kappa_1)}.
\end{align*}
\end{proof}

\subsection{Proof of Lemma \ref{lem:keylemmah-h+_part1}}\label{sec:proof_lemma_key_lemma_h+h-_part1}
First note that if there exists $j \leq k$ such that $R^*_{j-1,h} - R^*_{j,h}=0$, then there exists a set $A$ with $P(A)\geq h$ such that the restriction of $P$ to the set $A$, $P_{|A}$, is supported on $j-1$ points. Thus, $R^*_{k-1,h}=R^*_{k,h}=0$. As a consequence, when $R^*_{k-1,h} - R^*_{k,h}>0$, $\alpha$ is positive.

Note also that the third point follows on from the second point. Indeed, for every sequence $\cb^{*\,(n)}_{k,s}$ of $\frac{\alpha}{4n}$-minimizers of $R^*_{k,s}$, for every $i\in[\![1,k]\!]$, $\|T_s(\cb^{*\,(n)}_{k,s})_i\|\leq C_{h^-,h^+}$. Since $\left(\bar{B}(0,C_{h^-,h^+})\cap F_0\right)^{(k)}$ is a compact set, the limit in $\left(\bar{B}(0,C_{h^-,h^+})\cap F_0\right)^{(k)}$ of any converging subsequence of $(T_s(\cb^{*\,(n)}_{k,s}))_n$ is a minimizer of $R_{k,s}$.

Now assume that $R^*_{k-1,h}-R^*_{k,h}>0$. In order to prove the other points, we proceed recursively. Assume that $k=2$. Since, for $s>0$ and any $1$-point codebook $c$, $\|T_s(c)\| \leq P \|u\|/s$, optimal $1$-point codebooks can be found in $\bar{B}(0,C_1) \cap F_0$, with $C_1 = \frac{P\|u\|}{s}$. From a compactness argument there exists an optimal $1$-point codebook $\cb^*_{1,s}$ satisfying $\|\cb^*_{1,s}\| \leq P \|u\|/s$.

Denote by $\cb^*_{1,h^-}$ a minimizer of $R^*_{1,h^-}$, and $\cb^*_{2,h}$ an $\frac{\alpha}{8}$-minimizer of $R^*_{2,h}$. According to Lemma \ref{lm: continuite de VPhc}, for a fixed $\cb$, $s\mapsto R_s(\cb)$ is continuous, thus we may choose $h^+$ such that $R_{h^+}(\cb^*_{2,h}) \leq R_h(\cb^*_{2,h}) + \frac{\alpha}{8}$. Then,
\begin{align*}
R^*_{2,h^+}  \leq R_{h^+}(\cb^*_{2,h})  
             \leq R_{h}(\cb^*_{2,h}) + \frac{\alpha}{8} 
             \leq R^*_{2,h} + \frac{\alpha}{4}.
\end{align*}
On the other hand, set $h_1 = \frac{h}{2}$. Then $\sup_{s \geq h_1} \|\cb^*_{1,s}\| \leq \frac{P\|u\|}{h_1}= C_{h_1}$. According to Lemma \ref{lm: continuite de VPhc}, there exists $h>h_2 \geq h_1$ such that $\sup_{\|c\| \leq C_{h_1}} (R_h(c) - R_{h_2}(c)) \leq \frac{\alpha}{4}$. For such an $h_2$, we may write
\begin{align*}
R^*_{1,h_2}  =  R_{h_2}(\cb^*_{1,h_2})  \geq R_h( \cb^*_{1,h_2}) - \frac{\alpha}{4}  \geq R_{1,h}^* - \frac{\alpha}{4}. 
\end{align*}
Since $R^*_{1,h}-R^*_{2,h}\geq\alpha$, it comes that  $R^*_{1,h_2} - R^*_{2,h^+} \geq \frac{\alpha}{2}$.
 
Now, if $\cb=(c_1,c_2)$ is an $\alpha/4$-minimizer of $R^*_{2,s}$, for $h^+\geq s \geq h-(h-h_2)/2=:h^-$, it holds $P\tau_s(\cb) \1_{W_j(\cb)} \geq h-h^-$, for $j \in \{1,2\}$. Indeed, suppose that $P\tau_s(\cb) \1_{W_1(\cb)} < h-h^-$. Then
\begin{align*}
R^*_{2,h^+}  \geq R_s(\cb) - \frac{\alpha}{4} 
             \geq P d_{\phi}(u,c_2) \tau_s(\cb)(u)\mathbbm{1}_{W_{2}(\cb)}(u) - \frac{\alpha}{4} 
             \geq R^*_{1,h_2} - \frac{\alpha}{4},
\end{align*}
since $s-h+h^- \geq h_2$, hence the contradiction. Choosing $h^+\geq s \geq h^-$ entails that $\|T_s(\cb^*_{j,s})_p\|\leq\frac{P\|u\|}{h-h^-}$ for every $p\in\{1,2\}$, this gives the result for $k=2$.

Assume that the proposition is true for index $k-1$, we will prove that it is also true for index $k$.
Set $\alpha = \min_{j\in[\![2,k]\!]} R^*_{j-1,h}-R^*_{j,h}>0$.
Let $h^{--}$ and $h^{++}$ be the elements $h^-$ and $h^+$ associated with step $k-1$.
Set $\cb^*_{k-1,h^{--}}$ a minimizer of $R^*_{k-1,h^{--}}$ and $\cb^*_{k,h}$ an $\frac{\alpha}{8}$-minimizer of $R^*_{k,h}$.
According to Lemma \ref{lm: continuite de VPhc}, there exists  $h<h^+<h^{++}$ such that $R_{h^+}(\cb^*_{k,h})\leq R_h(\cb^*_{k,h})+\frac{\alpha}{8}$. Thus $R^*_{k,h^+}\leq R^*_{k,h}+\frac{\alpha}{4}$.
On the other hand, Lemma \ref{lm: continuite de VPhc} provides $h>h_1>h^{--}$ such that
\[\sup_{\cb\in\left(\bar{B}(0,C_{h^{--},h^{++}})\cap F_0\right)^{(k)}}(R_h(\cb)-R_{h_1}(\cb))\leq\frac{\alpha}{4}.\]
Then, according to step $k-1$, since $\|(\cb^*_{k-1,h_1})_j\| \leq C_{h^{--}, h^{++}}$ for $j\in[\![1,k]\!]$, we may write
\[R^*_{k-1,h_1} = R_{h_1}(\cb^*_{k-1,h_1}) \geq R^*_{k-1,h} - \frac{\alpha}{4}.\]
As a consequence, since $R^*_{k-1,h} - R^*_{k,h}\geq\alpha$, we have $R^*_{k-1,h_1} - R^*_{k,h^+}\geq\frac{\alpha}{2}$.
Now, let $\cb$ be an $\frac{\alpha}{4}$-minimizer of $R^*_{k,s}$, for $h^+\geq s\geq h^-=\frac{h+h_1}{2}$, and assume that $P\tau_s(\cb)\1_{W_1(\cb)}< h - h^-$. Then
\begin{align*}
R^*_{k,h^+}&\geq R_s(\cb) - \frac{\alpha}{4}
\geq P\sum_{j=2}^k d_{\phi}(u,c_j)\tau_s(\cb)(u)\1_{W_j(\cb)}(u) - \frac{\alpha}{4}.
\end{align*}
Then, $R^*_{k,h^+} \geq R^*_{k-1,h_1}- \frac{\alpha}{4}$,
hence the contradiction. Thus, for such a choice of $h^-$ and $h^- \leq s \leq h^+$, $P\tau_s(\cb)\1_{W_p(\cb)}\geq h - h^-$, which entails $\|(T_s(\cb))_p\| \leq P\|u\|/(h-h^-)$, for every $p\in[\![1,k]\!]$.
\qed

\subsection{Proof of Lemma \ref{lem:keylemmah-h+_part2}}\label{sec:proof_lem:keylemmah-h+_part2}
Let $s \in [h_{b_1}^-, h_{b_1}^+]$, $j\in [\![1,k]\!]$ and $\cb^*_{j,s}$ be a $R_{j,s}$ minimizer. If $j=1$, then $P \tau_s(\cb^*_{j,s}) = s \geq h_{b_1}^- \geq b(1-\kappa_1)(1-h)$. Now assume that $j \geq 2$, and, without loss of generality, that $P \tau_s(\cb^*_{j,s})\1_{W_1(\cb^*_{j,s})} < b(1-\kappa_1)(1-h) < h_{b_1}^- - h_{b}^-$. We may write
\begin{align*}
R^*_{j,h_b^+} & \geq R^*_{j,s} \geq \sum_{p=2}^j P d_{\phi}(u,c^*_{j,s,p})\tau_s(\cb^*_{j,s})(u)\1_{W_p(\cb^*_{j,s})}(u) \\
              & \geq R^*_{j-1,s - (h_{b_1}^- - h_b^-)} \geq R^*_{j-1,h_b^-},
\end{align*}
hence the contradiction.
\qed
\section{Technical proofs for Section \ref{sec:proofs_section_theoretical_results}}\label{tecsec:proofs_inter_theoretical_results}

\subsection{Proof of Proposition \ref{prop:deviations}}\label{tecsec:proof_deviations}
\begin{proposition*}[\ref{prop:deviations}]
With probability larger than $1-e^{-x}$, we have, for $k \geq 2$,
\[
\eqref{eq:deviationVCballs}: \quad  
\begin{array}{@{}ccc}
\sup_{\cb \in (\R^d)^{(k)}, r \geq 0} \left | (P - P_n) \mathbbm{1}_{B_\phi(\cb,r)} \right | & \leq & C \sqrt{\frac{k(d+1) \log(k)}{n}} + \sqrt{\frac{2x}{n}}, \\
\sup_{\cb \in (\R^d)^{(k)}, r \geq 0} \left | (P - P_n) \mathbbm{1}_{\partial B_\phi(\cb,r)} \right | & \leq & C \sqrt{\frac{k(d+1) \log(k)}{n}} + \sqrt{\frac{2x}{n}},
\end{array}
\]
where $\partial B_\phi(\cb,r)$ denotes $\left \{ x \mid d_\phi(x,\cb)=r^2 \right \}$ and $C$ denotes a universal constant. Moreover, if $r^+$ and $K$ are fixed, and $P \|u\|^2 \leq M_2<\infty$, we have, with probability larger than $1-e^{-x}$,
\[
\eqref{eq:deviationcontrastfunction}: \quad 
\begin{array}{@{}ccc}
\sup_{\cb \in \left(\bar{B}(0,K)\cap F_0\right)^{(k)}, r \leq r^+}\left | (P - P_n) d_{\phi}(.,\cb) \mathbbm{1}_{B_\phi(\cb,r)} \right |  & \leq & (r^+)^2 \left [C_{K,r^+,M_2} \frac{\sqrt{kd\log(k)}}{\sqrt{n}} +  \sqrt{\frac{2x}{n}} \right ], \\
\sup_{\cb \in \left(\bar{B}(0,K)\cap F_0\right)^{(k)}, r \leq r^+}\left | (P - P_n) d_{\phi}(.,\cb) \mathbbm{1}_{\partial B_\phi(\cb,r)} \right | & \leq &  (r^+)^2 \left [C_{K,r^+,M_2} \frac{\sqrt{kd\log(k)}}{\sqrt{n}} +  \sqrt{\frac{2x}{n}} \right ], 
\end{array}
\]
where we recall that $\overline{conv(supp(P))}=F_0 \subset \mathring{\Omega}$.
\end{proposition*}

The proof of Proposition \ref{prop:deviations} will make use of the following results. The first one deals with VC dimension of Bregman balls.
\begin{lemma}\label{lem:VCdimbregmanballs}
Let $\mathcal{C}$ (resp. $\bar{\mathcal{C}}$) denote the class of open (resp. closed) Bregman balls $B_\phi(x,r)= \{ y \in \R^d \mid \sqrt{d_{\phi}(y,x)} < r \}$ (resp. $\bar{B}_\phi(x,r)= \{ y \in \R^d \mid \sqrt{d_{\phi}(y,x)} < r \}$), $x \in \R^d$, $r\geq 0$. Then
\begin{align*}
d_{VC}(\mathcal{C}) & \leq d+1, \\ 
d_{VC}(\bar{\mathcal{C}}) & \leq d+1,
\end{align*}
where $d_{VC}$ denotes the Vapnik-Chervonenkis dimension.
\end{lemma}
\begin{proof}[Proof of Lemma \ref{lem:VCdimbregmanballs}]
Let $S = \{x_1, \dots, x_{d+2}\}$ be shattered by $\mathcal{C}$, and let $A_1$, $A_2$ be a partition of $S$. Then we may write
\begin{align*}
A_1 &= S \cap B_\phi(c_1,r_1) \cap B_\phi(c_2,r_2)^c \\
A_2 & = S \cap B_\phi(c_2,r_2) \cap B_\phi(c_1,r_1)^c,
\end{align*}
for $c_1$, $c_2$ $\in$ $\R^d$ and $r_1,r_2 \geq 0$. Straightforward computation shows that, for any $x \in A_1$, 
\[
\ell_{1,2}(x) :=d_\phi(x,c_1)-d_\phi(x,c_2)<0. 
\]
Similarly we have that, for any $x \in A_2$, $\ell_{1,2}(x) >0$. 
Since $\ell_{1,2}(x) = \phi(c_2) - \phi(c_1) + \left\langle x , \nabla_{c_2}\phi - \nabla_{c_1}\phi \right\rangle + \left\langle \nabla_{c_1}\phi,c_1 \right\rangle - \left\langle \nabla_{c_2}\phi,c_2 \right\rangle - r_1^2 + r_2^2$, $S$ is shattered by affine halfspaces (whose VC-dimension is $d+1$), hence the contradiction. The same argument holds for $\bar{C}$.
\end{proof}
Next, to bound expectation of suprema of empirical processes, we will need the following result. For any set of real-valued functions $\mathcal{F}$, let $\mathcal{N}(\mathcal{F},\varepsilon, L_2(P_n))$ denote its $\varepsilon$ covering number with respect to the $L_2(P_n)$ norm. In addition, the pseudo-dimension of $\mathcal{F}$, $d_{VC}(\mathcal{F})$, is defined as the Vapnik dimension of the sets $\{ (x,t) \mid f(x) \leq t \}$. 
\begin{theorem}{\cite[Theorem 1]{Mendelson03}}\label{thm:mendelson_l2_VC}
      If $\mathcal{F}$ is a set of functions taking values in $[-1,1]$. Then, for all $\varepsilon \leq 1$
      \[
      \mathcal{N}(\mathcal{F},\varepsilon,L_2(P_n)) \leq \left ( \frac{2}{\varepsilon} \right )^{\kappa d_{VC}(\mathcal{F})},
      \]
      where $\kappa$ denotes a universal constant and with a slight abuse of notation $d_{VC}(\mathcal{F})$ denotes the pseudo-dimension of $\mathcal{F}$.
      \end{theorem}
We are now in a position to prove Proposition \ref{prop:deviations}.               
\begin{proof}[Proof of Proposition \ref{prop:deviations}]
Let $\cb \in \R^{(k)}$. Since ${B_\phi(\cb,r)} = \bigcup_{j=1}^{k} B_\phi(c_j,r)$, according to \cite[Theorem 1.1]{vanderVaart09} and Lemma \ref{lem:VCdimbregmanballs}, we may write
\[
d_{VC} \left ( \left \{ B_\phi(\cb,r) \mid \cb \in (\R^d)^{(k)}, r \geq 0 \right \} \right ) \leq c_1k(d+1) \log(c_2k),
\]
where $c_1$ and $c_2$ are universal constant. Thus, applying \cite[Theorem 3.4]{Boucheron05} gives the first inequality of \eqref{eq:deviationVCballs}. The second inequality of \eqref{eq:deviationVCballs} follows from 
\begin{multline*}
\sup_{\cb \in (\R^d)^{(k)}, r \geq 0} \left | (P - P_n) \mathbbm{1}_{\partial B_\phi(\cb,r)} \right | \leq \sup_{\cb \in (\R^d)^{(k)}, r \geq 0} \left | (P - P_n) \mathbbm{1}_{B_\phi(\cb,r)} \right | \\
 + \sup_{\cb \in (\R^d)^{(k)}, r \geq 0} \left | (P - P_n) \mathbbm{1}_{\bar{B}_\phi(\cb,r)} \right |.
\end{multline*}

      The inequalities of \eqref{eq:deviationcontrastfunction} are more involved. Denote by $Z$ the left-hand side of the first inequality. A bounded difference inequality (see, e.g., \cite[Theorem 6.2]{Massart16}) yields
      \[
      \mathbb{P}\left (Z \geq \mathbb{E}Z + (r^+)^2 \sqrt{\frac{2x}{n}} \right ) \leq e^{-x}.
      \]
      It remains to bound
      \[
      \mathbb{E} Z \leq 2 \mathbb{E}_X \mathbb{E}_\sigma \frac{1}{n} \sup_{\cb \in \left(B(0,K) \cap F_0\right)^{(k)}, r \leq r^+} \sum_{i=1}^{n} \sigma_i d_{\phi}(X_i,\cb) \mathbbm{1}_{B_\phi(\cb,r)}(X_i),
      \]
      according to the symmetrization principle (see, e.g., \cite[Lemma 11.4]{Massart16}), where for short $\mathbb{E}_Y$ denotes expectation with respect to the random variable $Y$. 
       Let $\Gamma_0$ denote the set of functions  $\left \{ \frac{d_{\phi}(.,\cb)}{(r^+)^2} \mathbbm{1}_{B_\phi(\cb,r)} \mid \cb \in B(0,K)^{(k)}, r \leq r^+ \right \}$. 
            We have to assess the covering number $\mathcal{N}(\Gamma_0,\varepsilon, L_2(P_n))$.
           It is immediate that 
      \[
      \mathcal{N}(\Gamma_0,\varepsilon,L_2(P_n)) \leq \mathcal{N}(\Gamma_1,\varepsilon/2,L_2(P_n)) \times \mathcal{N}(\Gamma_2, \varepsilon/2,L_2(P_n)),
      \]
      where $\Gamma_1 = \left \{ \frac{d_{\phi}(.,\cb)}{(r^+)^2} \wedge 1 \right \}$ and $\Gamma_2 = \left \{\mathbbm{1}_{B_\phi(\cb,r)} \right \}$. On one hand, we have
      \begin{align*}
      \mathcal{N}(\Gamma_2,u,L_2(P_n)) & = \mathcal{N}(1-\Gamma_2,u,L_2(P_n)) \\
      & = \mathcal{N} \left ( \left \{ \prod_{j=1}^k \mathbbm{1}_{B_\phi(c_j,r)^c} \mid \cb, r \right \},u,L_2(P_n) \right ) \\
      & \leq \mathcal{N} \left ( \left \{ \mathbbm{1}_{B_\phi(c,r)} \mid c \in \R^d, r \geq 0 \right \}, u/k,L_2(P_n) \right )^k \\
      & \leq \left ( \frac{2k}{u} \right )^{\kappa(d+1)k},
      \end{align*}
      according to Theorem \ref{thm:mendelson_l2_VC}. 
      
      Now turn to $\Gamma_1$. According to Lemma \ref{lm:eq_normes_compact}, we may write
      \begin{align*}
      \mathcal{N} ( \Gamma_1,u,L_2(P_n)) & \leq \mathcal{N} \left ( \left \{ \frac{d_{\phi}(.,\cb)}{(r^+)^2} \mid \cb \in \left ( F_0 \cap B(0,K) \right)^{(k)} \right \},u,L_2(P_n) \right ) \\
      & \leq \mathcal{N} \left ( B(0,K)^{k}, \frac{(r^+)^2 u}{C_K (1+ \|x\|_{L_2(P_n)})},d_H \right ).
      \end{align*}
      Since $\mathcal{N}(B(0,1),u,\|\cdot\|) \leq \left( \frac{3}{u} \right )^d$, it follows that 
      \begin{align*}
      \mathcal{N} ( \Gamma_1,u,L_2(P_n)) \leq \left ( \frac{3K C_K (1+\|x\|_{L_2(P_n) )})}{ (r^+)^2 u} \right )^{kd},
      \end{align*}
      hence
      \begin{align*}
      \mathcal{N}(\Gamma_0,\varepsilon,L_2(P_n)) \leq \left ( \frac{6K C_K  (1+\|x\|_{L_2(P_n)})}{ (r^+)^2 \varepsilon } \right )^{kd} \times \left ( \frac{4k}{\varepsilon} \right )^{\kappa(d+1)k}.
      \end{align*}
Using Dudley's entropy integral (see, e.g., \cite[Corollary 13.2]{Massart16}) yields, for $k\geq 2$,
\begin{multline*}
\mathbb{E}_\sigma \frac{1}{n} \sup_{\cb \in \left(B(0,K) \cap F_0\right)^{(k)}, r \leq r^+} \sum_{i=1}^{n} \sigma_i \frac{d_{\phi}(X_i,\cb)}{(r^+)^2} \mathbbm{1}_{B_\phi(\cb,r)}(X_i) \\ \leq C \frac{(r^+)^2}{\sqrt{n}} \sqrt{kd \log \left ( \frac{C_K(1 + \|x\|_{L_2(P_n)})}{(r^+)^2} \right ) + \kappa(d+1) k \log(4k)}.
\end{multline*}
Thus, applying Jensen's inequality leads to 
\begin{align*}
\mathbb{E}_X \mathbb{E}_\sigma \frac{1}{n} \sup_{\cb \in \left(B(0,K) \cap F_0\right)^{(k)}, r \leq r^+} \sum_{i=1}^{n} \sigma_i \frac{d_{\phi}(X_i,\cb)}{(r^+)^2} \mathbbm{1}_{B_\phi(\cb,r)}(X_i) & \leq (r^+)^2 C_{K,r^+,M_2} \sqrt{\frac{kd\log(k)}{n}}.
\end{align*}
Replacing $\1_{B_\phi(\cb,r)}$ with $\1_{\partial B_ \phi(\cb,r)}$ in the definition of $\Gamma_0$ and $\Gamma_2$ gives the same inequality, and \eqref{eq:deviationcontrastfunction}. 
\end{proof}

\subsection{Proof of Proposition \ref{prop:bounded_empirical_codebooks}}\label{tecsec:Proof_Proposition_bounded_empirical_codebook}
\begin{proposition*}[\ref{prop:bounded_empirical_codebooks}]
Assume that $P\|u\|^p < + \infty $ for some $p \geq 2$, and let $b>0$ be such that $\min_{j\in[\![2,k]\!]} R^*_{j-1,h_b^-}-R^*_{j,h_b^+} >0$, where $h_b^- = (h-b)/(1-b)$, $h^+_b=h/(1-b)$, as in Lemma \ref{lm: key lemma h- h+}. Let $\kappa_2<1$, and denote by $b_2 = \kappa_2 b$.
 Then there exists $C_{P,h,k,\kappa_2}$ such that, for $n$ large enough, with probability larger than $1-n^{-\frac{p}{2}}$, we have, for all $j\in[\![2,k]\!]$, and $i\in[\![1,j]\!]$, 
\begin{align*}
\sup_{ h_{b_2^-} \leq s \leq h} \|\hat{c}_{j,s,i}\| \leq C_{P,h,k,\kappa_2,b},
\end{align*}
where $\hat{\cb}_{j,s}$ denotes a $j$-codepoints empirical risk minimizer with trimming level $s$.
\end{proposition*}
\begin{proof}[Proof of Proposition \ref{prop:bounded_empirical_codebooks}]
For a codebook $\cb$, let $\hat{\tau}_{s}(\cb)$ denote the trimming function $\1_{B_\phi(\cb,r_n({\cb}))} + \alpha_{n,s}(\cb)\1_{\partial B_\phi(\cb,r_n(\cb))}$, so that $\frac{1}{s}P_n \hat{\tau}_{s} \in \hat{\mathcal{P}}_s(\cb)$, and let $\tau_s(\cb)$ denote the trimming function for the distribution $P$.
Similarly to the proof of Theorem \ref{thm:existence_optimal}, we denote by $\hat{T}_s$ the operator that maps $\hat{\cb}$ to the empirical means of its Bregman-Voronoi cells, that is
\[
(\hat{T}_s(\cb))_j = \frac{P_n u \hat{\tau}_s(\cb) \1_{W_j(\cb)}}{P_n \hat{\tau}_s(\cb) \1_{W_j(\cb)}}.
\] 
We let $b_2=\kappa_2b$, $\kappa_2<1$, and choose $b_1 = (\kappa_2 + \frac{1-\kappa_2}{2})b$ so that $b_2 < b_1 <b$. At last we denote by $\eta= \left [ (h_{b_2}^- - h_{b_1}^-) \wedge (h_{b_1}^+ - h_{b_2}^+) \right ]/2$, and will prove recursively that, for $j\in[\![1,k]\!]$ and $i\in[\![1,j]\!]$, 
\begin{align*}
\sup_{ h_{b_1}^- + \eta + \frac{j \eta }{k} \leq s \leq h} \|\hat{c}_{j,s,i}\| \leq C_{P,h,k,\kappa_2,b},
\end{align*}
where $\hat{\cb}_{j,s}$ denotes a $j$-codepoints empirical risk minimizer with trimming level $s$.

Since $P\|u\|^p < +\infty$, Lemma \ref{lm: Marcinkiewicz} yields that $P_n\|u\| \leq C_1$, for $C_1$ large enough, with probability larger than $1-\frac{1}{8n^{\frac{p}{2}}}$. We set 
\[
C_{P,h,k,\kappa_2,b} = \frac{C_1}{h_{b_1}^- + \eta(1 + 1/k)}\vee C_{b_1,P,h} \vee \frac{2kC_1}{\eta},
\]
where $C_{b_1,P,h}$ is given by Lemma \ref{lm: key lemma h- h+}.
According to Proposition \ref{prop:deviations}, for $n$ large enough, we have that,
\begin{align}\label{eq:deviation_balls}
\sup_{\cb \in (\mathbb{R}^d)^{(k)}, r \geq 0} \left | (P-P_n) B_\phi(\cb,r) \right | & \leq \frac{\eta}{4k} \\
 \sup_{\cb \in (\mathbb{R}^d)^{(k)}, r \geq 0} \left | (P-P_n) \partial B_\phi(\cb,r) \right | & \leq \frac{\eta}{4k}, \notag
\end{align}
with probability larger than $1-\frac{1}{8n^{\frac{p}{2}}}$. On this probability event, from Lemma \ref{lem:bound_radius_balls} and the fact that $s\mapsto r_{s}(\cb)$ is non-decreasing, we deduce that for some $r^+>0$,
\begin{align*}
\sup_{c \in B(0,C_{P,h,k,\kappa_2,b})\cap F_0, s \leq h_{b_1}^+}{r_{n,s}(c) \vee r_{s}(c)} \leq r^+. 
\end{align*} 
We recall that since $P\|u\|^p < +\infty$, Lemma \ref{lm: Marcinkiewicz} yields that $P_n\|u\| \leq C_1$, for $C_1$ large enough, with probability larger than $1-\frac{1}{8n^{\frac{p}{2}}}$. Besides, choosing $x  = \log(8n^{\frac{p}{2}})$ in Proposition \ref{prop:deviations}, we also have, with probability larger than $1-\frac{1}{8n^{\frac{p}{2}}}$, 
\begin{align}
\label{eq: borne deviation par alphan}
\sup_{\cb \in (\bar{B}(0,C_{P,h,k,\kappa_2,b})\cap F_0)^{(k)},r \leq r^+} \left | (P-P_n) d_{\phi}(u,\cb) \mathbbm{1}_{B_{\phi}(\cb,r)}(u) \right | & \leq \alpha_n \\
\sup_{\cb \in (\bar{B}(0,C_{P,h,k,\kappa_2,b})\cap F_0)^{(k)},r \leq r^+} \left | (P-P_n) d_{\phi}(u,\cb) \mathbbm{1}_{\partial B_{\phi}(\cb,r)}(u) \right | & \leq \alpha_n,\notag
\end{align}
where $\alpha_n = O(\sqrt{\log(n)/n})$. We then work on the global probability event on which all these deviation inequalities are satisfied, that has probability larger than $1-\frac{1}{n^{\frac{p}{2}}}$. We proceed recursively on $j$.

For $j=1$ and $h \geq s \geq h_{b_1}^- + \eta(1 + 1/k)$, according to Proposition \ref{prop:centroid}, $\hat{T}_s(\hat{\cb}_{1,s}) = \hat{\cb}_{1,s}$, hence
\begin{align*}
\| \hat{\cb}_{1,s} \| \leq \frac{P_n\|u\|}{h_{b_1}^- + \eta(1 + 1/k)} \leq \frac{C_1}{h_{b_1}^- + \eta(1 + 1/k)} \leq C_{P,h,k,\kappa_2,b}.
\end{align*}
Now assume that the statement of Proposition \ref{prop:bounded_empirical_codebooks} holds up to order $j-1$. Let $\hat{\cb}_{j,s}$ be a $j$-points empirically optimal codebook with trimming level $h \geq s \geq h_{b_1}^- + \eta (1+j/k)$. 
Assume that there exists one cell (say $W_1$) such that $P_n(\1_{W_1(\hat{\cb}_{j,s})} \hat{\tau}_s(\hat{\cb}_{j,s})) \leq \frac{\eta}{k}$. On the one hand, we may write
\begin{align*}
\hat{R}_s(\hat{\cb}_{j,s})  & \leq \hat{R}_s(\cb^*_{j,h_{b_1}^+}) 
                           \leq P_n d_\phi(u,\cb^*_{j,h^+})  \hat{\tau}_{s}(\cb^*_{j,h_{b_1}^+})(u) \\
                           & \leq P_n d_\phi(u,\cb^*_{j,h^+})  \tau_{s+2 \eta/k}(\cb^*_{j,h_{b_1}^+})(u) \leq R^*_{j,h_{b_1}^+} + \alpha_n, 
\end{align*} 
where $\cb^*_{j,h_{b_1}^+}$ is a $R_{j,h_{b_1}^+}$ minimizer provided by Theorem \ref{thm:existence_optimal}. 

On the other hand, we have
\begin{align*}
\hat{R}_s(\hat{\cb}_{j,s}) & \geq \sum_{p=2}^{j} P_n d_\phi(u,\hat{c}_{j,s,p}) \mathbbm{1}_{W_p(\hat{\cb}_{j,s})} \hat{\tau}_{j,h_{b_1}^- + \eta(1+(j-1)/k)}(u)  \\
& \geq \hat{R}_{h_{b_1}^- + \eta(1+(j-1)/k)}(\hat{\cb}_{j-1,h_{b_1}^- + \eta(1+(j-1)/k)}).
\end{align*}
Thus, 
\begin{align*}
\hat{R}_s(\hat{\cb}_{j,s})
                   & \geq P d_\phi(u,\hat{\cb}_{j-1,h_{b_1}^- + \eta(1+(j-1)/k)}) \tau_{h_{b_1}^- + \eta(1+(j-1)/k) - \eta/2k}(\hat{\cb}_{j-1,h_{b_1}^- + \eta(1+(j-1)/k)}) - \alpha_n,
\end{align*}
according to the recursion assumption and \eqref{eq: borne deviation par alphan}. It comes 
\begin{align*}
\hat{R}_s(\hat{\cb}_{j,s}) & \geq R^*_{j-1,h_{b_1}^-} - \alpha_n,
\end{align*}
hence $R^*_{j-1,h_{b_1}^-} \leq R^*_{j,h_{b_1}^+} + 2 \alpha_n$, that is impossible for $n$ large enough. Therefore, for $n$ large enough and every $p\in[\![1,j]\!]$, 
\begin{align*}
P_n(\1_{W_p(\hat{\cb}_{j,s})} \hat{\tau}_s(\hat{\cb}_{j,s})) \geq 
 \frac{\eta}{k}.
 \end{align*}
 According to Proposition \ref{prop:centroid}, equality $\hat{T}_s(\hat{\cb}_{j,s}) = \hat{\cb}_{j,s}$ holds and entails
$
\|\hat{c}_{j,s,p}\| \leq \frac{2 k P_n \|u\|}{\eta} \leq C_{P,k,b,\kappa_2}
$.
\end{proof}

\subsection{Proof of Lemma \ref{lm: Marcinkiewicz}}\label{tecsec:proof_lm_marcinkiewicz}
\begin{lemma*}[\ref{lm: Marcinkiewicz}]
If $P\|u\|^p<\infty$ for some $p\geq 2$, then, there exists some positive constant $C$ such that with probability larger than $1-n^{-\frac{p}{2}}$,
\[P_n\|u\|\leq C.\]
\end{lemma*}
\begin{proof}[Proof of Lemma \ref{lm: Marcinkiewicz}]
According to the Markov inequality, we may write
\[\mathbb{P}\left(P_n\|u\|-P\|u\|\geq\epsilon\right)\leq\frac{\E\left[\left|\frac{1}{n}\sum_{i=1}^n\|X_i\|-P\|u\|\right|^p\right]}{\epsilon^p}.\]
That leads to 
\[\mathbb{P}\left(P_n\|u\|-P\|u\|\geq\epsilon\right)\leq\frac{\E\left[\left|\sum_{i=1}^n\left(\|X_i\|-P\|u\|\right)\right|^p\right]}{n^p\epsilon^p}.\]
From the Marcinkiewicz-Zygmund inequality applied to the real-valued centered random variables $Y_i=\|X_i\|-P\|u\|$ and the Minkowski inequality, it follows that
\begin{align*}
\E\left[\left|\sum_{i=1}^n\left(\|X_i\|-P\|u\|\right)\right|^p\right]&= \E\left[\left|\sum_{i=1}^nY_i\right|^p\right]\\
&\leq B_p \E\left[\left(\sum_{i=1}^nY_i^2\right)^{\frac p2}\right]\\
&\leq B_p \left(\sum_{i=1}^n\left(\E|Y_i|^p\right)^{\frac{2}{p}}\right)^{\frac{p}{2}}\\
&= B_p n^{\frac{p}{2}}\E\left[|Y|^p\right]\\
&= B_p n^{\frac{p}{2}}P\left(|\|u\|-P\|u\||^p\right),
\end{align*}
for some positive constant $B_p$.
Since $P\left(|\|u\|-P\|u\||^p\right)\leq P\|u\|^p + (P\|u\|)^p \leq 2P\|u\|^p$, according to Jensen inequality,
the result derives from a suitable choice of $\epsilon$.
\end{proof}  
\subsection{Proof of Lemma \ref{lemme: r_borne}}\label{tecsec:proof_lemma_r_borne}
\begin{lemma*}[\ref{lemme: r_borne}]
Let $(P_n)_{n\in\N}$ be a sequence of probabilities that converges weakly to a distribution $P$. Assume that $supp(P_n)\subset supp(P)\subset\R^d$, $F_0 = \overline{conv(supp(P))}\subset\mathring{\Omega}$ and $\phi$ is $\Ccal_2$ on $\Omega$. Then, for every $h\in(0,1)$ and $K>0$, there exists $K_+>0$ such that for every $\cb\in\Omega^{(k)}$ satisfying $|c_i|\leq K$ for some $i\in[\![1,k]\!]$ and every $n\in\N$,
\[r_{n,h}(\cb)\leq r_+ = \sqrt{4(2K+K_+)\sup_{c\in F_0\cap\bar{B}(0,2K+K_+)}\|\nabla_c\phi\|}.\]
\end{lemma*}
\begin{proof}[Proof of Lemma \ref{lemme: r_borne}]
Set $c\in\bar{B}(0,K)\cap F_0$. Since $P_n$ converges weakly to $P$, according to the Prokhorov theorem, $(P_n)_n$ is tight. Thus, there is $K_+>0$ such that $P_n(B(0,K_+))>h$ for all $n\in\N$ and $P(B(0,K_+))>h$.
It comes that $P_n(B(c,K+K_+))>h$.
Moreover, for every $x$, $y$ in $F_0\cap\bar{B}(0,2K+K_+)$, the mean value theorem yields
\begin{align*}
d_{\phi}(x,y)&\leq 2\sup_{c\in F_0\cap\bar{B}(0,2K+K_+)}\|\nabla_c\phi\| \|x-y\|\\
&\leq 4(2K+K_+)C_+ = (r^+)^2,
\end{align*}
for $C_+ = \sup_{c\in F_0\cap\bar{B}(0,2K+K_+)}\|\nabla_c\phi\| < + \infty$.
Thus, it follows that
\begin{equation}
\label{eq: boule eucli dans Bregman}
B(c,K+K_+)\subset B_\phi(c,r_+).
\end{equation}
As a consequence, $P_n(B_\phi(c,r_+))>h$ and $P_n(B_\phi(\cb,r_+))>h$ if $c\in\cb$ and $r_{n,\phi,h}(\cb)\leq r_+$.
\end{proof}
\subsection{Proof of Lemma \ref{lem:expected_distortion_tail}}\label{tecsec:proof_lem_expected_distortion_tail}
\begin{lemma*}[\ref{lem:expected_distortion_tail}]
Under the assumptions of Corollary \ref{cor:slow_rates_expectation}, if $P\|u\|^q \psi^q(k\|u\|/h) < \infty$, then there exists a constant $C_q$ such that $\E R_h^q(\hat{\cb}_n) \leq C_P^q$. 
\end{lemma*}
Let $\hat{\tau}_h(\hat{\cb}_n)$ be such that $\frac{1}{h} P_n \hat{\tau}_h(\hat{\cb}_n) \in \mathcal{P}_{n,h}(\hat{\cb}_n)$, and $\hat{j}$ be such that $P_n( \hat{\tau}_h(\hat{\cb}_n) \1_{W_{\hat{j}}(\hat{\cb}_n)} ) \geq \frac{nh}{k}$. According to the mean-value theorem and since $q \geq 1$ we may write
\begin{align*}
\mathbb{E}\left ( R_h(\hat{\cb}_n) \right )^q & \leq \mathbb{E} P d_\phi^q(u,\hat{c}_{n,\hat{j}}) \\
& \leq \E P 3^{q-1}\left [ \phi^q(u) + \phi(\hat{c}_{n,\hat{j}})^q + \psi^q(\|\hat{c}_{n,\hat{j}}\|)\|u-\hat{c}_{n,\hat{j}}\|^q \right ] \\
& \leq \E P 3^{q-1}\left [ \phi^q(u) + \phi(\hat{c}_{n,\hat{j}})^q + \psi^q(\|\hat{c}_{n,\hat{j}}\|)2^{q-1} \left ( \|u\|^q + \|\hat{c}_{n,\hat{j}}\|^q \right ) \right ] \\
& \begin{multlined}[t] \leq 3^{q-1} P \|u\|^q \psi^q(u) + 3^{q-1}(1+2^{q-1})\E \|\hat{c}_{n,\hat{j}}\|^q \psi^q(\|\hat{c}_{n,\hat{j}}\|) \\ + 6^{q-1} P\|u\|^q \E \psi^q(\|\hat{c}_{n,\hat{j}}\|).
\end{multlined}
\end{align*}
Since $\psi(t) \leq \psi \left ( \frac{k t}{h} \right )$, the first term is bounded. Also, note that since $p\geq 2$, $q\leq 2 \leq p$ so that $P\|u\|^q < \infty$. Next, since $\hat{\cb}_n$ satisfies the centroid condition, we have
\begin{align*}
\|\hat{\cb}_{n,\hat{j}}\| \leq \frac{P_n u \hat{\tau}_h(\hat{\cb}_n) (u)\1_{W_{\hat{j}}(\hat{\cb}_n)}(u) }{P_n\hat{\tau}_h(\hat{\cb}_n) (u)\1_{W_{\hat{j}}(\hat{\cb}_n)}(u) } \leq \frac{k}{nh} \sum_{i=1}^{n} \|X_i\|.
\end{align*}
Since $u \mapsto \|u\|^q \psi^q(u)$ is convex we may write
\begin{align*}
\E \|\hat{c}_{n,\hat{j}}\|^q \psi^q(\|\hat{c}_{n,\hat{j}}\|) & \leq \E \left [ \left (\frac{k}{nh} \sum_{i=1}^{n} \|X_i\|\right )^q \psi^q \left ( \frac{k}{nh} \sum_{i=1}^{n} \|X_i\| \right ) \right ] \\
& \leq \left ( \frac{k}{h} \right )^q P \left ( \|u\|^q \psi^q( k \|u\|/h) \right ) < \infty.
\end{align*}
At last, note that
\begin{align*}
P \psi^q(k\|u\|/h) & \leq P \left (\left ( \|u\|^q \vee 1 \right ) \psi^q(k\|u\|/h) \right ) \\
                & \leq P \left ( \|u\|^q \psi^q(k\|u\|/h) \1_{\|u\|>1} \right ) + P \left ( \psi^q(k\|u\|/h) \1_{\|u\|\leq 1} \right ) \\
                & \leq P\|u\|^q \psi^q(k\|u\|/h) + \psi^q(k/h) < \infty, 
\end{align*}
so that, using convexity of $\psi$, 
\begin{align*}
\E \psi^q(\|\hat{c}_{n,\hat{j}}\|) & \leq \E\left [ \psi^q \left (\frac{k}{nh} \sum_{i=1}^{n} \|X_i\|\right ) \right ] \\
                                  & \leq P \psi^q (k\|u\|/h) < \infty.
\end{align*}
Combinining all pieces entails that $\E(R_h(\hat{\cb}_n)) < \infty$.
\qed

\subsection{Proof of Lemma \ref{lem:enough_weight_optimalcells}}\label{tecsec:proof_lem_enough_weight_optimalcells}
\begin{lemma*}[\ref{lem:enough_weight_optimalcells}]
Assume that $B_h>0$ (see Definition \ref{def:discernability_factor}), let $b < B_h$ and $b<b_1<B_h$ such that $b = \kappa_1 b_1$, with $\kappa_1 <1$. Denote by $\beta_1 = (1-\kappa_1)b_1\left [h \wedge (1-h)\right ]/2$. Assume that $s/(n+s) \leq b$. Then, for $n$ large enough, with probability larger than $1-n^{-\frac{p}{2}}$, we have, for all $j\in[\![1,k]\!]$, 
\[
P_n \left ( \hat{\tau}_{h_b^-}(\hat{\cb}_{n+s,h})\1_{W_j(\hat{\cb}_{n+s,h})} \right ) \geq \beta_1.
\]
\end{lemma*}
\begin{proof}[Proof of Lemma \ref{lem:enough_weight_optimalcells}]
As in the proof of Proposition \ref{prop:bounded_empirical_codebooks}, we assume that
\begin{align*}
\sup_{\cb \in (\mathbb{R}^d)^{(k)}, r \geq 0} \left | (P-P_n) B_\phi(\cb,r) \right | \vee \left | (P-P_n) \partial B_\phi(\cb,r) \right |  \leq \beta_n  \leq \beta_1.
\end{align*}
According to Proposition \ref{prop:deviations}, for $n$ large enough, this occurs with $\beta_n = O \left  ( \sqrt{\log(n)/n} \right )$, and with probability larger than $1-\frac{1}{8n^{\frac{p}{2}}}$. On this probability event, we deduce as well that
$
\sup_{c \in B(0,C_P)\cap F_0, s \leq h_{b_1}^+}{r_{n,s}(c) \vee r_{s}(c)} \leq r^+ 
$,
for some $r^+>0$. 
We also assume that $P_n\|u\| \leq C_1$, for $C_1$ large enough, and
\begin{multline*}
\sup_{\cb \in (B(0,C_P)\cap F_0)^{(k)},r \leq r^+} \left | (P-P_n) d_{\phi}(u,\cb) \mathbbm{1}_{B_{\phi}(\cb,r)}(u) \right | \\ \vee \left | (P-P_n) d_{\phi}(u,\cb) \mathbbm{1}_{\partial B_{\phi}(\cb,r)}(u) \right | \leq \alpha_n,
\end{multline*}
where $\alpha_n = O(\sqrt{\log(n)/n})$. We then work on the global probability event on which all these deviation inequalities are satisfied, that have probability larger than $1-\frac{1}{n^{\frac{p}{2}}}$, according to Proposition \ref{prop:deviations} and Lemma \ref{lm: Marcinkiewicz}. 
We let $\alpha_1>0$ be such that $\min_{j\in[\![2,k]\!]} R^*_{k-1,h_{b_1}^-} - R^*_{k,h_{b_1}^+} \geq \alpha_1$, according to Lemma \ref{lm: key lemma h- h+}, and let $b < B_h$ such that $s/(n+s) \leq b = \kappa_1b_1$.
Let $\hat{\cb}_{n+s,h}$ denote an $h$-trimmed empirical risk minimizer based on $\{X_1, \dots, X_n, x_{n+1}, \dots, x_{n+s}\}$, and $\cb^*_{h_{b_1}^+}$ a $h_{b_1}^+$-trimmed optimal codebook. Then
    \begin{align*}
    \hat{R}_{n+s,h} (\hat{\cb}_{n+s,h}) & \leq \hat{R}_{n+s,h}(\cb^*_{h_{b_1}^+}) 
      \leq \frac{1}{n+s} \left [ \sum_{i=1}^{n} d_\phi(X_i,\cb^*_{h_{b_1}^+}) \hat{\tau}_{h_b^+}(\cb^*_{h_{b_1}^+})(X_i) \right ],
    \end{align*}
 since $ (n+s) h \leq  nh_b^+ < nh_{b_1}^+ \leq n$. We may write
 \begin{align*}
 \hat{R}_{n+s,h} (\hat{\cb}_{n+s,h}) & \leq \frac{n}{n+s} \left ( P_n d_\phi(u,\cb^*_{h_{b_1}^+}) \tau_{h_b^+ + \beta_n}(\cb^*_{h_{b_1}^+})(u)  \right ) \\ 
                   & \leq \frac{n}{n+s} \left ( P_n d_\phi(u,\cb^*_{h_{b_1}^+}) \tau_{h_{b_1}^+}(\cb^*_{h_{b_1}^+})(u) \right ) \\
                   & \leq \frac{n}{n+s} \left ( P  d_\phi(u,\cb^*_{h_{b_1}^+}) \tau_{h_{b_1}^+}(\cb^*_{h_{b_1}^+})(u) + \alpha _n \right ) \\
                   & \leq \frac{n}{n+s} \left ( R^*_{h_{b_1}^+} + \alpha_n \right ), 
 \end{align*}
 for $n$ large enough. 
 Now assume that $ P_n \left ( \hat{\tau}_{h_b^-}(\hat{\cb}_{n+s,h})\1_{W_1(\hat{\cb}_{n+s,h})} \right )< \beta_1$.
 Then, 
 \begin{align*}
  \hat{R}_{n+s,h} (\hat{\cb}_{n+s,h}) & \geq \frac{n}{n+s} \hat{R}_{n,h_b^-}(\hat{\cb}_{n+s,h}),
 \end{align*}
 since $n - (n+s)(1-h) \geq  n (1-h_b^-)$. Thus, removing one quantization point,
 \begin{align*}
 \hat{R}_{n+s,h} (\hat{\cb}_{n+s,h}) & \geq \frac{n}{n+s}P_n \left [ d_\phi(u,\hat{\cb}_{n,h_b^- - \beta_1}^{(k-1)}) \hat{\tau}_{h_b^- - \beta_1}(\hat{\cb}_{n,h_b^- - \beta_1}^{(k-1)})(u) \right ] \\
 & \geq \frac{n}{n+s} P_n  \left [ d_\phi(u,\hat{\cb}_{n,h_b^- - \beta_1}^{(k-1)}) \tau_{h_{b_1}^-}(\hat{\cb}_{n,h_b^- - \beta_1}^{(k-1)})(u) \right ],
 \end{align*}
 where $\hat{\cb}_{n,h_b^- - \beta_1}^{(k-1)}$ denotes a $h_b^- - \beta_1$-trimmed empirical risk minimizer with $k-1$ codepoints. Since $h_b^- - \beta_1 \geq h_b^- - 2\beta_1 \geq h_{b_1^-}$,  Proposition \ref{prop:bounded_empirical_codebooks} implies 
 \begin{align*}
 \hat{R}_{n+s,h} (\hat{\cb}_{n+s,h}) \geq &\frac{n}{n+s} P_n  \left [ d_\phi(u,\hat{\cb}_{n,h_b^- - \beta_1}^{(k-1)}) \tau_{h_{b_1}^-}(\hat{\cb}_{n,h_b^- - \beta_1}^{(k-1)})(u) \right ] \\
  & \geq \frac{n}{n+s} P \left [ d_\phi(u,\hat{\cb}_{n,h_b^- - \beta_1}^{(k-1)}) \tau_{h_{b_1}^-}(\hat{\cb}_{n,h_b^- - \beta_1}^{(k-1)})(u) - \alpha_n \right ].  
 \end{align*}
 Thus, $\hat{R}_{n+s,h} (\hat{\cb}_{n+s,h}) \geq  n ( R^*_{k-1,h_{b_1}^-} + \alpha_n  )/(n+s)$ hence the contradiction for $2 \alpha_n < \alpha_1$.
\end{proof}

\section{Supplementary material for Section \ref{sec:numerical_experiments}}\label{tecsec:supp_author}
\subsection{Additional files for the comparison of Bregman clusterings for mixtures with noise}

\subsubsection{Details on the different clustering procedures}
\label{sec: Details sur les procedures de clustering}

In Section \ref{sec: Comparative performances of  Bregman clustering for mixtures with noise}, we compared our trimmed Bregman procedures with the following clustering schemes : trimmed $k$-median \cite{Cardot13}, \texttt{tclust} \cite{Fritz12}, single linkage, ToMATo \cite{Chazal_Oudot} and \texttt{dbscan} \cite{DBSCAN}. 
Trimmed $k$-median denotes the $k$-median clustering trimmed afterwards. Actually, we keep the $q=110$ points which $l_1$-norm to their center is the smallest. In order to compute the centers, we use the function \texttt{kGmedian} from the R package \texttt{Gmedian}, with parameters $gamma = 1$, $alpha = 0.75$ and $nstart = nstartkmeans = 20$. For tclust, we use the function \texttt{tclust} from the R package \texttt{tclust} with parameters $k = 3$ (number of clusters) and $alpha = 10/120$, the proportion of points to consider as outliers. We use the C++ ToMATo algorithm, available at https://geometrica.saclay.inria.fr/data/ToMATo/. We compute the inverse of the distance-to-measure function \cite{Chazal11} with paramter $m0 = 10/120$ (that can be considered as a density) at every sample point, and keep the 110 highest valued points. We use the first parameter 5 (the radius for the Rips graph built from the resulting sample points) and the second parameter 0.01 (related to the number of clusters). For the single linkage method, we first keep the 110 points with the smallest distance to their 10th nearest neighbor, then, cluster points according to the R functions \texttt{hclust} with the method ``single'' and \texttt{cutree} with parameter $h = 4$ (related to the number of clusters). For the dbscan method, we use the \texttt{dbscan} function from the R package \texttt{dbscan}. We set the parameters $eps $ to the 110-th smallest distance to a third nearest neighbor among points in the sample, $minPts = 3$ and $borderPoints = FALSE$. 

For these three last methods, we cannot calibrate the parameters so that the algorithms return 3 clusters, because of the systemmatic presence of many additional small clusters.

\subsubsection{Clustering for 12000 sample points}
\label{sec: Les echantillons de taille 12000}

This section exposes additional experimental results. We proceed exactly like in Section \ref{sec: Comparative performances of  Bregman clustering for mixtures with noise}, but with samples made of 10000 signal points and 2000 noise points. 

\begin{figure}[H]
	\begin{minipage}[h]{.3\linewidth}
	Gaussian
		\centering\includegraphics[scale=0.16]{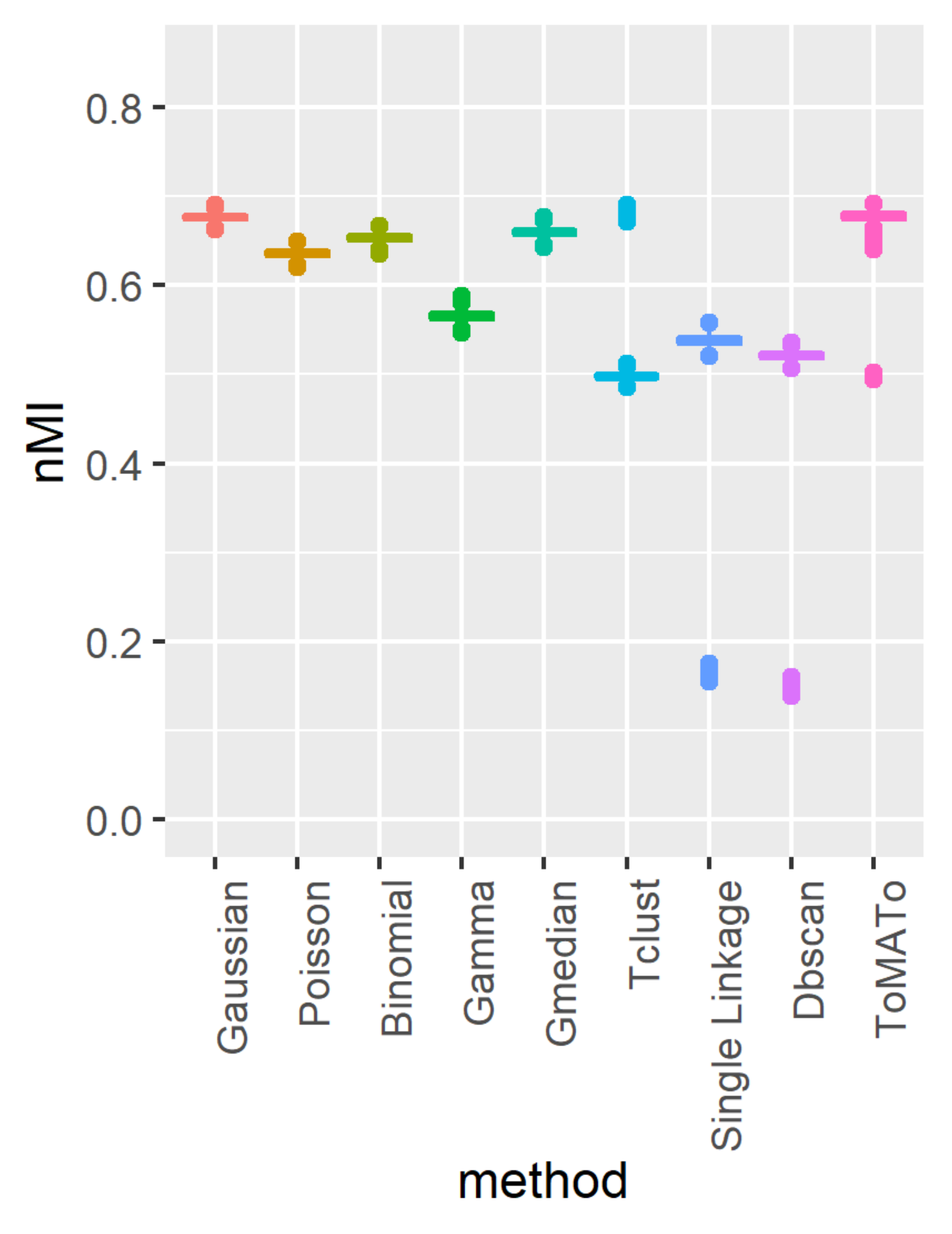}
		
	\end{minipage} 
	\begin{minipage}[h]{.3\linewidth}
	Poisson
\centering\includegraphics[scale=0.16]{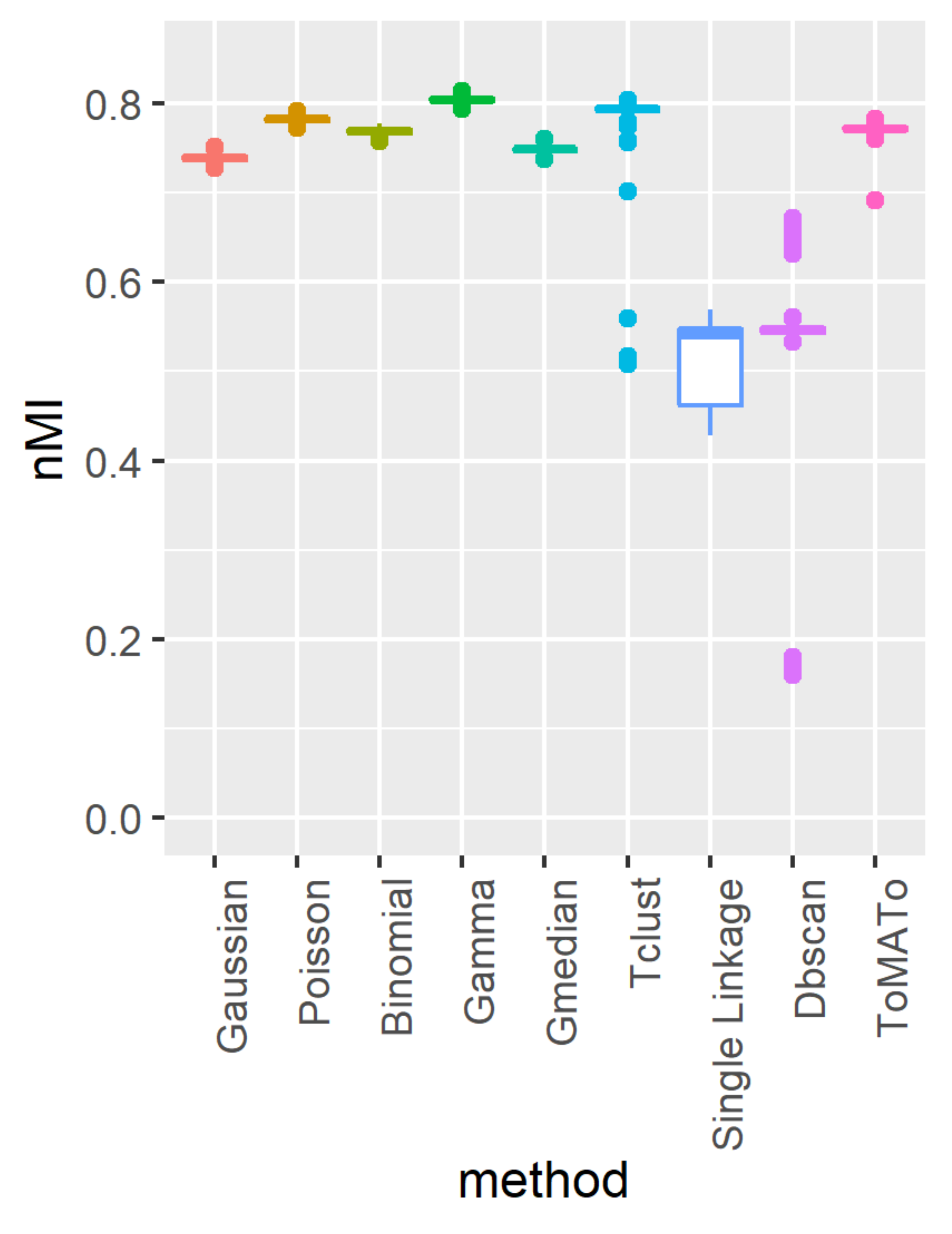}
		
	\end{minipage}
		\begin{minipage}[h]{.38\linewidth}
		Binomial
		\centering\includegraphics[scale=0.16]{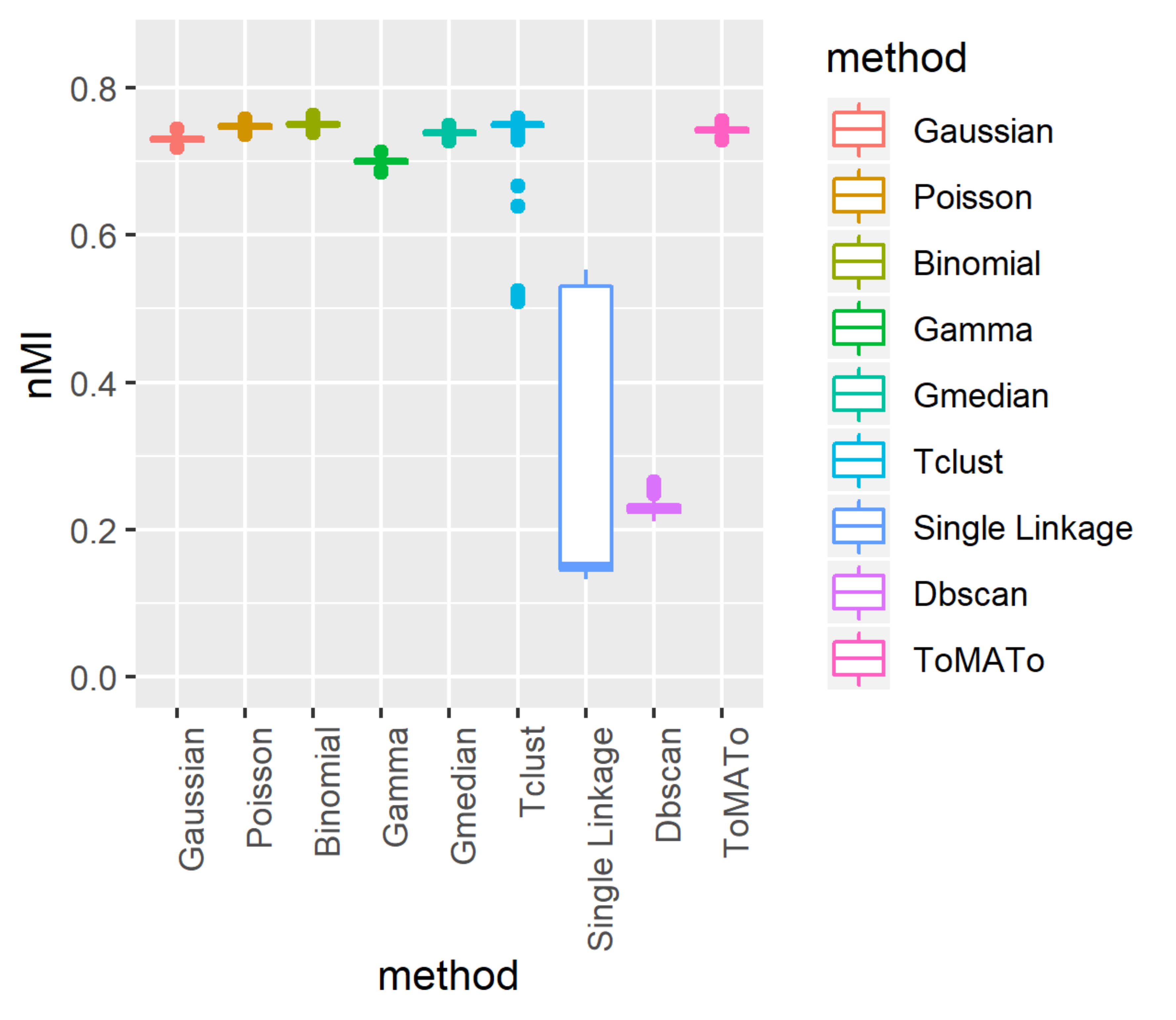}
		
	\end{minipage} 
	\begin{minipage}[h]{.3\linewidth}
	Gamma
\centering\includegraphics[scale=0.16]{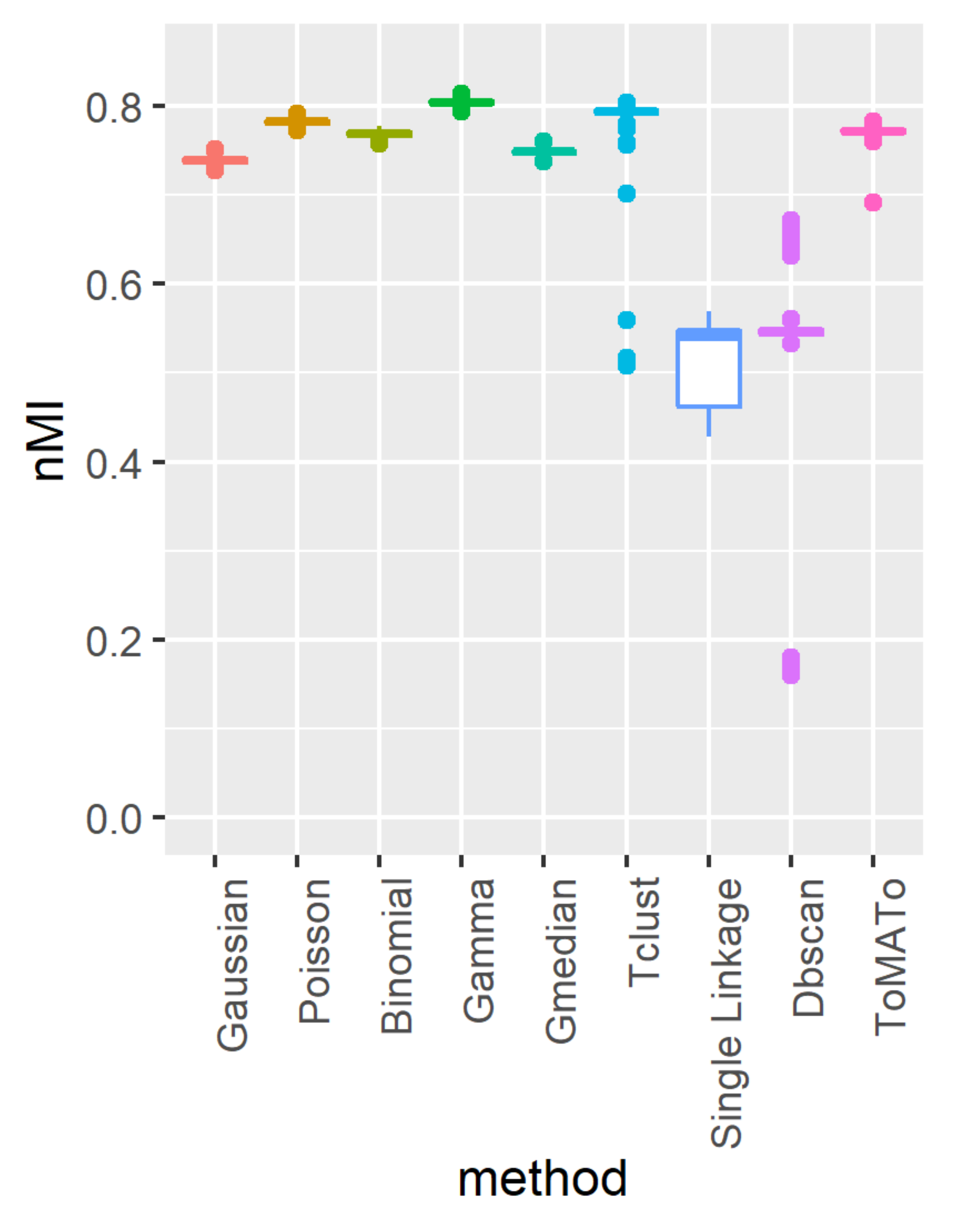}
		
	\end{minipage}
		\begin{minipage}[h]{.3\linewidth}
		Cauchy
		\centering\includegraphics[scale=0.16]{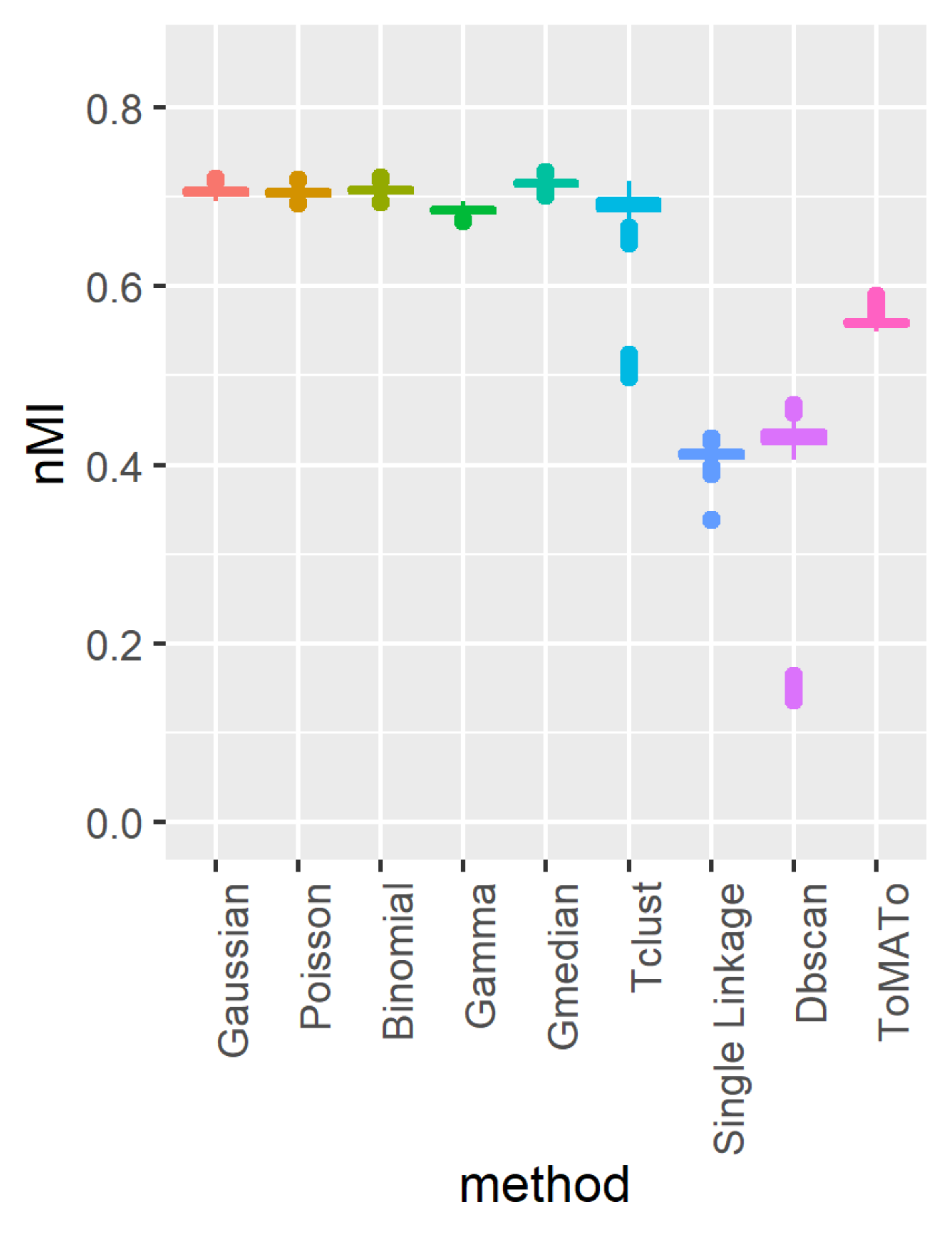}
		
	\end{minipage} 
	\begin{minipage}[h]{.38\linewidth}
	Mixture
\centering\includegraphics[scale=0.16]{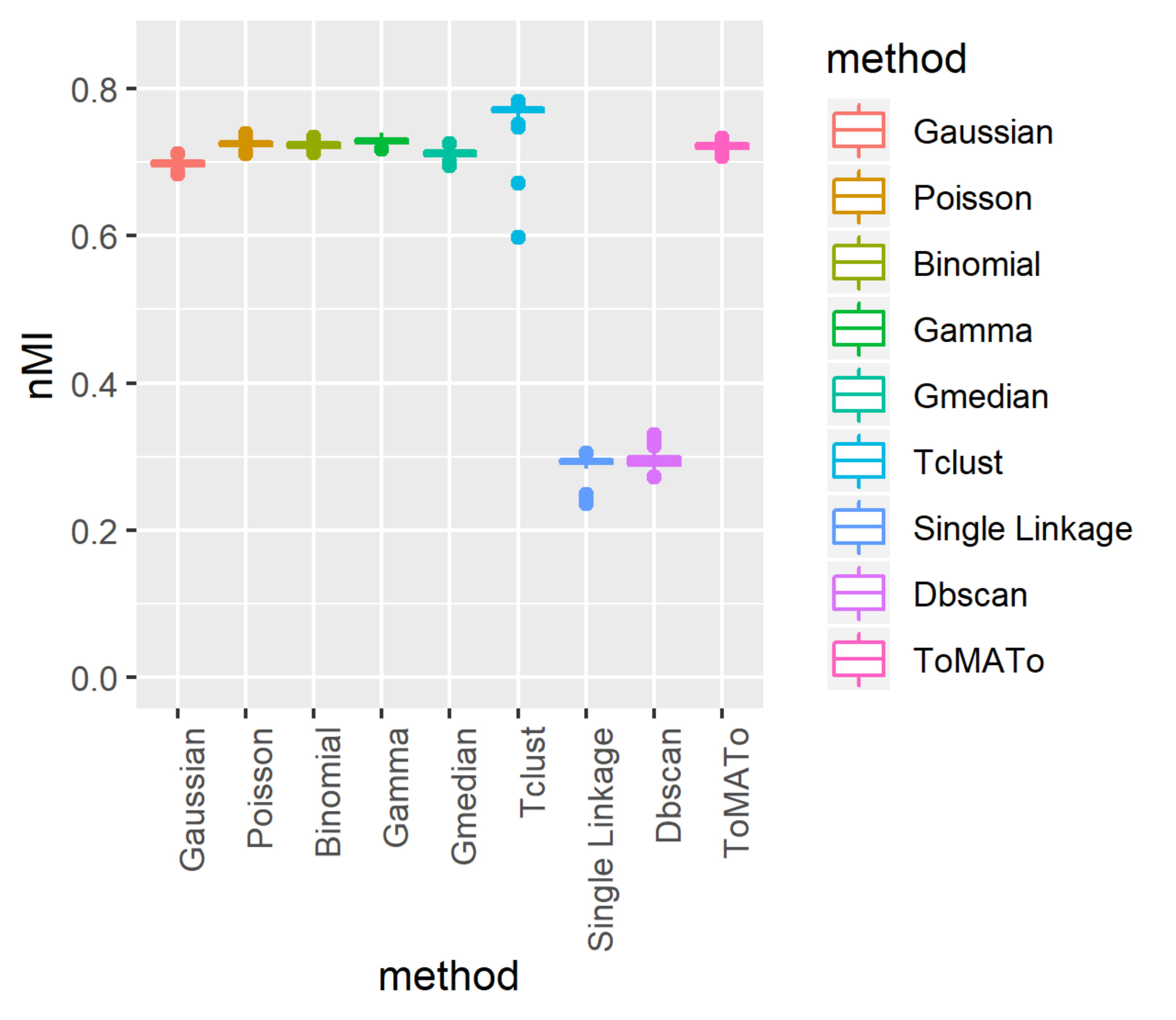}
		
	\end{minipage}
\caption{NMI for samples of 12000 points\label{fig: boxplot 10000 points}}
\end{figure}

For each clustering procedure, we decide to consider 11000 points as signal. 
The parameters for the different procedures are the same as described in Section \ref{sec: Details sur les procedures de clustering}, except for the ToMATo algorithm, we set the parameter $m0 = 200/12000$. As well, the number of nearest neighbors for the single linkage method is set to 200 and the parameter $h$ is set to 0.8 for Gaussian distribution, 0.8 for Poisson, 1.1 for Binomial, 0.6 for Gamma, 0.4 for Cauchy and 0.4 for the mixture of 3 different distributions. For dbscan, $eps$ is the 11000-th smallest distance of a point to its third nearest neighbor.

The NMI over 1000 replications of the experiments are represented via boxplots in Figure \ref{fig: boxplot 10000 points}. Algorithm \ref{algo:BTKM} with the proper Bregman divergence systematically (slightly) outperforms other clustering schemes.

\subsection{Discussion about the choice of the Bregman divergence}
\label{sec: Discussion about the choice of the Bregman divergence}

We consider three mixtures of Gaussian distributions $\Lcal(c,\sigma) = \frac13\Ncal(c_1,\sigma_1I_2)+\frac13\Ncal(c_2,\sigma_2I_2)+\frac13\Ncal(c_3,\sigma_3I_2)$ with $c = (c_1,c_2,c_3)$ for $c_1 = (10,10)$, $c_2 = (25,25)$ and $c_3 = (40,40)$, $I_2$ the identity matrix on $\R^2$ and $\sigma = (\sigma_1,\sigma_2,\sigma_3)$. The first distribution $\Lcal_1$ corresponds to clusters with the same variance, with $\sigma = (5,5,5)$, the second distribution $\Lcal_2$ to clusters with increasing variance, with $\sigma = (1,4,7)$, and the third distribution $\Lcal_3$ to clusters with increasing and decreasing variance, with $\sigma = (2,7,2)$. We cluster samples of 100 points from $\Lcal_1$, $\Lcal_2$ and then $\Lcal_3$. We use Algorithm \ref{algo:BTKM} with the Gaussian, Poisson, Binomial and Gamma Bregman divergences. Note that we set $N = 50$ for the Binomial divergence, so that we expect a clustering with clusters size symmetric with respect to 25.  The performance of the clustering in terms of NMI is represented in Figure \ref{fig:boxplot NMI}, after 1000 replications of the experiments. 

\begin{figure}[H]
\centering\includegraphics[scale = 0.4]{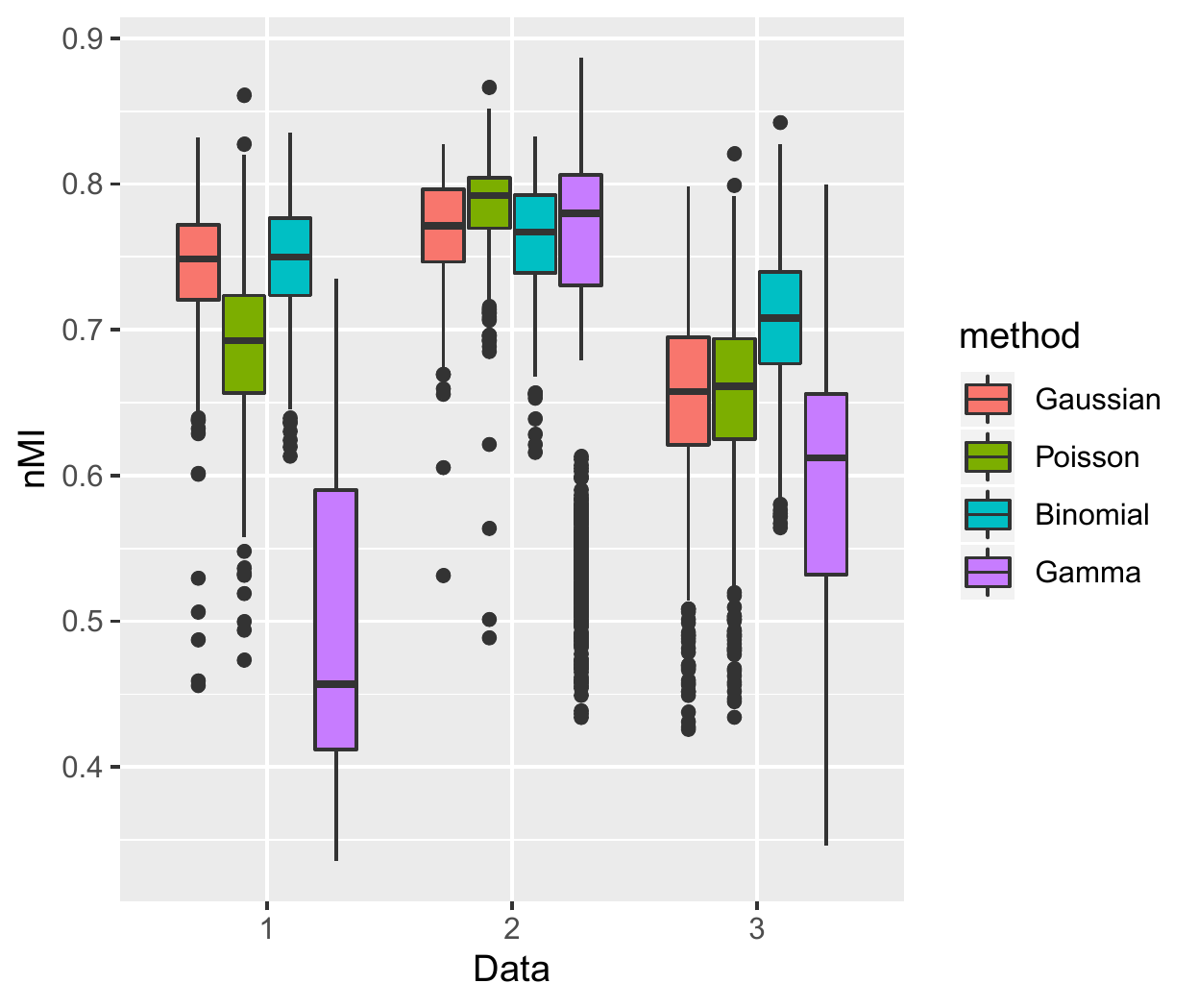}
\caption{Comparison of Bregman divergences efficiency for different clusters variances.}
\label{fig:boxplot NMI}
\end{figure}

The corresponding clustering with the best suited Bregman divergence is represented in Figure \ref{fig:Clustering_best div}. In particular, we used the Gaussian divergence for $\Lcal_1$, the Poisson divergence for $\Lcal_2$ and the Binomial divergence for $\Lcal_3$.

\begin{figure}[H]
	\begin{minipage}[h]{.31\linewidth}
		\centering\includegraphics[scale=0.12]{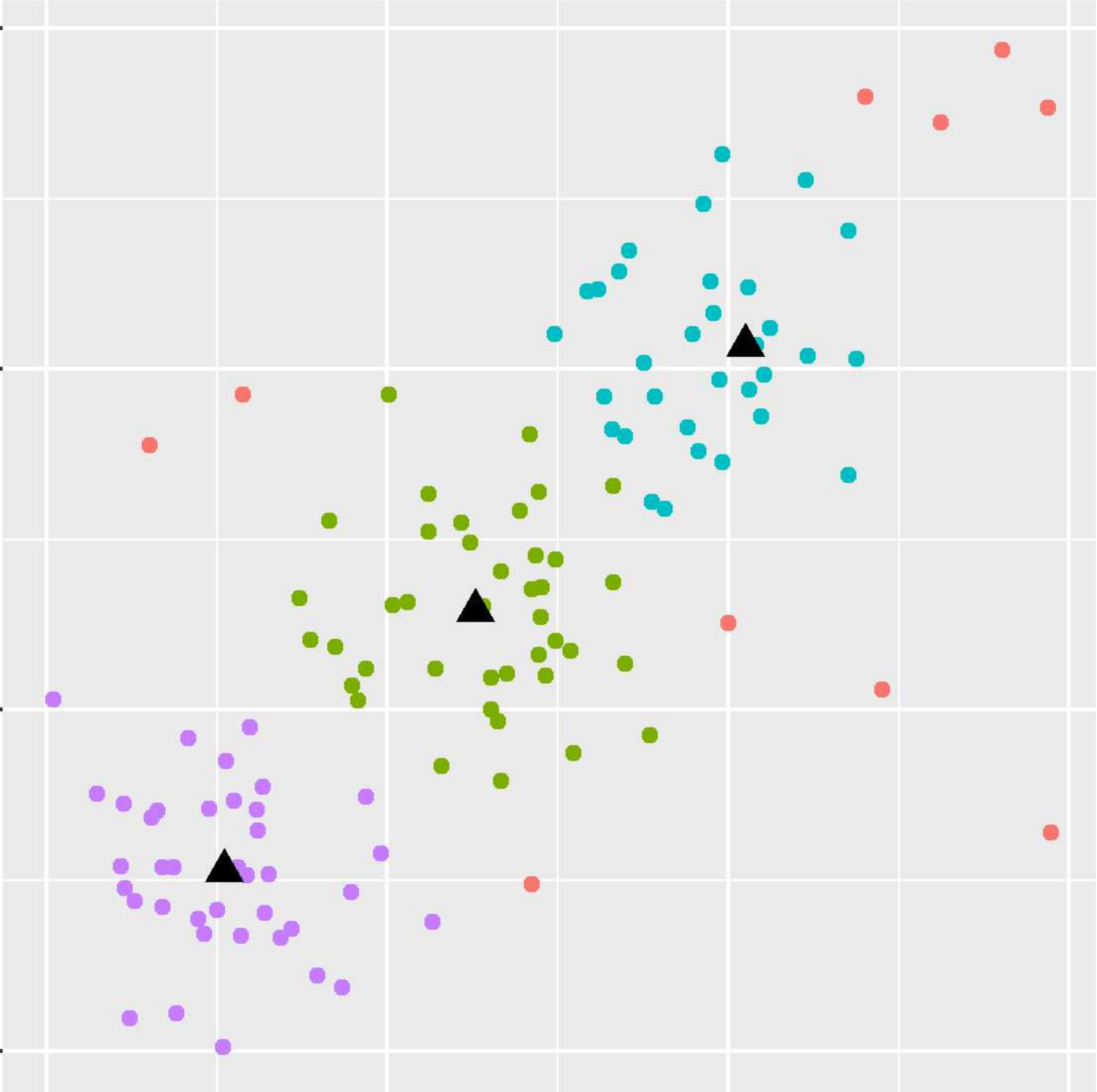}\\
		$\Lcal_1$ -- Gaussian divergence
	\end{minipage}  \hfill
	\begin{minipage}[h]{.31\linewidth}
		\centering\includegraphics[scale=0.12]{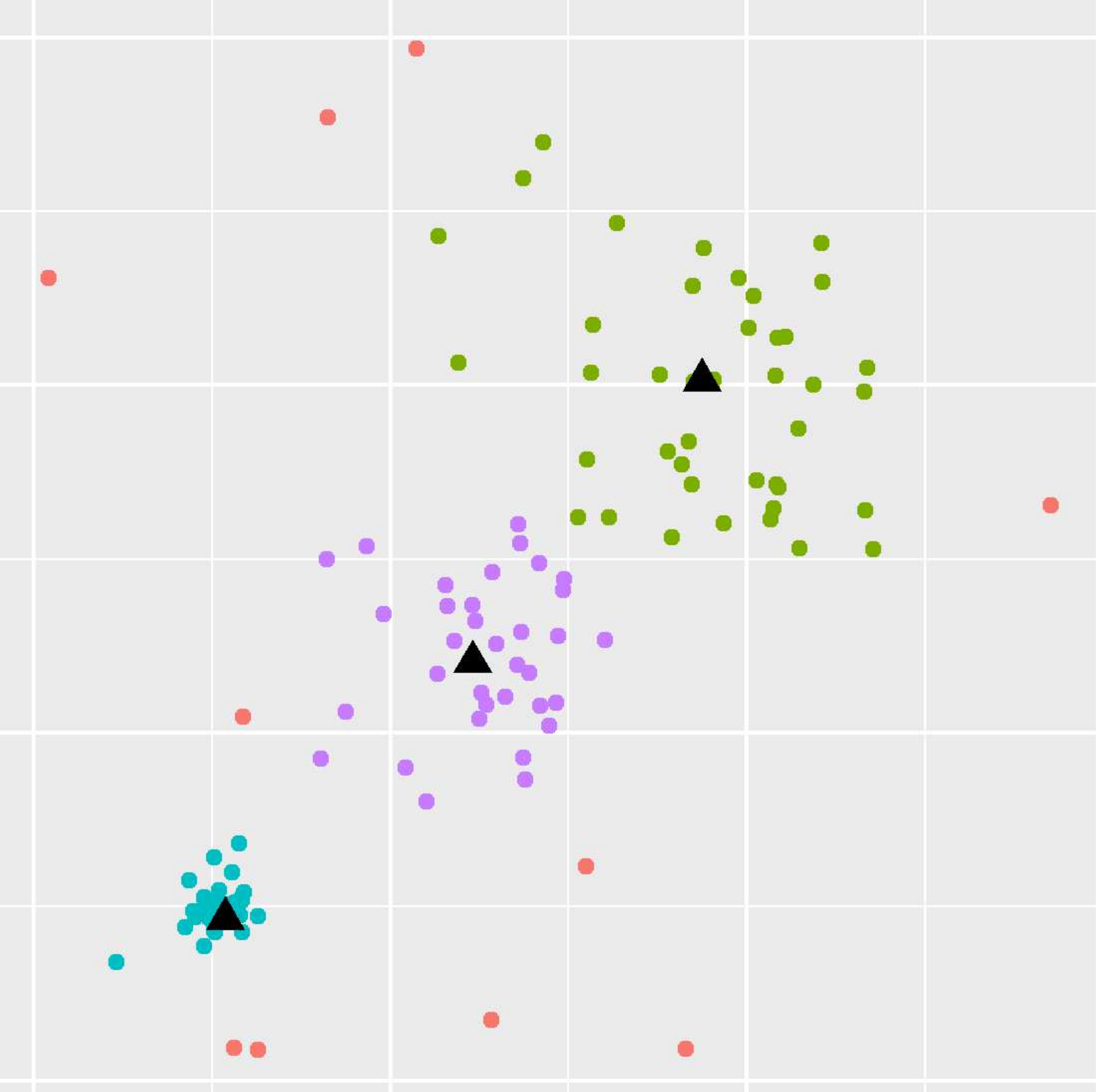}\\
		$\Lcal_2$ -- Poisson divergence
	\end{minipage} \hfill
	\begin{minipage}[h]{.34\linewidth}
		\centering\includegraphics[scale=0.12]{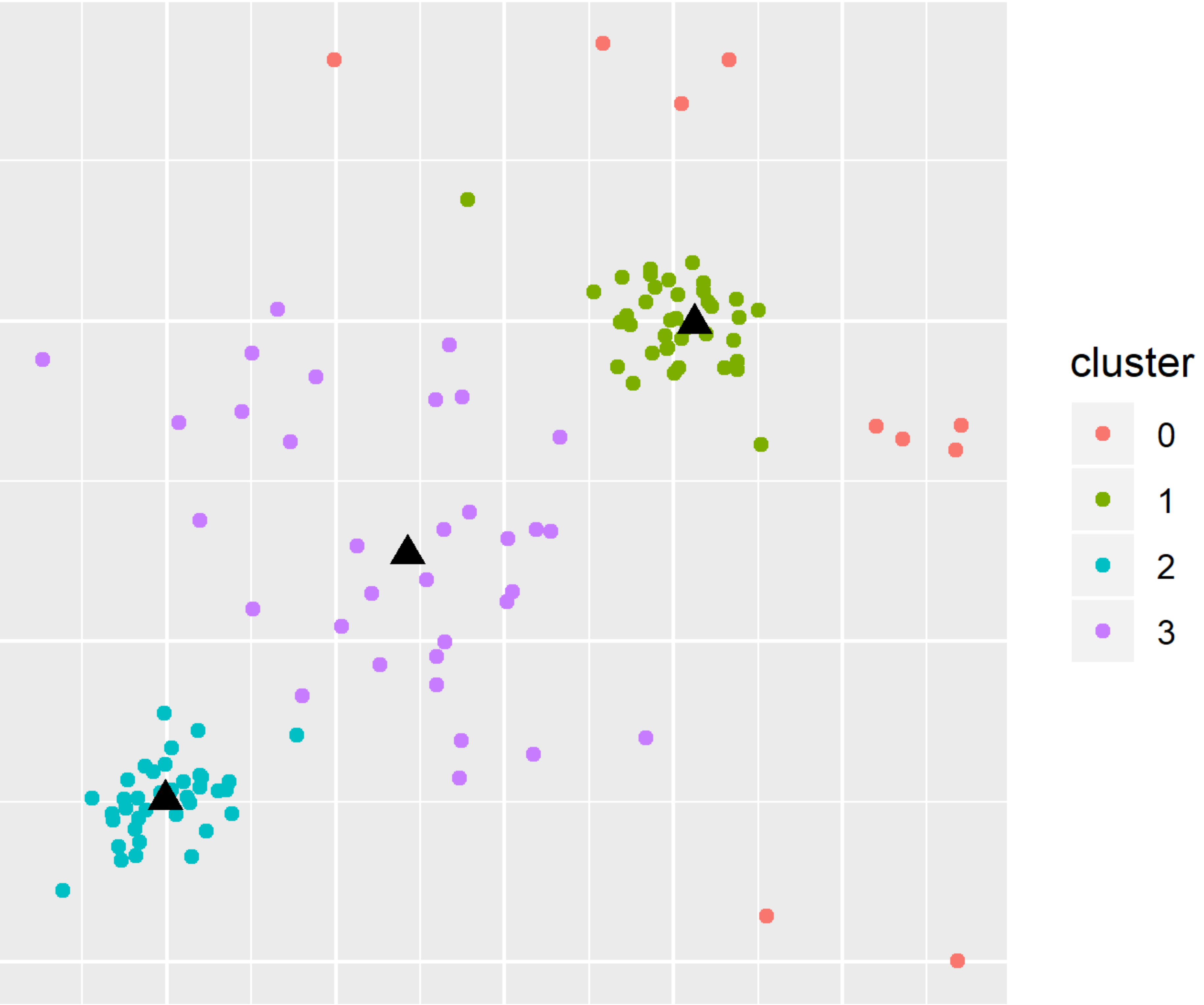}\\
		$\Lcal_3$ -- Binomial divergence
	\end{minipage}
	\caption{Clustering with the best suited Bregman divergence\label{fig:Clustering_best div}}
\end{figure}

Since the sum of two Bregman divergences is a Bregman divergence, it is also possible to cluster data with Algorithm \ref{algo:BTKM}, with a different divergence on the different coordinates. For instance, we sample 100 points from $\frac13\Ncal(c_1,\Sigma_1)+\frac13\Ncal(c_2,\Sigma_2)+\frac13\Ncal(c_3,\Sigma_3)$, with for every $i\in[\![1,3]\!]$, $\Sigma_i$ diagonal with coefficients $(\sigma^{(2)}_i,\sigma^{(3)}_i)$ with $\sigma^{(2)} = (1,4,7)$ and $\sigma^{(3)} = (2,7,2)$. In Figure \ref{Clustering with hybrid Bregman divergence}, we represented the clustering obtained with Algorithm \ref{algo:BTKM} with the Poisson divergence on the first coordinate and the Binomial divergence on the second coordinate. We observe that the shape of the clusters obtained correspond roughly to the shape of the sublevel sets of the sampling distribution. 

\begin{figure}[H]
\centering\includegraphics[scale = 0.12]{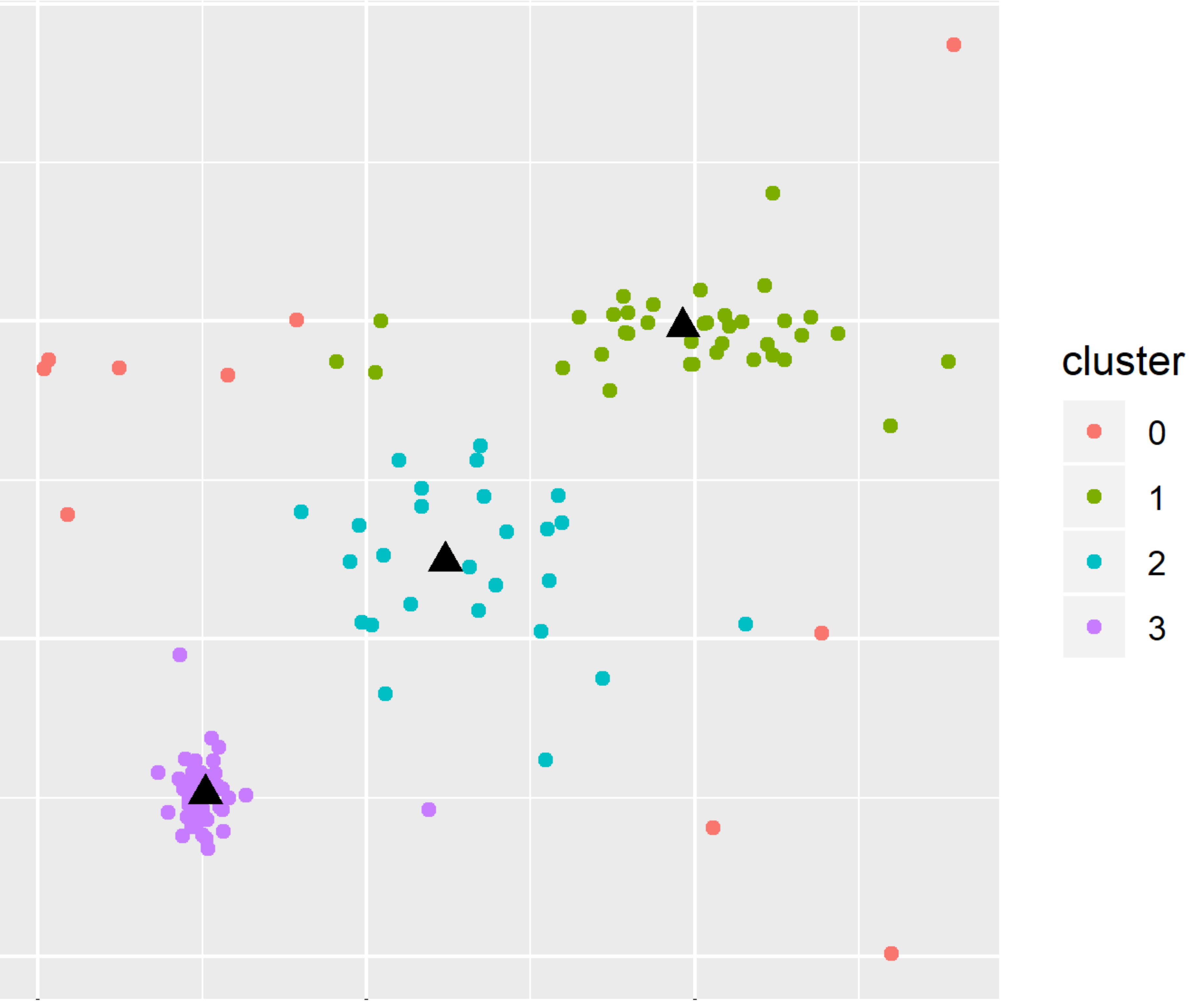}
\caption{Clustering with hybrid Bregman divergence\label{Clustering with hybrid Bregman divergence}}
\end{figure}

\subsection{Additional files for Stylometric author clustering}
This section exposes the graphics and additional numerics that support several results from Section \ref{sec:author_clustering}, for instance about the calibration of parameters.

\noindent\textbf{Trimmed $k$-median}:

In Figure \ref{Author_median_cost_NMI} we plot the cost and the NMI as a function of $q$ for different numbers of clusters $k$, in Figure \ref{Author_median_cost_NMI_k4} we focus on the case $k = 4$. Finally, in Figure \ref{Author_median_clusterings} we plot the best clusterings in terms of NMI for $k=4$ and $k=6$.
\begin{figure}[H]
	\begin{minipage}[h]{.49\linewidth}
		\centering\includegraphics[scale=0.45]{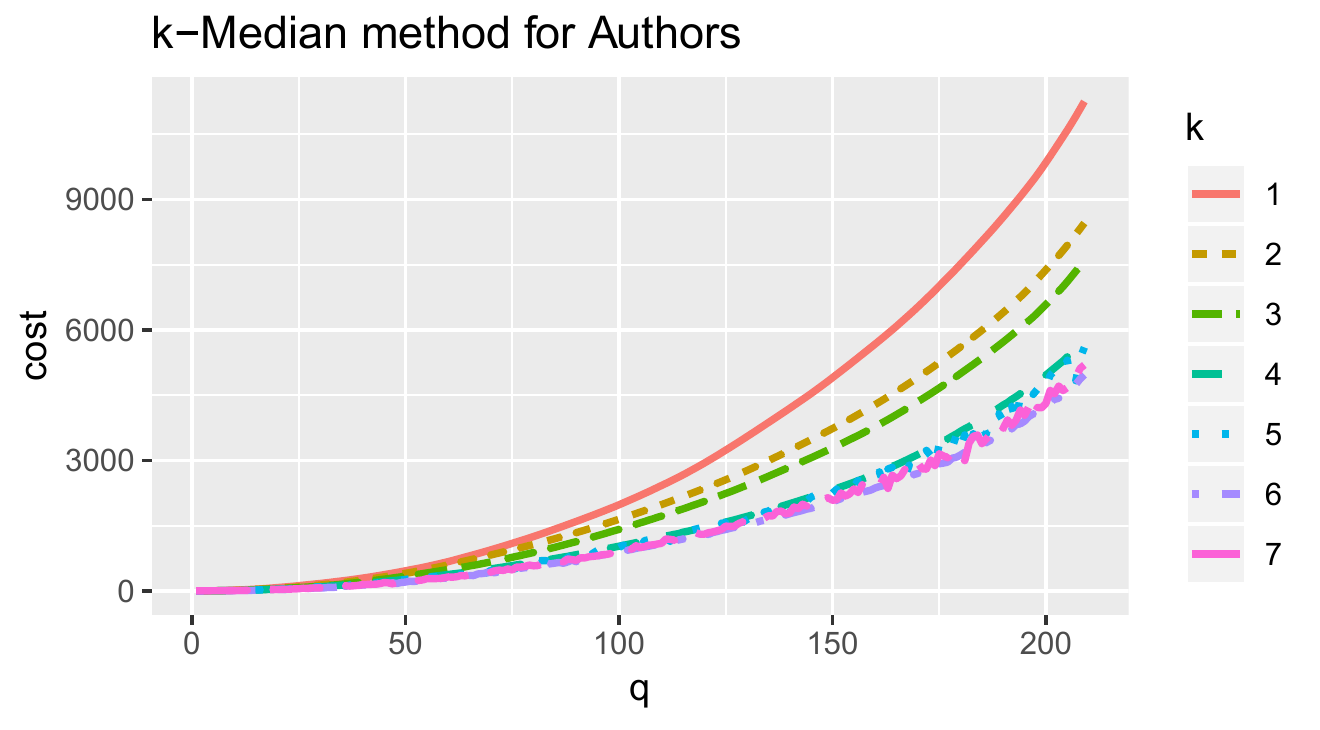}\\
	\end{minipage}
	\begin{minipage}[h]{.49\linewidth}
\centering\includegraphics[scale=0.45]{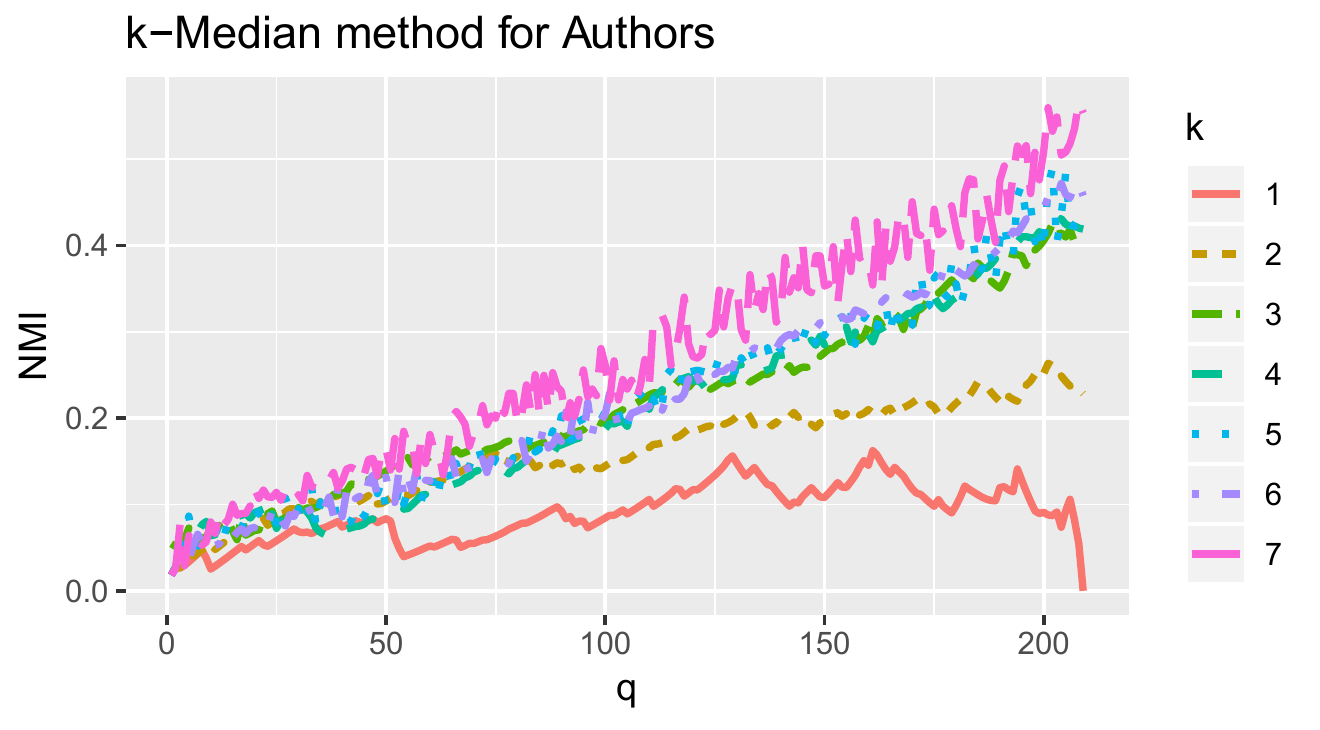}
	\end{minipage}
\caption{Cost and NMI for Author clustering with $k$-Median method\label{Author_median_cost_NMI}}
\end{figure}

\begin{figure}[H]
	\begin{minipage}[h]{.49\linewidth}
		\centering\includegraphics[scale=0.45]{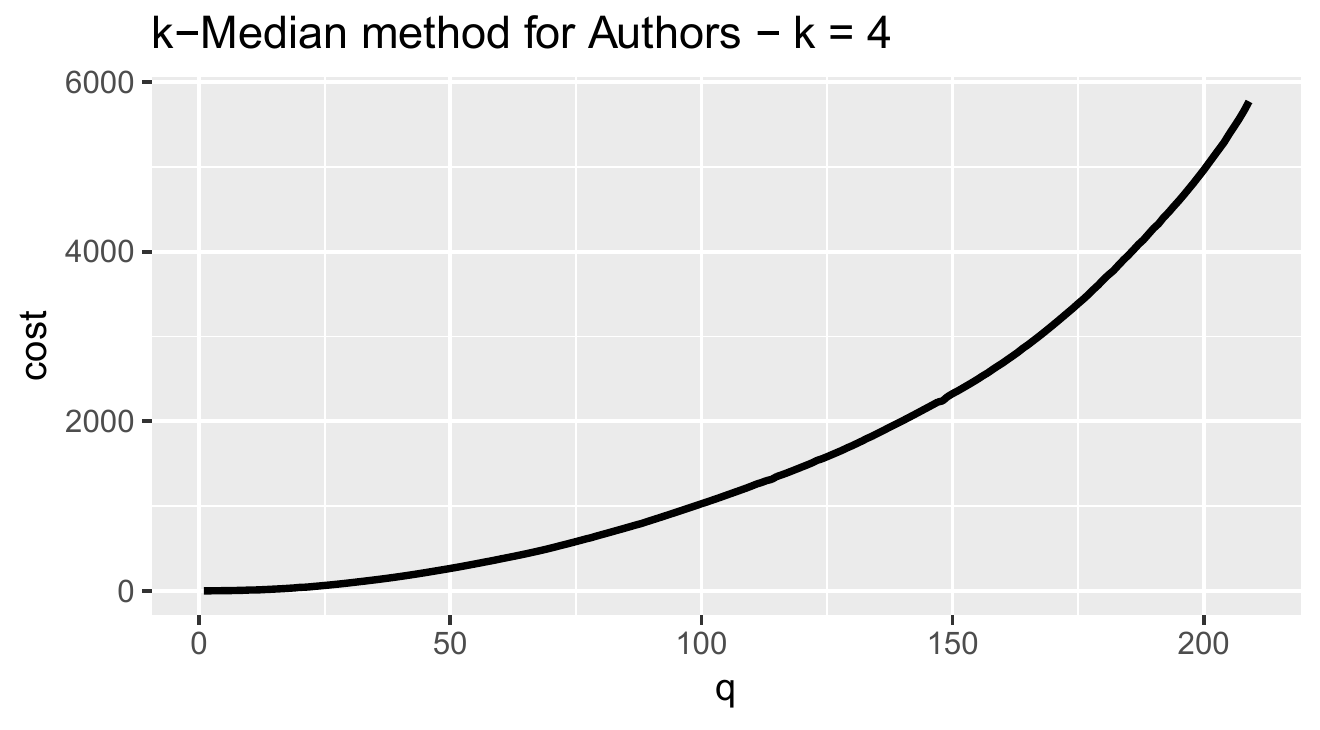}\\
	\end{minipage} 
	\begin{minipage}[h]{.49\linewidth}
\centering\includegraphics[scale=0.45]{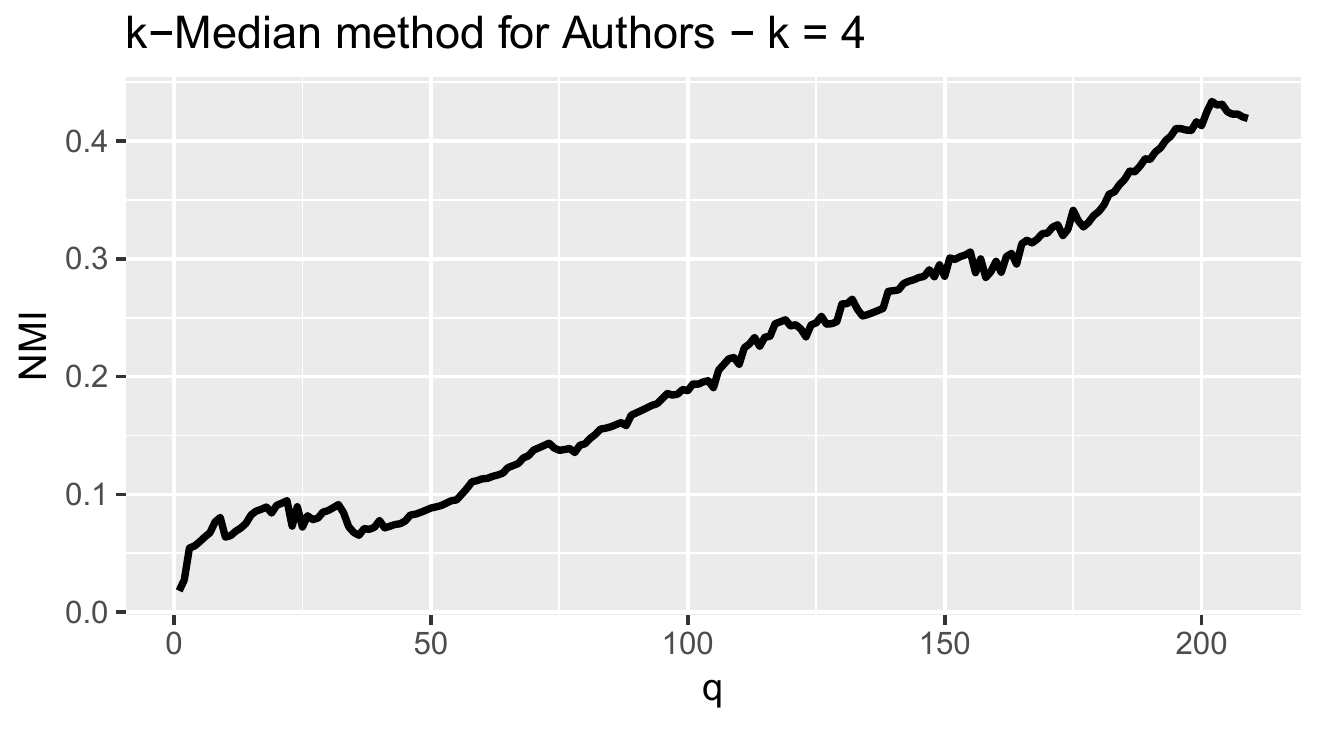}
	\end{minipage}
\caption{Cost and NMI for Author clustering with $k$-Median method -- $k=4$\label{Author_median_cost_NMI_k4}}
\end{figure}

\begin{figure}[H]
	\begin{minipage}[h]{.49\linewidth}
		\centering\includegraphics[scale=0.2]{Author_median_k4_q202.pdf}\\
	\end{minipage}
	\begin{minipage}[h]{.49\linewidth}
\centering\includegraphics[scale=0.2]{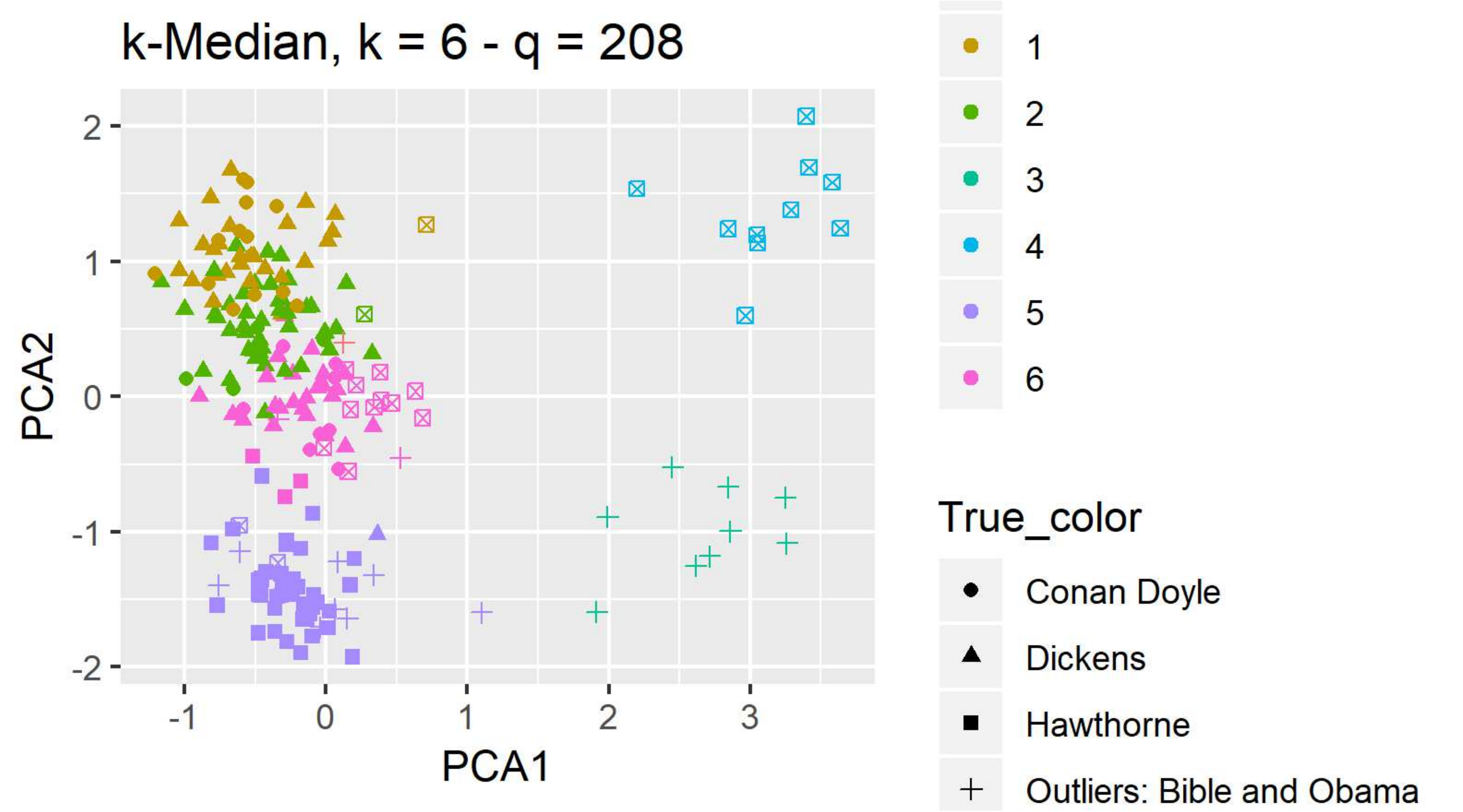}
	\end{minipage}
\caption{Examples of Author clusterings obtained with $k$-Median method\label{Author_median_clusterings}}
\end{figure}
These graphics suggest that $k=4$ and $k=6$ are possible choices. The corresponding $q$ that minimize NMI's are respectively $q = 202$  ($NMI = 0.4334372$), and $q = 208$  ($NMI = 0.4721967$).

\noindent\textbf{tclust}:

Figure \ref{Author_tclust_cost_NMI} and \ref{Author_tclust_cost_NMI_k4} do not allow to select $k$. If $k=4$ is chosen, Figure \ref{Author_tclust_cost_NMI_k4} suggests that $q\simeq 184$  is a relevant choice. Figure \ref{Author_tclust_clusterings} provides the associated clustering, whose $NMI$ is $0.4912537$.
\begin{figure}[H]
	\begin{minipage}[h]{.49\linewidth}
		\centering\includegraphics[scale=0.45]{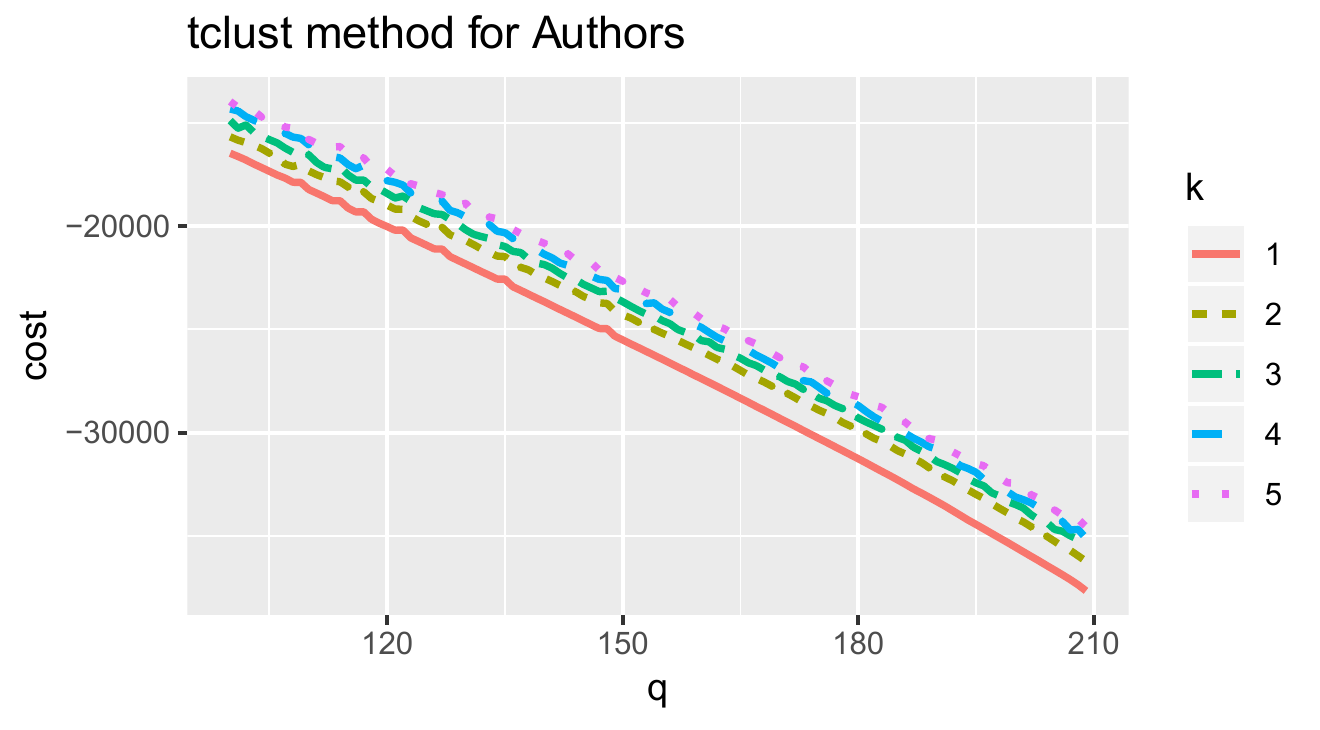}\\
	\end{minipage}
	\begin{minipage}[h]{.49\linewidth}
\centering\includegraphics[scale=0.45]{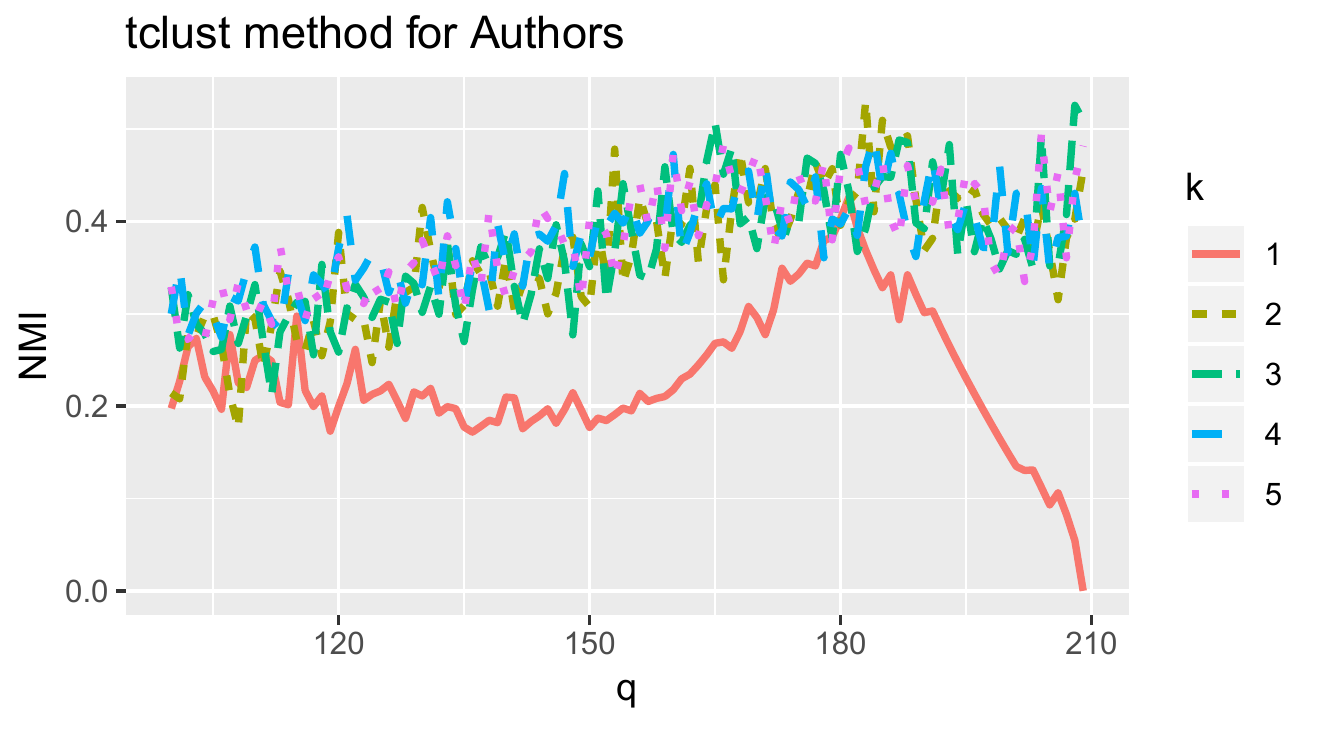}
	\end{minipage}
\caption{Cost and NMI for Author clustering with tclust algorithm\label{Author_tclust_cost_NMI}}
\end{figure}

\begin{figure}[H]
	\begin{minipage}[h]{.49\linewidth}
		\centering\includegraphics[scale=0.45]{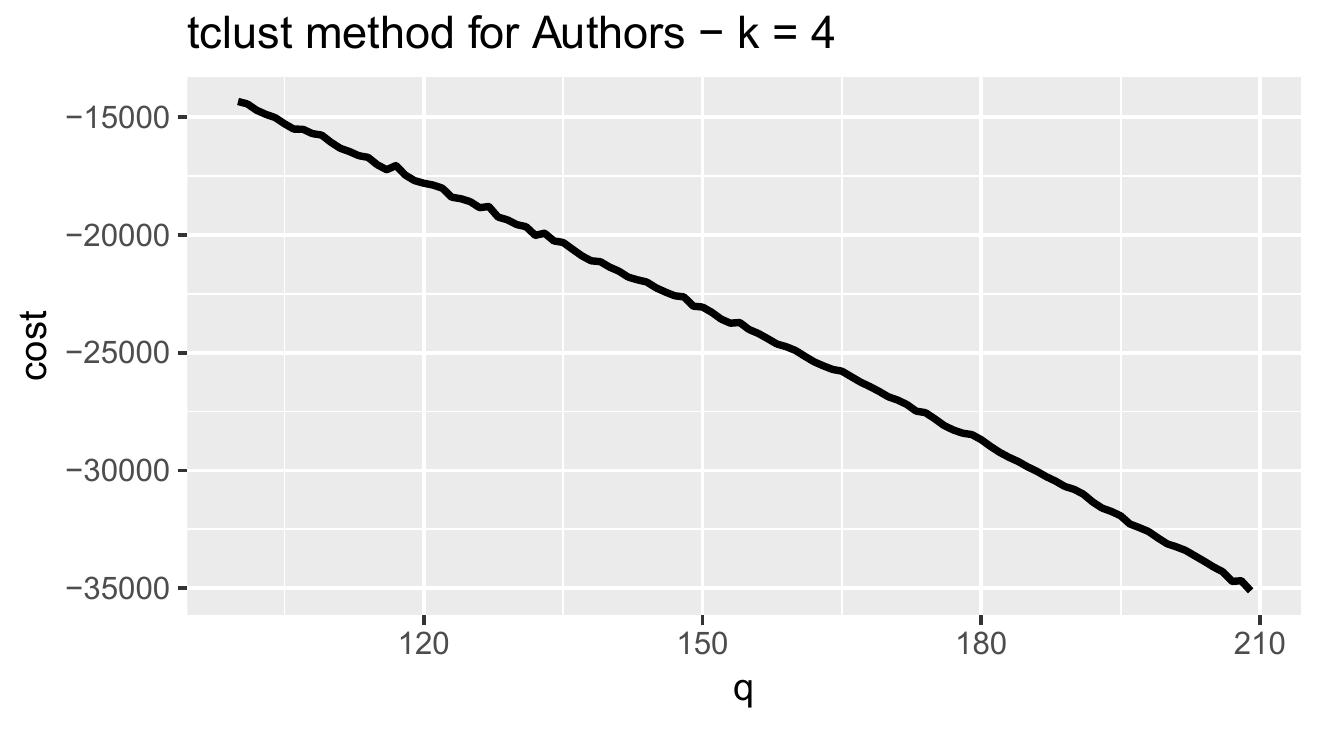}\\
	\end{minipage} 
	\begin{minipage}[h]{.49\linewidth}
\centering\includegraphics[scale=0.45]{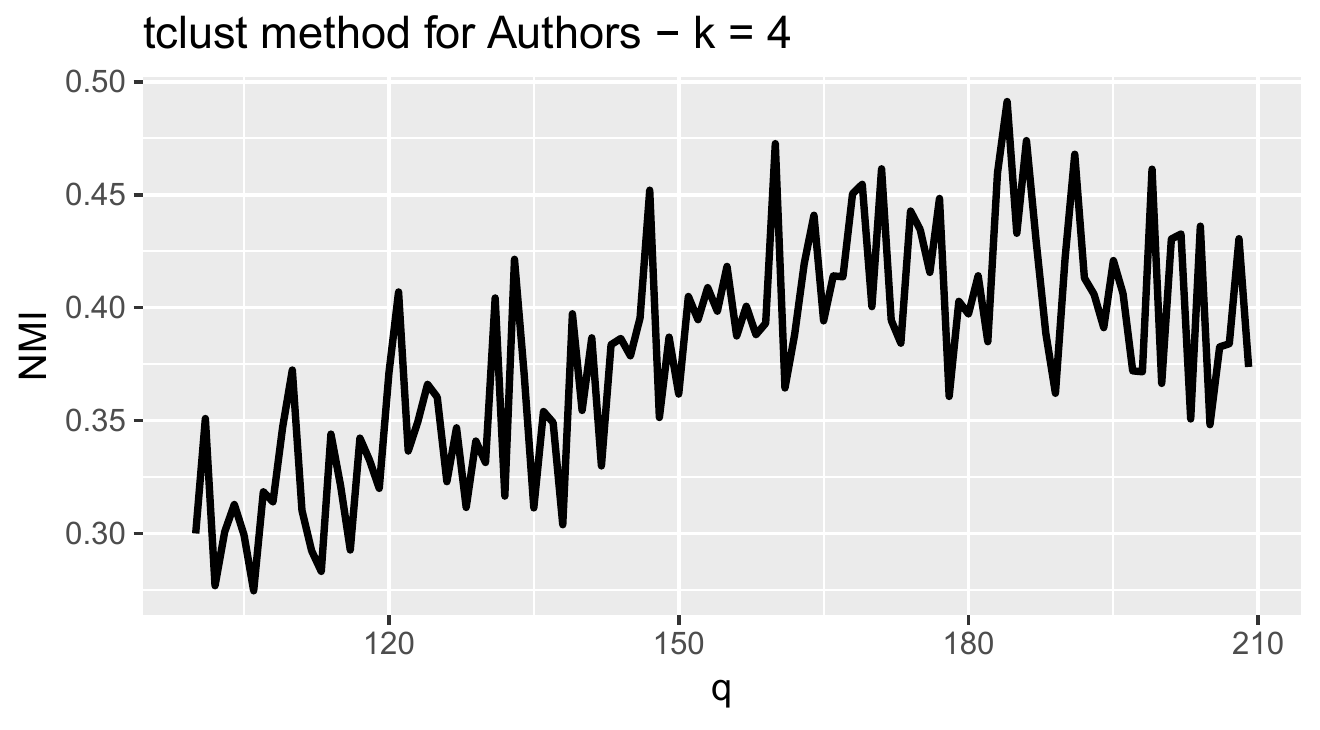}
	\end{minipage}
\caption{Cost and NMI for Author clustering with tclust algorithm -- $k=4$\label{Author_tclust_cost_NMI_k4}}
\end{figure}

\begin{figure}[H]
	\centering\begin{minipage}[h]{.49\linewidth}
		\includegraphics[scale=0.2]{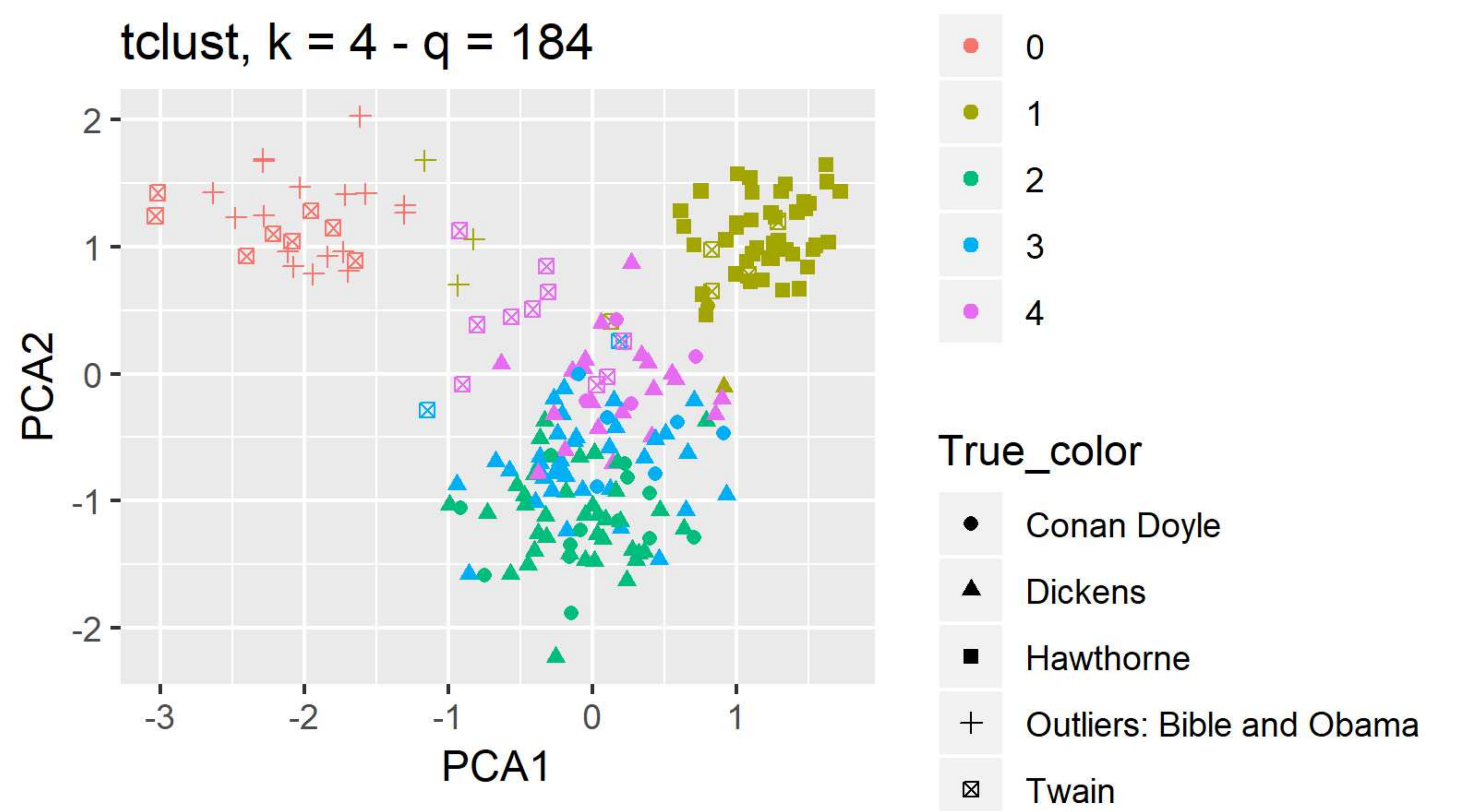}\\
	\end{minipage}
\caption{Examples of Author clusterings obtained with tclust algorithm\label{Author_tclust_clusterings}}
\end{figure}

\noindent\textbf{Trimmed $k$-means}:

Figure \ref{Author_Gauss_cost_NMI} suggests the choice $k=4$, and Figure \ref{Author_Gauss_cost_NMI_k4} shows that $q=190$ yields a slope jump and NMI peak. The associated clustering is depicted in Figure \ref{Author_Gauss_clusterings}, its NMI is  $ 0.5336308$.
\begin{figure}[H]
	\begin{minipage}[h]{.49\linewidth}
		\centering\includegraphics[scale=0.35]{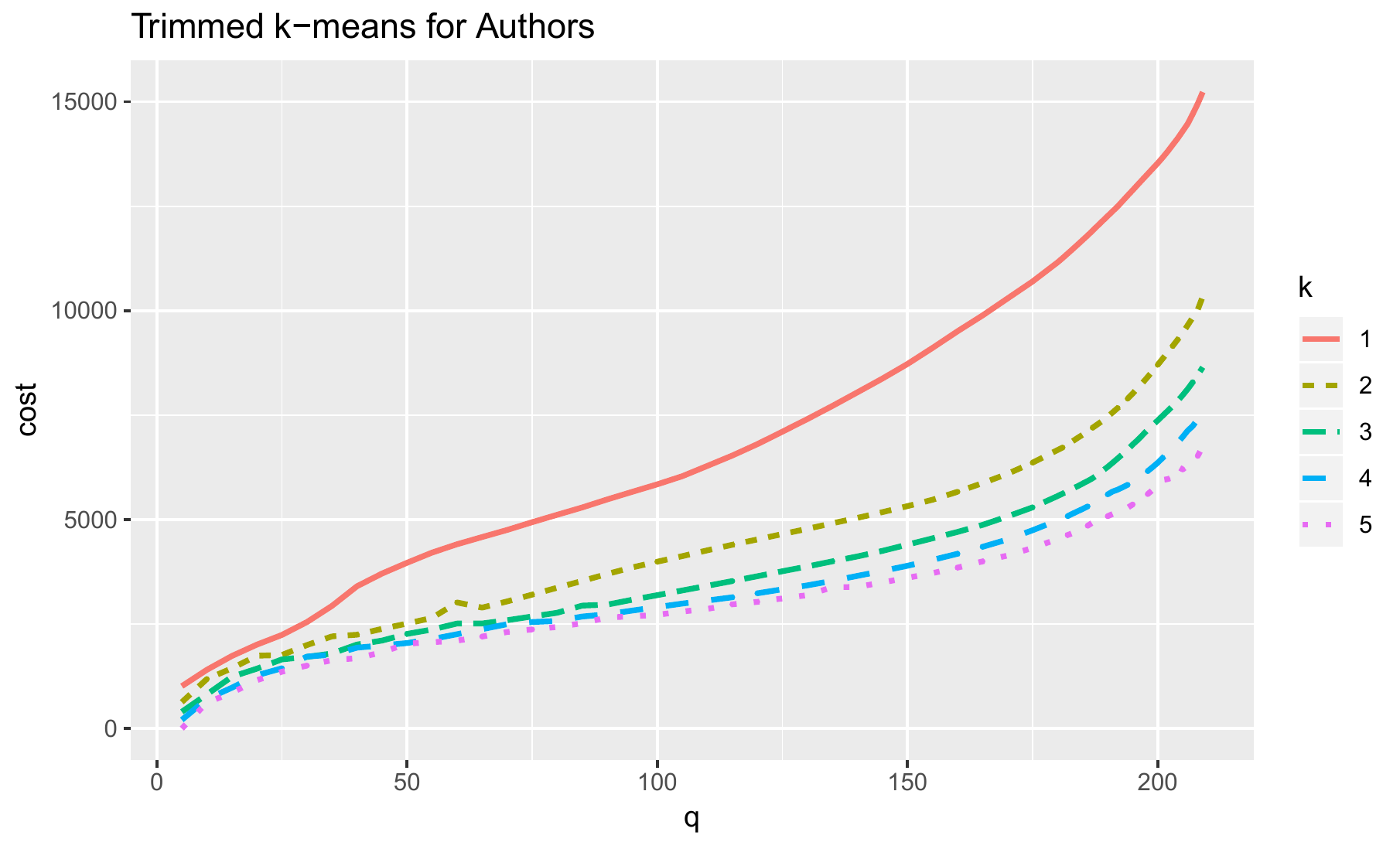}\\
	\end{minipage}
	\begin{minipage}[h]{.49\linewidth}
\centering\includegraphics[scale=0.35]{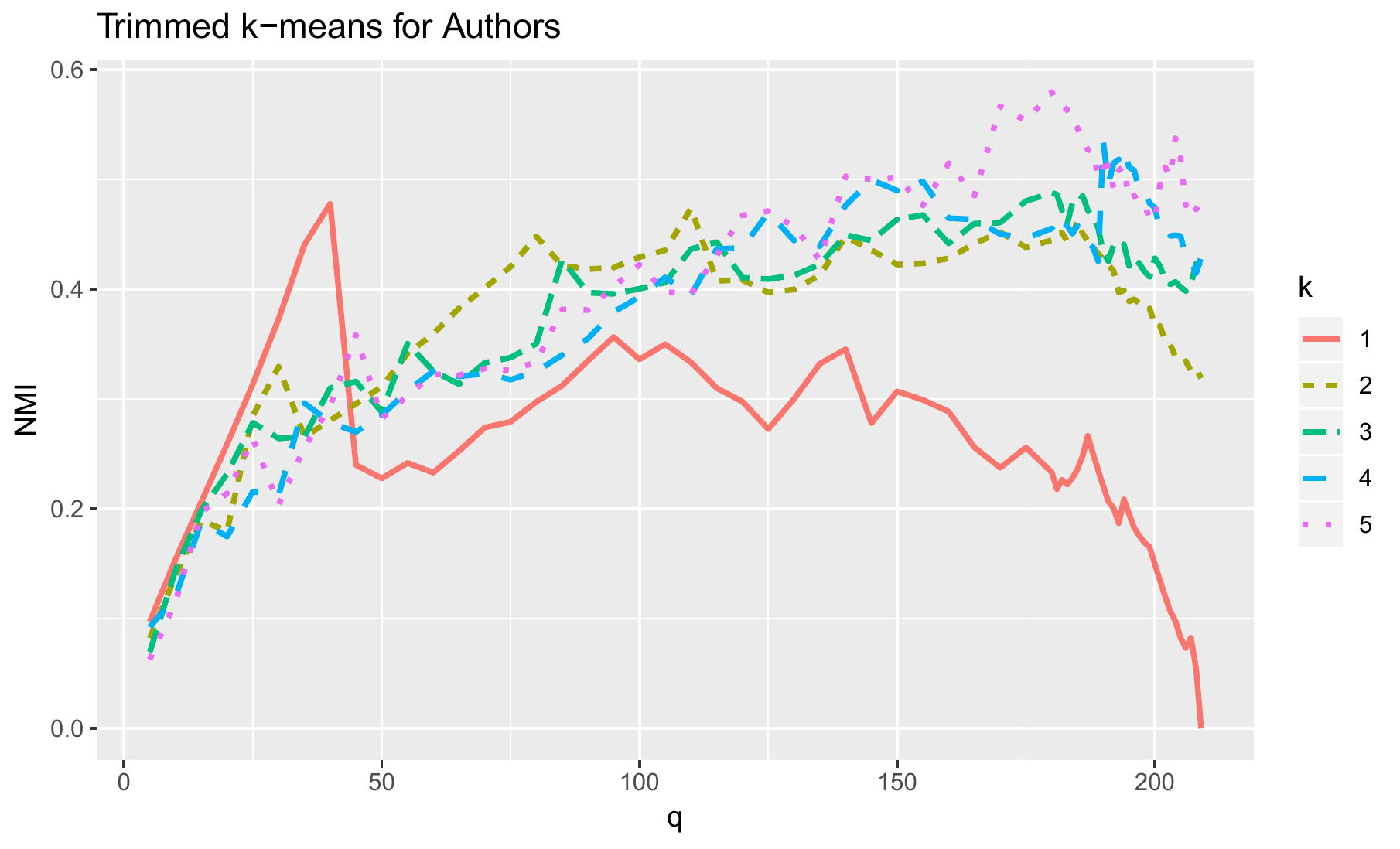}
	\end{minipage}
\caption{Cost and NMI for Author clustering with trimmed $k$-means\label{Author_Gauss_cost_NMI}}
\end{figure}

\begin{figure}[H]
	\begin{minipage}[h]{.49\linewidth}
		\centering\includegraphics[scale=0.35]{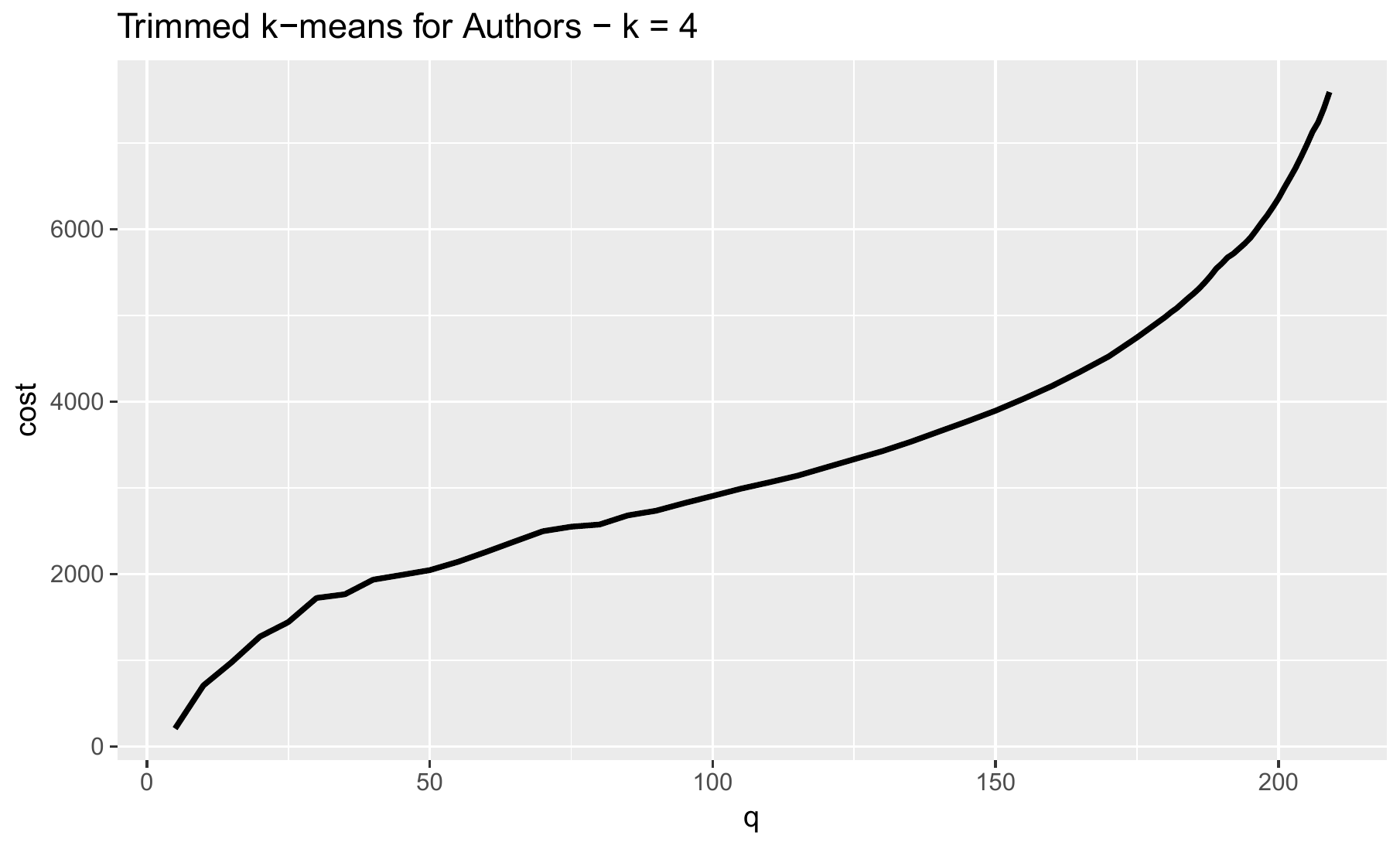}\\
	\end{minipage} 
	\begin{minipage}[h]{.49\linewidth}
\centering\includegraphics[scale=0.35]{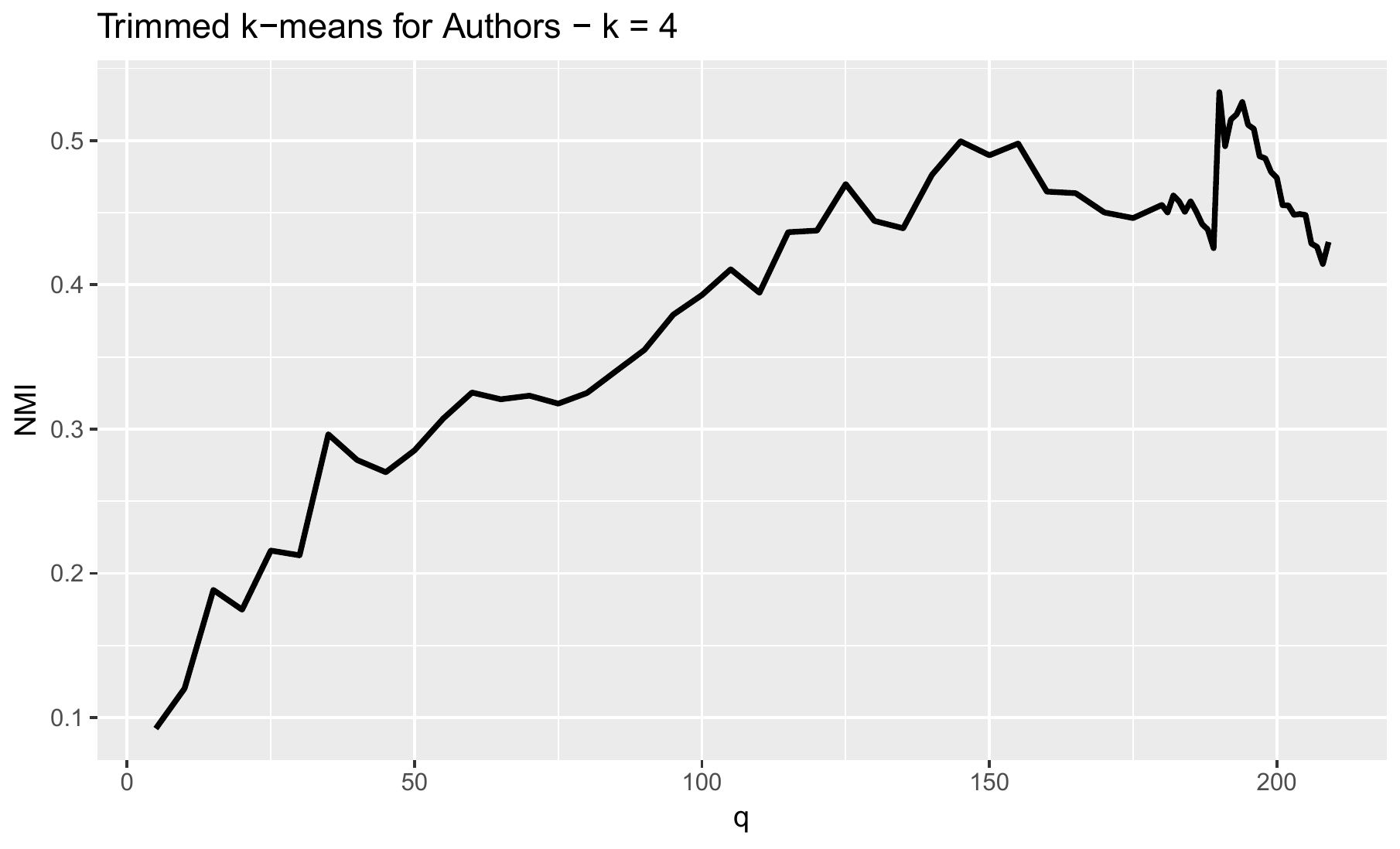}
	\end{minipage}
\caption{Cost and NMI for Author clustering with trimmed $k$-means -- $k=4$\label{Author_Gauss_cost_NMI_k4}}
\end{figure}

\begin{figure}[H]
	\centering\begin{minipage}[h]{.49\linewidth}
		\includegraphics[scale=0.14]{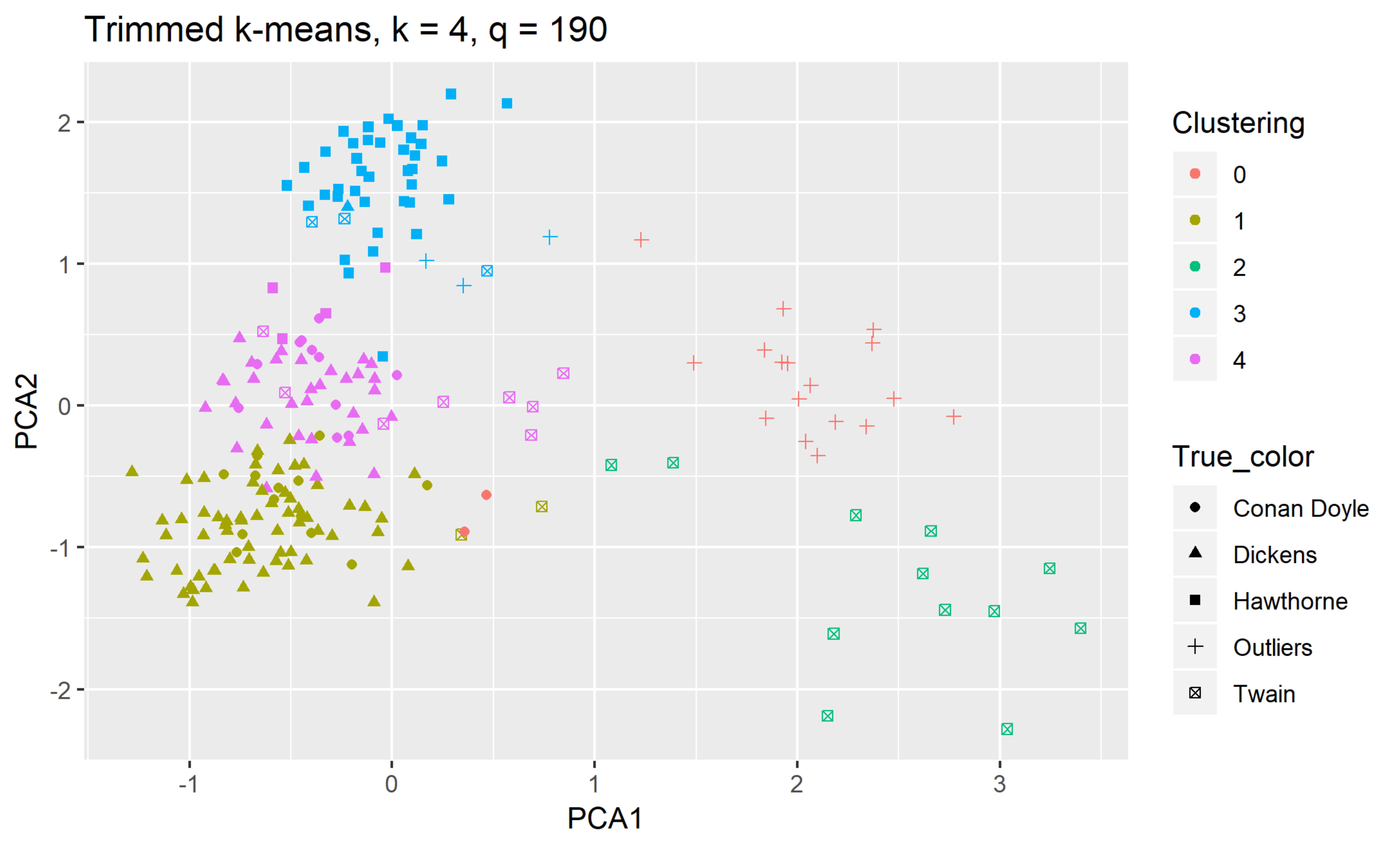}\\
	\end{minipage}
\caption{Examples of Author clusterings obtained with trimmed $k$-means\label{Author_Gauss_clusterings}}
\end{figure}
\end{document}